\documentclass[]{interact}

\usepackage{epstopdf}
\usepackage[caption=false]{subfig}

\usepackage[numbers,sort&compress]{natbib}
\bibpunct[, ]{[}{]}{,}{n}{,}{,}
\renewcommand\bibfont{\fontsize{10}{12}\selectfont}

\theoremstyle{plain}
\newtheorem{theorem}{Theorem}[section]
\newtheorem{lemma}[theorem]{Lemma}

\newtheorem{proposition}[theorem]{Proposition}

\theoremstyle{definition}

\theoremstyle{remark}
\newtheorem{remark}{Remark}[section]

\usepackage{hyperref}
\usepackage{graphicx}
\usepackage{dsfont}
\usepackage{mathtools}
\usepackage{wrapfig}
\usepackage{amssymb, amsmath, latexsym}
\usepackage{nicefrac}
\usepackage{hhline}
\usepackage{multirow}
\usepackage{comment}
\usepackage{pifont}

\usepackage[colorinlistoftodos,bordercolor=orange,backgroundcolor=orange!20,linecolor=orange,textsize=scriptsize]{todonotes}
\usepackage{algorithm}\usepackage{algpseudocode}
\usepackage[symbol]{footmisc}

\usepackage{tikz}
\usetikzlibrary{matrix, positioning, fit}

\newcommand{\argmin}{\mathop{\arg\!\min}}

\def \R {\mathbb R}

\def\wtgg{\widetilde{\nabla}}
\def\wha{\widehat{\alpha}}
\def\whL{\widehat{L}}
\def\BJ{\boldsymbol{J}}
\def\fp{\mathbf{FP}}

\def\mach{\bm{\epsilon}_{{\textbf{\tiny mach}}}}

\begin{document}


\title{Lower and upper bounds of the convergence rate of gradient methods with composite noise in gradient\thanks{}}

\author{
\name{
Artem Vasin \textsuperscript{a} and
Alexander Gasnikov \textsuperscript{b, a, c, d}
}
\affil{
\textsuperscript{a}Moscow Independent Research Institute of Artificial Intelligence, Moscow, Russia.
\textsuperscript{b}Innopolis University, Innopolis, Russia.
\textsuperscript{c}Steklov Mathematical Institute of the Russian Academy of Sciences, Moscow, Russia.
\textsuperscript{d}Institute for Information Transmission Problems, Moscow, Russia.
}
}

\date{Received: date / Accepted: date}
\maketitle              
\begin{abstract}
We introduce a detailed analysis of the convergence of first-order methods with composite noise (sum of relative and absolute) in gradient for convex and smooth function minimization. This paper illustrates instances of practical problems where the utilization of inexact oracles becomes necessary, such as biased compressors, use of floating-point arithmetic and gradient-free optimization. We propose an algorithm that optimally accumulates absolute error, with intermediate convergence depending on the relative component of the noise. Usage of restart technique, regularization transformation, and stopping criteria has been demonstrated to yield results for various function classes. Also, gradient descent adaptive to relative error parameter is provided. For relative noise, lower bounds of convergence are given, confirming the dependence of the parameter of the noise on the condition number of the problem.
\end{abstract}


\section{Introduction}\label{Introduction}
We consider a global convex optimization problem:
\begin{equation} \label{optim}
    \min\limits_{x \in \mathbb{R}^n} f(x),
\end{equation}
the definition of convexity of a function can be formulated as follows:
\begin{equation} \label{convexity}
    (\forall x, y \in \mathbb{R}^n) \; f(x) + \langle \nabla{f}(x), y - x \rangle \leqslant f(y).
\end{equation}
We assume that the objective $f$ is $L$-smooth:
\begin{equation} \label{smooth cond}
    (\forall x, y \in \mathbb{R}^n) \; f(y) \leqslant f(x) + \langle \nabla{f}(x), y - x \rangle + \frac{L}{2} \|x - y \|_2^2.
\end{equation}
To study the convergence, we propose considering convexity~\eqref{convexity} enhancement -- $\mu$ strongly convexity:
\begin{equation} \label{strong conv}
    (\forall x, y \in \mathbb{R}^n) \; f(x) + \langle \nabla{f}(x), y - x \rangle + \frac{\mu}{2} \|x -y \|_2^2 \leqslant f(y).
\end{equation}
In some cases we will consider functions $f$ satisfying Polyak-{\L}ojasiewicz (P{\L}) condition with $\mu$ parameter:
\begin{equation} \label{PL cond}
    (\forall x \in \mathbb{R}^n) \; \frac{1}{2\mu} \|\nabla f(x) \|_2^2 \geqslant f(x) - f^*.
\end{equation}
We define $f^*$ - as minimum value of $f$ (solution for problem~\ref{optim}) and $x^*$: $f(x^*) = f^*$. In this paper we consider iterative first order (gradient) methods discussed in detail~\cite{gasnikov2017universal, bubeck2015convex, beck2017first, nemirovski1995information, nesterov2018lectures, polyak1987introduction}. The class of such algorithms can be defined as:
\begin{equation} \label{linear first order method}
    \begin{gathered}
        x^N = x^0 + \text{span}\left \lbrace \wtgg f(x^0), \wtgg f(x^1), \dots , \wtgg f(x^{N - 1}) \right \rbrace, \\
        \wtgg f(x^k) \text{ -- some estimation of gradient } \nabla f(x^k).
    \end{gathered}
\end{equation}
Thus we introduce starting point $x^0$ and $R =  \|x^0 - x^* \|_2$. Solving problem~\eqref{optim}, we aim to achieve a certain precision $\varepsilon$:
\begin{equation} \label{def:func_eps_solution}
    f(\widehat{x}) - f^* \leqslant \varepsilon,
\end{equation}
where $\widehat{x}$ -- point, which some algorithm will produce.

We consider, that we we don't have access to the real gradient $\nabla f(x)$, only to the noisy one $\widetilde{\nabla} f(x)$ satisfying composite noise condition:
\begin{equation} \label{noise condition}
    \left(\forall x \in \R^n \right) \; \|\widetilde{\nabla} f(x) - \nabla f(x) \|_2 \leqslant \alpha \|\nabla f(x) \|_2 + \delta.
\end{equation}
Different models of the inexact oracle $\wtgg f$ were also considered at different literature~\cite{devolder2013first, devolder2014first, polyak1987introduction, vasin2023accelerated, gasnikov2017universal, ajalloeian2020convergence}. Let us designate $\zeta(x) = \widetilde{\nabla} f(x) - \nabla f(x)$. By Lemma~\ref{noise decompose} an equivalent condition to~\ref{noise condition} can be proposed, decomposing noise $\zeta = \zeta_a + \zeta_r$ to absolute and relative component:
\begin{equation} \label{noise def abs rel}
    \zeta(x) = \wtgg f(x) - \nabla f(x) = \zeta_a(x) + \zeta_r(x), \quad  \| \zeta_a(x) \| \leqslant \delta, \hspace{5mm} \| \zeta_r(x) \|_2 \leqslant \alpha \| \nabla f(x) \|_2.
\end{equation}
Special cases are naturally occurring relative ($\delta = 0$):
\begin{equation} \label{relative noise condition}
    \left(\forall x \in \R^n \right) \; \|\widetilde{\nabla} f(x) - \nabla f(x) \|_2 \leqslant \alpha \|\nabla f(x) \|_2,
\end{equation}
and absolute noise ($\alpha = 0$):
\begin{equation} \label{absolute noise condition}
    \left(\forall x \in \R^n \right) \; \|\widetilde{\nabla} f(x) - \nabla f(x) \|_2 \leqslant \delta.
\end{equation}
In Section~\ref{motivation section} one can find examples of inexact oracles. Using model of absolute noise~\eqref{absolute noise condition} one can note, that using any first order method in case $\| \nabla f(x) \|_2 \leqslant \delta$ does not make sense. Then, using $P\L$ condition~\eqref{PL cond}: $f(x) - f^* \leqslant \nicefrac{\delta^2}{2 \mu}$. In this paper we estimate number of iterations in the presence of composite noise~\eqref{noise condition} to reach level $\left( \frac{L}{\mu} \right)^{p} \nicefrac{\delta^2}{\mu}$. At~\cite{devolder2013first, kamzolov2021universal} intermediate methods (exploring a similar question) were discussed, which guarantee optimal (in terms of Theorem~\ref{lower bound oracle absolute}) convergence rate. More about different noised gradient oracles and previous papers on relative and absolute error we provided in Section~\ref{section related work}. Our contribution described in Section~\ref{contribution section}.

\section{Motivation} \label{motivation section}

In this section we will provide some motivation (sufficient conditions) for composite noise condition~\eqref{noise condition} consideration. Specific imprecise oracles are formulated in Propositions~\ref{propostion:fp_abs},~\ref{propostion:compressor_rel},~\ref{propostion:zo_abs},~\ref{propostion:rel_interpret}.

\subsection{IEEE754} \label{motivation ieee754}

Modern computational systems use floating point representation (rounding) described at~\cite{kahan1996ieee}:
\begin{equation} \label{def:floating}
    \begin{gathered}    
        \hat{x} = (-1)^s B \cdot 2^{E - E_0}, \quad
        B = 1.b_1 \dots b_{p}, \\
        b_j \in \lbrace 0, 1 \rbrace, \; s \in \lbrace 0, 1 \rbrace, \\
        \mach = 2^{-p} \text{ - relative precision of IEEE754 (machine epsilon)}.
    \end{gathered}
\end{equation}
Such representation provides the following approximation:
\begin{equation*}
    | \widehat{x} - x | \leqslant \mach |x| \text{, where } \widehat{x} \text{ floating point representation of } x.
\end{equation*}
The most popular formats are single-precision (FP32, $p = 23$) and double-precision (FP64, $p = 52$). The basic properties of floating-point arithmetic are described at~\cite{overton2001numerical, gautschi2011numerical}, some authors investigated the numerical stability of various methods~\cite{boldo2017round} (stability of Runge-Kutta method).

Let $\fp(x \textbf{ op } y)$ an operation performed in floating-point arithmetic of $x \textbf{ op } y$, where $\textbf{ op } \in \lbrace +, -, \cdot, / \rbrace$. IEEE754 standard guarantees the following arithmetic properties for two float point numbers~\eqref{def:floating} $x, y$:
\begin{equation} \label{eq:fp_props}
    \begin{gathered}
        \fp(x + y) = (x + y) \cdot (1 + \varepsilon_0), \quad \fp(x - y) = (x - y) \cdot (1 + \varepsilon_0), \\    
        \fp(x \cdot y) = (x \cdot y) \cdot (1 + \varepsilon_0), \quad \fp(x / y) = (x / y) \cdot (1 + \varepsilon_0), \\
        \text{where } |\varepsilon_0| \leqslant \mach.
    \end{gathered}
\end{equation}
The main issue, using this calculation model will be the stability (error propagation) of the particular method. One can compare two ways of calculation for two floating point numbers $x, y$:
\begin{equation*}
    \fp\left(\fp(x - y) \cdot \fp(x + y) \right) \quad \text{vs } \quad \fp \left( \fp(x^2) - \fp(y^2) \right),
\end{equation*}
at~\cite[p. 56]{boldo2023floating} mentioned, that first way is more numerical stable, then second.

We can estimate error accumulation for recursive product algorithm of $n$ floating point numbers $\lbrace x_k \rbrace_{k = 1}^n$:
\begin{equation*}
    p_n = \prod_{k = 1}^n x_k, \quad \widehat{p}_{k + 1} = \fp(\widehat{p}_k \cdot x_{k + 1}), \quad \widehat{p}_1 = x_1.
\end{equation*}
Then using~\eqref{eq:fp_props} the following estimation can be provided:
\begin{equation*}
    \begin{gathered}
        \widehat{s}_n = s_n \cdot \prod_{k = 2}^n (1 + \varepsilon_k), \; |\varepsilon_k| \leqslant \mach, \text{then} \\
        | \widehat{s}_n - s_n | \leqslant \max \left\lbrace (1 + \mach)^{n - 1} - 1, 1 - (1 - \mach)^{n - 1} \right\rbrace |s_n|, \\
        | \widehat{s}_n - s_n | \leqslant \max \left\lbrace e^{\mach (n - 1)} - 1, 1 - e^{-\mach (n - 1)} \right\rbrace |s_n| = \left(e^{\mach (n - 1)} - 1 \right) |s_n|.
    \end{gathered}
\end{equation*}
Thus, one can obtain product of $n$ numbers, using recursive procedure above, herewith error of $\widehat{s}_n$ calculation will be relative with parameter $e^{\mach (n - 1)} - 1$.

The case of a sum of several terms demonstrates the problem of error accumulation. At~\cite{blanchard2020class} were mentioned plenty of summation methods -- recursive, blocked, pairwise (binary cascade) and compensated. The last one is also known as Kahan Summation~\cite[p. 203]{goldberg1991every}. That algorithm estimating a float point approximation error on each step.
\begin{equation*}
    \begin{gathered}
        s_n = \sum_{k = 1}^n x_k, \quad \widehat{s}_n - \text{ Kahan Summation output}, \\
        | \widehat{s}_n - s_n | = O \left( \mach + n \mach^2 \right) \cdot \sum_{k = 1}^n | x_k |.
    \end{gathered}
\end{equation*}
So, in this case error of method will be absolute and proportional to $\sum_{k = 1}^n | x_k |$. It is worth noting, that if all $x_k$ is positive, then error will be relative (such a case can be considered when calculating the $2$-norm in $\mathbb{R}^n$). For applied problems, calculating the scalar product $\sum\limits_{k = 1}^n x_k \cdot y_k$ can be useful. We can reduce it to summation problem of $\fp(x_k \cdot y_k)$, then we can guarantee, using Kahan Summation algorithm provides scalar product estimation $\widehat{s}$, such that:
\begin{equation*}
    \left| \widehat{s} - \sum_{k = 1}^n x_k  y_k \right| = O \left( \mach + n \mach^2 \right) \cdot \sum_{k = 1}^n | x_k  y_k |.
\end{equation*}
Thus we can obtain conclusions about the gradient estimation for a particular function, for example quadratic function $f(x) = \frac{1}{2} x^T A x + b^T x \Rightarrow \nabla f(x) = A x + b$ ($A$ is symmetric positive definite). If $A = (a_{k,j})_{1 \leqslant k, j \leqslant n}, g = \nabla f(x) = A x + b, g_k = \sum\limits_{j = 1}^n a_{k, j} x_j$, then Kahan Summation provides $\widetilde{g} = \widetilde{\nabla} f(x)$ such that:
\begin{equation*}
    | \widetilde{g}_k - g_k | = O \left( \mach + n \mach^2 \right) \cdot \left( | b_k | + \sum\limits_{j = 1}^n | a_{k, j} x_j | \right).
\end{equation*}
And estimation for whole norm error:
\begin{eqnarray*}
    \| \widetilde{g} - \left( Ax + b \right) \|_2 & \leqslant & \|  \widetilde{g} - \left( Ax + b \right) \|_1 = O \left( \mach + n \mach^2 \right) \cdot \sum_{k = 1}^n \left( | b_k | + \sum\limits_{j = 1}^n | a_{k, j} x_j | \right)
    \\
    & = & O \left( \mach + n \mach^2 \right) \cdot \left( \| b \|_1 + \| A \|_1 \| x \|_1 \right),
\end{eqnarray*}
here $\| A \|_1$ is corresponding operator norm:
\begin{equation*}
    \| A \|_1 = \sup_{\|x \| = 1} \frac{\| A x \|_1}{\| x \|_1} = \max_{1 \leqslant j \leqslant n} \sum_{k = 1}^n | a_{k, j} |.
\end{equation*}
Thus, we can formulate the following statement.
\begin{proposition} \label{propostion:fp_abs}
    Let $A$ is symmetric positive definite matrix, $b, x$ - vectors of dimension $n$, whose elements are floating point numbers~\eqref{def:floating}. Let $\widetilde{\nabla} f(x)$ - gradient estimation of function $f(x) = \frac{1}{2} x^T A x + b^T x, \; \nabla f(x) = Ax + b$, calculated using floating point operations and Kahan Summation Algorithm for dot product (including matrix multiplication), then $\wtgg f(x)$ satisfies absolute noise condition~\eqref{absolute noise condition}:
    \begin{equation*}
        \| \wtgg f(x) - \nabla f(x) \|_2 = O \left( \mach + n \mach^2 \right) \cdot \left( \| b \|_1 + \| A \|_1 \| x \|_1 \right).
    \end{equation*}
    If all elements of $A, b, x$ is nonnegative, then $\wtgg f(x)$ satisfies relative noise condition~\eqref{relative noise condition}:
    \begin{equation*}
        \| \wtgg f(x) - \nabla f(x) \|_2 = O \left( \sqrt{n} \mach + n\sqrt{n} \mach^2 \right) \cdot \| \nabla f(x) \|_2.
    \end{equation*}
    $\sqrt{n}$ appears, since $\forall z \; \| z \|_1 \leqslant \sqrt{n} \| z \|_2$.
\end{proposition}

\subsection{Biased compressors for distributed optimization}

Nowadays distributed (federated, decentralized) optimization continues to gain popularity. Various problems such as saddle point problems, variational inequalities and minimization were considered in these papers~\cite{beznosikov2025distributed, gorbunov2023high, rogozin2025decentralized, kovalev2022optimal, rogozin2021towards}. In decentralized setting communication between workers is crucial problem, inter alia reducing the load on the data transmission channel. Bounded compressors $\mathcal{Q}$ considered at~\cite{beznosikov2023biased, condat2022ef, gorbunov2020linearly}:
\begin{equation*}
    \| \mathcal{Q}(x) \|_2^2 \leqslant \rho \| x \|_2^2.
\end{equation*}
Worth noting, that we define deterministic compressors, although papers above study stochastic ones. The following proposition describes different compressors.

\begin{proposition} \label{propostion:compressor_rel}

    Below $u^k$ - orthogonal basis in $\R^n$.

    1. Consider the Top-K compressor (see~\cite{beznosikov2023biased}):
    \begin{equation*}
        \mathcal{Q}_{\text{Top-K}}(x) = \sum_{j = n - k + 1}^n x_{(j)} u^{(j)}, \quad x_{(1)} \leqslant \dots \leqslant x_{(n)}.
    \end{equation*}
    Then gradient estimation $\wtgg f(x) = \mathcal{Q}_{\text{Top-K}} \left( \nabla f(x) \right)$ satisfies relative error~\eqref{relative noise condition}:
    \begin{equation*}
        \| \wtgg f(x) - \nabla f(x) \|_2 \leqslant \sqrt{1 - \nicefrac{k}{n}} \cdot \| \nabla f(x) \|_2.
    \end{equation*}

    2. Consider the sparsification compressor (see~\cite{beznosikov2023biased}):
    \begin{equation*}
        \mathcal{Q}_{\text{spars}}(x) = \sum_{j = 1}^n \frac{s_j}{m} u^j, \quad s_j = \argmin_{s \in \mathbb{Z}} \left | \frac{s}{m} - x_j \right |, \quad m - \text{fixed natural number}.
    \end{equation*}
    Then gradient estimation $\wtgg f(x) = \mathcal{Q}_{\text{spars}} \left( \nabla f(x) \right)$ satisfies absolute error~\eqref{relative noise condition}:
    \begin{equation*}
        \| \wtgg f(x) - \nabla f(x) \|_2 \leqslant \frac{\sqrt{n}}{2m}.
    \end{equation*}

    3. Consider the sign compressor (see sign-sgd~\cite{bernstein2018signsgd}):
    \begin{equation*}
        \mathcal{Q}_{\text{sign}}(x) = \left( \frac{1}{n} \sum_{j = 1}^n | x_j | \right) \cdot \sum_{j = 1}^n \text{sign}(x_k) u^j.
    \end{equation*}
    Then gradient estimation $\wtgg f(x) = \mathcal{Q}_{\text{sign}} \left( \nabla f(x) \right)$ satisfies relative error~\eqref{relative noise condition}:
    \begin{equation*}
        \| \wtgg f(x) - \nabla f(x) \|_2 \leqslant \sqrt{1 - \nicefrac{1}{n}} \cdot \| \nabla f(x) \|_2.
    \end{equation*}

\end{proposition}

\subsection{Gradient-free methods}

Assume we don't have access to any gradient oracle (including noised), however, there is access to first-order information:
\begin{equation} \label{def:func_inexact_oracle}
    \left| \widetilde{f}(x) - f(x) \right| \leqslant \delta_f.
\end{equation}
In this case, we can also solve the minimization problem~\eqref{optim}, various approaches have been proposed at~\cite{kornilov2023accelerated, lobanov2024black, gasnikov2023randomized, lobanov2024acceleration}. To consider the deterministic noise model (biased stochastic noise more common in that setting) we refer to~\cite{berahas2022theoretical}. Using this paper we introduce the following proposition (Theorem 2.1 at~\cite[p. 5]{berahas2022theoretical}).
\begin{proposition} \label{propostion:zo_abs}
    Let $f$ is smooth~\eqref{smooth cond} and $\widetilde{f}$ satisfies~\eqref{def:func_inexact_oracle}. Consider the following gradient estimation:
    \begin{equation*}
        \wtgg f(x) = \sum_{k = 1}^n \frac{f(x + h u^k) - f(x)}{h} u^k, \; u^k \text{ - orthogonal basis in } \mathbb{R}^n.
    \end{equation*}
    Then $\wtgg f$ satisfies absolute noise condition~\eqref{absolute noise condition}:
    \begin{equation*}
        \| \wtgg f(x) - \nabla f(x) \|_2 \leqslant \sqrt{n} \left( \frac{L h}{2} + \frac{2 \delta_f}{h} \right).
    \end{equation*}
\end{proposition}

\subsection{Relative interpretation of absolute noise}

\begin{proposition} \label{propostion:rel_interpret}
    Assume $\wtgg f$ satisfies absolute noise~\eqref{absolute noise condition} and $\| \nabla f(x) \| \geqslant K \delta, \; K > 1$. Then $\wtgg f$ also satisfies relative error~\eqref{relative noise condition}:
    \begin{equation*}
        \| \wtgg f(x) - \nabla f(x) \|_2 \leqslant \delta \leqslant \frac{1}{K} \| \nabla f(x) \|_2.
    \end{equation*}
\end{proposition}

The proposition above seems naive, however such trick we used to provide lower bounds (see Section~\ref{lower bound section}) and series of positive results on convergence for absolute noise (see Section~\ref{section main relative interpretation}).

\section{Related Work} \label{section related work}

\subsection{Inexact $(\delta, L, \mu)$ oracle}

In~\cite{devolder2013first} one of inexact oracle model was introduced. Pair $(g_{\delta, L, \mu}(y), f_{\delta, L, \mu}(y))$ called $(\delta, \mu, L)$ oracle at point $y$ for function $f$, if $\forall x \in \mathbb{R}^n$:
\begin{equation} \label{delta L mu oracle}
    \frac{\mu}{2} \|x - y \|_2^2 \leqslant f(x) - \left(f_{\delta, L, \mu}(y) + \langle g_{\delta, L, \mu}(y), x - y \rangle \right) \leqslant \frac{L}{2} \|x - y \|_2^2 + \delta.   
\end{equation}
It follows from the Lemma~\ref{reduction to oracle} that absolute inexact noise in gradient~\ref{absolute noise condition} is a special case of $(\delta, L, \mu)$ oracle~\ref{delta L mu oracle}. Such an inexact oracle model can also be used for functions which do not satisfy smoothness~\eqref{smooth cond}, for example the Hölder continuity:
\begin{equation} \label{def:holder continity}
    \forall x, y \; \| \nabla f(x) - \nabla f(y) \|_2 \leqslant L_{\nu} \| x - y \|_2^{\nu}.
\end{equation}
If functions satisfies~\eqref{def:holder continity}, then $(\nabla f(x), f(x))$ satisfies $\left( \delta, L(\delta, \nu) \right)$ (see~\cite{devolder2014first}), where
\begin{equation*}
    L = L_{\nu} \left( \frac{L_{\nu}}{2 \delta} \frac{1 - \nu}{1 + \nu} \right)^{\frac{1 - \nu}{1 + \nu}}.
\end{equation*}

Authors of~\cite{alkousa2024higher} proposed a generalization of the model above, namely $(\delta, L, \mu, q)$ oracle:
\begin{equation} \label{delta L mu q oracle}
    \frac{\mu}{2} \|x - y \|_2^2 \leqslant f(x) - \left(f_{\delta, L, \mu, q}(y) + \langle g_{\delta, L, \mu, q}(y), x - y \rangle \right) \leqslant \frac{L}{2} \|x - y \|_2^2 + \delta \| x - y \|_2^q,
\end{equation}
and conclude that relative error in gradient~\ref{relative noise condition} is a special case of $(\delta, L, \mu, q)$ oracle~\ref{delta L mu q oracle} with $q = 1, \delta = \frac{\alpha}{1 - \alpha} \wtgg f(y)$. It is worth noting, that relative inexactness has not been adapted yet to $(\delta, L, \mu)$ oracle so that appropriate methods ensure convergence.

$(\delta, L, \mu)$ oracle was mentioned, because we used lower bound convergence of first order method to obtain lower bound convergence (see Theorem~\ref{lower bound relative noise mu > 0 text}) for method using gradient estimation $\wtgg f$ satisfies relative error~\eqref{relative noise condition}.

\subsection{Relative error}

Although relative inaccuracy seems less harmful to convergence, nowadays there is no comprehensive answer on lower and upper bound for convergence of first order methods~\eqref{linear first order method} in presence of relative error~\eqref{relative noise condition} in gradient. We will mention here the key current results on this problem.
\begin{itemize}
    \item In~\cite{polyak1987introduction} was proved convergence for smooth~\eqref{smooth cond} function satisfying $P\L$~\eqref{PL cond}:
    \begin{equation*}
        f(x^N) - f^* \leqslant \left( 1 - \left(\frac{1 - \alpha}{1 + \alpha} \right)^2 \frac{\mu}{L} \right)^N \left( f(x^0) - f^* \right).
    \end{equation*}
    \item In~\cite{vernimmen2025worst} worst case convergence analysis of Gradient Descent (Algorithm~\ref{alg gd}) convergence (via PEP framework~\cite{goujaud2024pepit}).
    \item In~\cite{alkousalipschitz} lipschitz-free optimization for strongly convex functions was considered with absolute and relative noise.
    \item In~\cite{puchinin2023gradient} adaptive gradient descent for $\alpha$ and $L$ was proposed. We used same idea for Algorithm~\ref{alg rel adaptive gd}.
    \item In~\cite{hallak2024study} the problem of convergence for convex functions~\eqref{convexity} was studied with additional assumptions:
    \begin{equation*}
        f(x^N) - \frac{1 - \alpha}{1 + \alpha} f^* \leqslant \frac{8 (1 + \alpha) LR^2}{(1 - \alpha) (N + 1)}.
    \end{equation*}
    \item In~\cite[Proposition 1.8]{gannot2022frequency} was introduced accelerated method with convergence (for strongly convex~\eqref{strong conv} and smooth~\eqref{smooth cond} function) for $\alpha = O \left( \sqrt{\frac{\mu}{L}} \right)$:
    \begin{equation*}
        \| x^N - x^* \|_2 \leqslant \left( 1 - O(1) \sqrt{\frac{\mu}{L}} \right)^N \| x^0 - x^* \|_2.
    \end{equation*}
    \item In~\cite[Theorem 5.1]{vasin2023accelerated} proposed convergence (via reduction from relative noise to absolute) for $\alpha = O \left( \frac{\mu}{L} \right)$:
    \begin{equation*}
        f(x^N) - f(x^*) \leqslant \left(\frac{5LR^2}{4} + \frac{15}{196} \sqrt{\frac{2L}{\mu}}\left[f(x^0) - f(x^*) \right] \right)\exp{\left({\displaystyle-\frac{N}{4}\sqrt{\frac{\mu}{2L}}} \right)}.
    \end{equation*}
    \item In~\cite[Theorem 3.2]{kornilov2025intermediate} proposed convergence (via restart technique) for $\alpha = O \left( \sqrt{\frac{\mu}{L}} \right)$:
    \begin{equation*}
        f(x^N) - f(x^*) \leqslant \varepsilon, \quad N = O \left( \sqrt{\frac{L}{\mu}} \log_2 \left( \frac{\mu R^2}{\varepsilon} \right) \right).
    \end{equation*}
    \item In~\cite[Theorem 5]{vasin2025solving} proved intermediate convergence (our previous paper):
    \begin{equation*}
        f(x^N) - f^* \leqslant LR^2 \left( 1 - \frac{1}{10 \sqrt{2}} \left(\frac{\mu}{L} \right)^{\frac{1}{2} + \tau} \right), \quad \alpha = \frac{1}{3} \left( \frac{\mu}{L} \right)^{\frac{1}{2} - \tau}.
    \end{equation*}
    Also lower bound, which we also proved in this paper (Theorem~\ref{lower bound relative noise mu > 0 text}) was introduced.
\end{itemize}

\section{Contribution} \label{contribution section}

Let us present the main theoretical results of this paper.
\begin{itemize}
    \item Analysis for Gradient Descent (Algorithm~\ref{alg gd}) in presence of composite noise~\eqref{noise condition} -- Theorem~\ref{gd pl conv text}:
    \begin{equation*}
        f(x^N) - f^* \leqslant \left(1 -  O(1) \frac{(1 - \alpha)^3}{1 + \alpha} \frac{\mu}{L} \right)^N (f(x^0) - f^*) + O(1) \frac{1 + \alpha}{(1 - \alpha)^3} \frac{\delta^2}{\mu}.
    \end{equation*}
    Using regularization technique obtained result for convex function and relative noise~\eqref{relative noise condition} ($\alpha < \nicefrac{1}{2}$) -- Theorem~\ref{gd reg text}:
    \begin{equation*}
        f(x^N) - f^* \leqslant \varepsilon, \quad N = O \left( \frac{LR^2}{\varepsilon} \ln \left( \frac{LR^2}{\varepsilon} \right) \right).
    \end{equation*}
    For $\alpha$-adaptive version (Algorithm~\ref{alg rel adaptive gd}) provides convergence -- Theorem~\ref{adaptive PL convergence only alpha text}:
    \begin{equation*}
        f(x^N) - f^* \leqslant \left( 1 - O(1) (1 - \alpha)^3 \frac{\mu}{L} \right)^N (f(x^0) - f^*) + O\left(\frac{1}{(1 - \alpha)^{3}} \frac{\delta^2}{\mu} \right).
    \end{equation*}

    \item Analysis for accelerated method -- RE-AGM (Algorithm~\ref{alg re-agm}) in presence of composite noise~\eqref{noise condition} -- Theorem~\ref{acc re agm conv text}.
    \begin{equation*}
        \begin{gathered}
            f(x^N) - f^* \leqslant \left(1 - O(1) \left(\frac{\mu}{L}\right)^{1 - \gamma^*} \right)^N \left(f(x^0) - f^* + \frac{\mu}{4} R^2 \right)
            + O \left(\left(\frac{L}{\mu} \right)^{\gamma^*} \frac{\delta^2}{\mu} \right), \\
            \gamma^* = \min \left \lbrace \log_{\mu / 2L} (3 \alpha ), \frac{1}{2} \right \rbrace.
        \end{gathered}
    \end{equation*}
    Using regularization technique obtained result for convex function and relative noise~\eqref{relative noise condition} ($\alpha \leqslant \frac{1}{3} \left(\frac{\varepsilon}{12 LR^2} \right)^{\beta}, 0 \leqslant \beta \leqslant 1/2$) -- Theorem~\ref{re-agm reg text}:
    \begin{equation*}
        f(x^N) - f^* \leqslant \varepsilon, \quad N = O \left( \left(\frac{LR^2}{\varepsilon} \right)^{1 - \beta } \ln \left(\frac{LR^2}{\varepsilon} \right) \right).
    \end{equation*}

    \item Convergence in presence of relative noise in gradient~\eqref{relative noise condition} lower bounds for strongly convex~\eqref{strong conv} and convex~\eqref{convexity} functions provided in Section~\ref{lower bound section}. For strongly convex functions lower bound presented in Theorem~\ref{lower bound relative noise mu > 0 text}:
    \begin{equation*}
        \begin{gathered}
            \text{If: } f(x^N) - f^* = O\left(LR^2 \exp \left(- O(1) \left(\frac{\mu}{L} \right)^{p(\alpha)} N \right) \right), \\
            \text{Then: } p(\alpha) \geqslant 1 - 2 \cdot \min \left \lbrace 1/4, \log_{\mu / L}(\alpha) \right \rbrace.
        \end{gathered}
    \end{equation*}
    For convex functions lower bound presented in Theorem~\ref{lower bound relative noise mu = 0 text}:
    \begin{equation*}
        \begin{gathered}
            \text{If: } f(x^N) - f^* = O\left( \frac{LR^2}{N^{p(\alpha)}} \right), \\
            \text{Then: } p(\alpha) \leqslant 1.
        \end{gathered}
    \end{equation*}
\end{itemize}

There are many other results and techniques provided below. Proofs for theorems are moved to the Appendix. Numerical experiments visualizing described methods are given in Section~\ref{exp section}.

\section{Gradient descent} \label{section gd}

\begin{algorithm}[H]
\caption{Gradient Descent}
\label{alg gd}
\begin{algorithmic}[1]
\State 
\noindent {\bf Input:} Starting point $x^0$, number of steps $N$, $h$ - learning rate.
\For {$k = 1 \dots N$}
        \State $x^{k} = x^{k - 1} - h \widetilde{\nabla}f(x^{k - 1})$
\EndFor
\State 
\noindent {\bf Output:} $x^N$.
\end{algorithmic}
\end{algorithm}

This procedure can be interpreted as a standard Euler discretization of autonomous dynamic system:
\begin{equation*}
    \dot{x} = - \widetilde{\nabla} f(x).
\end{equation*}
In the classical literature in the convex optimization~\cite{polyak1987introduction} convergence of such method was analyzed in details, including using the first and second Lyapunov methods. It is worth noting that this approach in our case has a number of disadvantages, namely, considering the Lie derivative in the direction $\widetilde{\nabla} f$:
\begin{equation*}
    \mathcal{L}_{\wtgg f} \phi = \langle \wtgg f(x), \nabla \phi (x) \rangle,
\end{equation*}
we obtain, that $\frac{1}{2} \| x - x^* \|_2^2$ will not Lyapunov function for dynamic system above, however using convexity one can note, that $\mathcal{L}_{-\nabla f} \left( \frac{1}{2} \|x - x^* \|_2^2 \right) = \langle -\nabla f(x), x - x^* \rangle \leqslant f^* - f(x) \leqslant 0$. At the same time $f$ will be Lyapunov function for noised dynamics if $\delta = 0$:
\begin{equation*}
    \mathcal{L}_{-\wtgg f} f = \langle -\wtgg f(x), \nabla f(x) \rangle \overset{\ref{cos lb}}{\leqslant} - \sqrt{1 - \alpha^2} \|\wtgg f(x) \|_2 \| \nabla f(x) \|_2 < 0.
\end{equation*}
The analysis of such a method has already been carried out under various conditions, both functions and noise models - variational inequalities~\cite{beznosikov2023stochastic}, distributed optimization~\cite{woodworth2020minibatch}, stochastic optimization~\cite{garrigos2023handbook}, $(\delta, L)$-oracle~\cite{devolder2014first, gasnikov2017universal} and generalized smoothness $(L_0, L_1)$~\cite{lobanov2024linear, gorbunov2024methods}. Therefore, it is important to consider such a method in the concept of the composition of absolute and relative deterministic noise.

\begin{theorem} \label{gd norm conv text}
    Let $f$ smooth~\eqref{smooth cond},  $\widetilde{f}$ satisfies error condition~\eqref{noise condition}. Then Algorithm~\ref{alg gd} with step:
    \begin{equation*}
        h = \left( \frac{1 - \alpha}{ 1 + \alpha} \right)^{3 / 2} \frac{1}{4 L},
    \end{equation*}
    guarantees convergence:
    \begin{eqnarray*}
            \min_{k \leqslant N} \|\nabla f(x^k) \|_2^2
            & \leqslant & \frac{(1 + \alpha)}{(1 - \alpha)^3} \frac{16L (f(x^0) - f^*)}{N + 1} + \frac{3}{(1 - \alpha)^3(1 + \alpha)} \delta^2.
    \end{eqnarray*}
\end{theorem}

In the book~\cite{polyak1987introduction}, a convergence estimation was proposed for functions satisfying Polyak-{\L}ojasiewicz condition~\eqref{PL cond} and relative error~\eqref{relative noise condition}.

\begin{theorem} \label{gd pl conv text}
    Let $f$ function satisfies P$\L$ condition~\eqref{PL cond} and smoothness~\eqref{smooth cond}, $\wtgg f$ satisfies~\eqref{noise condition}. Then Algorithm~\ref{alg gd} with step:
    \begin{equation*}
        h = \left( \frac{1 - \alpha}{ 1 + \alpha} \right)^{3 / 2} \frac{1}{4 L},
    \end{equation*}
    guarantees convergence:
    \begin{eqnarray*}
            f(x^N) - f^* 
            & \leqslant &  \left(1 -  \frac{(1 - \alpha)^3}{1 + \alpha} \frac{\mu}{8 L} \right)^N (f(x^0) - f^*) + \frac{3}{2} \frac{1 + \alpha}{(1 - \alpha)^3} \frac{\delta^2}{\mu}.
    \end{eqnarray*}
\end{theorem}

We will return to gradient descent convergence in Section~\ref{section main relative interpretation} when we cover regularization and early stopping techniques.

\section{Accelerated method (RE-AGM) } \label{re-agm section text}

Now we will consider the accelerated method proposed in~\cite{vasin2025solving}. We propose an improved version of this method that guarantees convergence under composite noise condition~\eqref{noise condition}.

\begin{algorithm}[H]
    \caption{ RE-AGM (Relative Error Accelerated Gradient Method).} 
    \label{alg re-agm}
    \begin{algorithmic}[1]
        \State {\bfseries Input:} Starting point $x^0$, number of steps $N$, $L$ - smoothness parameter,
        $\mu$ - strongly convexity parameter, $\alpha$ - relative error parameter.
        \State {\bf Set} $u^0 = x^0$.
        \State {\bf Set} $h, \omega = \text{CalculateParameters}(\alpha, \mu, L)$,
        \For { $k = 0 \dots N - 1$}
            \State $\displaystyle y^k = \frac{\omega u^k + x^k}{1 + \omega}$, 
            \State $\displaystyle u^{k + 1} = (1 - \omega) u^{k} + \omega y^k - \frac{2 \omega}{\mu} \widetilde{\nabla} f(y^{k})$, 
            \State $\displaystyle x^{k + 1} = y^{k} - h \widetilde{\nabla} f(y^{k})$.
        \EndFor
        \State
        \noindent {\bf Output:} $x^N$.

    \State

    \Function{CalculateParameters}{$\alpha, \mu, L$}
        \State {\bf Set} $\displaystyle h = \frac{1}{4L}\left(\frac{1 - \alpha}{1 + \alpha} \right)^\frac{3}{2}$,
        \State {\bf Set} $\displaystyle \widehat{L} = 8\frac{1 + \alpha}{(1 - \alpha)^3} L$, 
        \State {\bf Set} $\gamma^* = \min \left \lbrace \log_{\mu / 2L} (3 \alpha ), \frac{1}{2} \right \rbrace$.
        \State {\bf Set} $\displaystyle s = \left(1 + \frac{1}{4} \left(\frac{\mu}{2L} \right)^{\gamma^*} \right) (1 + \alpha)^2 + 2\alpha^2$.
        \State {\bf Set} $\displaystyle m = \left(1 - \frac{1}{4} \left(\frac{\mu}{2L} \right)^{\gamma^*} \right) (1 - \alpha)^2 - 2\alpha^2$.
        \State {\bf Set} $q = \mu / 2\widehat{L}$.
        \State {\bf Set} $\displaystyle \omega = \frac{(m - s) + \sqrt{(s - m)^2 +4mq }}{2 m}$,
        \State i.e. the largest root of $ \displaystyle
           m \omega^2 + (s - m) \omega - q = 0,  
        $
        \State
        \Return $\left(h, \omega \right)$.
    \EndFunction
\end{algorithmic}
\end{algorithm}

\begin{theorem} \label{acc re agm conv text}
Let $f$ be an $L$-smooth and $\mu$-strongly convex function, $\widetilde{\nabla} f$ satisfies noise condition~\eqref{noise condition}, $\alpha \in [0, 1/3]$. Then Algorithm~\ref{alg re-agm} with parameters $(L, \mu, x^0, \alpha)$ generates $x^N$, s.t. 
\begin{eqnarray*}
    f(x^N) - f^* & \leqslant & 
    \left(1 - \frac{1}{150} \left(\frac{\mu}{2L}\right)^{1 - \gamma^*} \right)^N \left(f(x^0) - f^* + \frac{\mu}{4} R^2 \right)
    +  \left( \left(\frac{2L}{\mu} \right)^{\gamma^*} + 5 \right) \frac{\delta^2}{\mu}
    \\
    & \leqslant & 
    \left(1 - \frac{1}{300} \left(\frac{\mu}{L}\right)^{1 - \gamma^*} \right)^N \left(f(x^0) - f^* + \frac{\mu}{4} R^2 \right)
    + \left(2 \left(\frac{L}{\mu} \right)^{\gamma^*}  + 5 \right) \frac{\delta^2}{\mu}
\end{eqnarray*}
where $R := \|x^0 - x^* \|_2, \gamma^* = \min \left \lbrace \log_{\mu / 2L} (3 \alpha ), \frac{1}{2} \right \rbrace$.
\end{theorem}

\begin{remark} \label{re-agm convergence remark}
    Let's consider the various options for convergence that are guaranteed by Theorem~\ref{acc re agm conv text}. Firstly, it is worth noting, that convergence rate $1 - \frac{1}{300} \left(\frac{\mu}{L}\right)^{1 - \gamma^*}$ depends on $\alpha$ and $1 - \gamma^* \leqslant 1 / 2$, which is consistent with the lower bound provided at~\cite{nemirovski1995information, nesterov2018lectures}.

    In papers~\cite{vasin2023accelerated,devolder2013exactness} (for $\alpha = 0$) the following convergence was provided for strong convex~\eqref{strong conv} and smooth~\eqref{smooth cond} function:
    \begin{equation*}
        f(x^N) - f(x^*) \leqslant L R^2 \exp{\left(-\frac{1}{2}\sqrt{\frac{\mu}{2L}}N \right)} + 2 \left(1 + \sqrt{{\frac{2L}{\mu}}} \right) \frac{\delta^2}{\mu}.
    \end{equation*}
    Using Theorem~\ref{acc re agm conv text} for $\alpha \leqslant \frac{1}{3} \left(\frac{\mu}{2 L} \right)^{1 / 2}$ we can obtain similar convergence rate:
    \begin{equation*}
        f(x^N) - f^* \leqslant \left(1 - \frac{1}{300} \left(\frac{\mu}{L}\right)^{1/2} \right)^N \left(f(x^0) - f^* + \frac{\mu}{4} R^2 \right)
        + \left(2 \left(\frac{L}{\mu} \right)^{1/2}  + 5 \right) \frac{\delta^2}{\mu}
    \end{equation*}
    If $\alpha = \frac{1}{3} \Longrightarrow \gamma^* = 0$, then using Theorem~\ref{acc re agm conv text} we obtain:
    \begin{equation*}
        f(x^N) - f^* \leqslant \left(1 - \frac{1}{300} \left(\frac{\mu}{L}\right) \right)^N \left(f(x^0) - f^* + \frac{\mu}{4} R^2 \right)
        + 7 \frac{\delta^2}{\mu},
    \end{equation*}
    that corresponds to Theorem~\ref{gd pl conv text}.
    Let $\alpha = \frac{1}{3} \left(\frac{\mu}{2 L} \right)^{1 / 4} $
    \begin{equation*}
        f(x^N) - f^* \leqslant \left(1 - \frac{1}{300} \left(\frac{\mu}{L}\right)^{3/4} \right)^N \left(f(x^0) - f^* + \frac{\mu}{4} R^2 \right)
        + \left(2 \left(\frac{L}{\mu} \right)^{1/4}  + 5 \right) \frac{\delta^2}{\mu}.
    \end{equation*}
\end{remark}

\begin{remark} \label{remark:intermediate_conv_accelerated}
    One can obtain method with intermediate convergence. That is, we can run Algorithm~\ref{alg re-agm} with $\wha = \frac{1}{3} \left(\frac{\mu}{2 L} \right)^{p}, \; \frac{1}{2} \leqslant p \leqslant 1$. If $\wtgg f$ satisfies~\eqref{absolute noise condition}, then Algorithm~\eqref{alg re-agm} with parameter $\alpha = \wha$ guarantees convergence:
    \begin{equation*}
        f(x^N) - f^* \leqslant \left(1 - \frac{1}{300} \left(\frac{\mu}{L}\right)^{1 - p} \right)^N \left(f(x^0) - f^* + \frac{\mu}{4} R^2 \right)
        + \left(2 \left(\frac{L}{\mu} \right)^{p}  + 5 \right) \frac{\delta^2}{\mu}.
    \end{equation*}
    Such convergence satisfies lower bound described at Theorem~\ref{lower bound oracle absolute}.
\end{remark}

\section{Regularization} \label{section regularization}

We introduce common technique (for example \cite{vasin2023accelerated}) for obtaining convergence rate for convex functions~\eqref{convexity} $f$. Such method proposes to consider a new function:
\begin{equation} \label{reg def}
    f_{\mu}(x | x^0) = f(x) + \frac{\mu}{2} \|x - x^0 \|_2^2.
\end{equation}
Let us introduce the notation:
\begin{equation*}
    \begin{gathered}
        x_\mu^* \text{ - solution for problem: } f_{\mu}(\cdot | x^0) \to \min_{x \in \mathbb{R}^n}, \\
        f_{\mu}^* = f_{\mu}(x_{\mu} | x^0), \\
        R_\mu = \|x_{\mu}^* - x^0 \|_2, \\
        \wtgg f_{\mu}(x) = \wtgg f(x) + \mu (x - x^0).
    \end{gathered}
\end{equation*}

This approach allows us to reduce degenerate problems to strongly convex ones. The inverse reduction (from strongly convex to degenerate problems) is restarts (such trick was used at~\cite{kornilov2025intermediate}), in our paper we use it in proof of Theorem~\ref{lower bound relative noise mu = 0}. Consider usage of regularization technique for Algorithm~\ref{alg gd} and Algorithm~\ref{alg re-agm}.

\begin{theorem} \label{gd reg text}
    Let $f$ is convex~\eqref{convexity} and $L$ smooth~\eqref{smooth cond}, $\widetilde{\nabla} f$ satisfies relative noise condition~\eqref{relative noise condition} with $\alpha < 1 / 2$.
    Then we can solve problem~\eqref{optim} with precision $f(x^N) - f^* \leqslant \varepsilon \; (\varepsilon < LR^2)$, using Algorithm~\ref{alg gd} via regularization technique with following amount of iterations:
    \begin{equation*}
        N \leqslant  12 \frac{(1 + \alpha)^2}{(1 - \alpha)^6} \frac{LR^2}{\varepsilon} \ln \left(\frac{2 LR^2}{\varepsilon} \right) + 1.
    \end{equation*}
\end{theorem}

\begin{theorem} \label{re-agm reg text}
    Let $f$ is convex~\eqref{convexity} and $L$ smooth~\eqref{smooth cond}, $\widetilde{\nabla} f$ satisfies relative noise condition~\eqref{relative noise condition} with $\alpha \leqslant \frac{1}{3} \left(\frac{\varepsilon}{12 LR^2} \right)^{\beta}, 0 \leqslant \beta \leqslant 1/2$.
    Then we can solve problem~\eqref{optim} with precision $f(x^N) - f^* \leqslant \varepsilon \; (\varepsilon < LR^2)$, using Algorithm~\ref{alg re-agm} via regularization technique with following amount of iterations:
    \begin{equation*}
        N \leqslant 150 \left(\frac{12 LR^2}{\varepsilon} \right)^{1 - \beta } \ln \left(\frac{4 LR^2}{\varepsilon} \right) + 1.
    \end{equation*}
\end{theorem}

\section{Adaptiveness} \label{main Adap}

Adaptivity is a fairly popular tool for convergence analysis when the function parameters are unknown, for example papers~\cite{dvinskikh2019adaptive, dvurechensky2017adaptive, gasnikov2019adaptive, ivanova2020adaptive, ivanova2021adaptive, kornilov2025intermediate}. In this paper we will consider adaptivity on $\alpha$ (relative noise~\eqref{relative noise condition} parameter) and $L$ (smoothness~\eqref{smooth cond}). Usually adaptivity used in case, when $L$ parameter is unknown, particularly adaptive methods applied to different oracle models, for example $(\delta, L, \mu)$ model~\eqref{delta L mu oracle}. This technique also allows to provide universal method~\cite{yurtsever2015universal, gasnikov2017universal}, taking into account the Hölder continuity~\eqref{def:holder continity}.
Such universal methods provides $\varepsilon$ solution~\eqref{def:func_eps_solution} by:
\begin{equation*}
    N = \inf_{\nu \in [0; 1]} \left( \frac{2 L_{\nu} R^{1 + \nu} }{\varepsilon} \right)^{\frac{2}{1 + \nu}} + 1 \text{, iterations},
\end{equation*}
in this case the set $\lbrace L_{\nu} \rbrace_{\nu \in [0; 1]}$ may be unknown and some of $L_{\nu}$ can equal $+\infty$. In this paper we will not take into account the possible Hölder~\eqref{def:holder continity} property of the function and assume only smoothness, what is more the corresponding parameter $L$ may be unknown.

We can introduce modification of Algorithm~\eqref{alg gd}, which can adapt relative error parameter $\alpha$ and smothenss $L$.
\begin{algorithm}[H]
\caption{Relative Noise Adaptive Gradient Descent $(N, L_0, x^0, \delta, \tau)$}
\label{alg rel adaptive gd}
\begin{algorithmic}[1]
\State
\noindent {\bf Input:} Starting point $x^0$, number of steps $N$, $L_0$ - initial estimation of smoothness parameter, $\delta$ - absolute error, $\tau \in \lbrace \textbf{false}, \textbf{true} \rbrace$ - whether to adapt the smoothness parameter.
\State {\bf Set} $J_0 = 1$.
\For {$k = 0 \dots N - 1$}
    \State $\wha_k, \whL_k, \theta_k, x^{k + 1} = \text{UpdateCoefficientsAndMakeStep}\left(J_k, L_0, x^k, \wtgg f, \tau \right)$.

    \While {$f(x^{k + 1}) > f(x^{k}) - \theta_k \| \wtgg f(x^{k}) \|_2^2 + \frac{3}{4 (1 + \wha_k)^2 \whL_k} \delta^2$}
        \State $J_k = J_k + 1$.
        \State $\wha_k, \whL_k, \theta_k, x^{k + 1} = \text{UpdateCoefficientsAndMakeStep}\left(J_k, L_0, x^k, \wtgg f, \tau \right)$.
    \EndWhile
   \State $J_{k + 1} = \max(1, J_{k} - 1)$.
\EndFor
\State 
\noindent {\bf Output:} $x^N$.

\State

\Function{UpdateCoefficientsAndMakeStep}{$t, L_0, x, \wtgg f, \tau$}
    \State $\wha = 1 - 2^{-t}$.
    \If{$\tau$ is \textbf{true}}
        \State $\whL = L_0 \cdot 2^t$.
    \Else
        \State $\whL = L_0$.
    \EndIf
    \State $h = \frac{1}{4 \whL} \left(\frac{1 - \wha}{1 + \wha} \right)^\frac{1}{2}$.
    \State $\theta = \frac{1}{32 \whL} \frac{1 - \wha}{1 + \wha}$.
    \State $y = x - h \wtgg f(x)$.
    \State
    \Return $\left(\wha, \whL, \theta, y \right)$.
\EndFunction
\end{algorithmic}
\end{algorithm}

Usually, adaptive algorithms use $L_{k} = 2 L_k$ and $L_{k + 1} = L_k / 2$. In the Algorithm above we used sequence $J_k$ to update sequences $\wha_k, \whL_k$ simultaneously. The Algorithm above uses the idea of binary search by $\widehat{\alpha}_k$ and $\whL_k$ with predicate, corresponding to Gradient Descent step decreasing the value of a function (check Lemma~\ref{GD step}). Algorithm~\ref{alg rel adaptive gd} provides option $\tau$, which enables/disables $L$-smoothness adaptiveness. We will provide two theorems below with $\tau = \textbf{true}$ and $\tau = \textbf{false}$.

\begin{theorem} \label{adaptive PL convergence alpha and smooth text}
    Let $f$ function satisfies P$\L$ condition~\eqref{PL cond} and smoothness~\eqref{smooth cond}, $\wtgg f$ satisfies error condition~\eqref{noise condition}. Then Algorithm~\ref{alg rel adaptive gd} with $\tau = \textbf{true}$ (enable smooth adaptiveness) provides convergence:
    \begin{eqnarray*}    
        f(x^N) - f^* & \leqslant & \left( 1 - \frac{(1 - \alpha)^2}{256} \min \left\lbrace (1 - \alpha)^{2}, \left(\nicefrac{L_0}{L} \right)^2 \right \rbrace \frac{\mu}{L_0} \right)^N (f(x^0) - f^*)
        \\
        & + & \frac{200}{(1 - \alpha)^{2}} \max \left\lbrace (1 - \alpha)^{-2}, \left(\nicefrac{L}{L_0} \right)^2 \right\rbrace \frac{\delta^2}{\mu},
    \end{eqnarray*}
    and total number iterations of the inner loop is bounded by:
    \begin{equation*}
         N + \max \left \lbrace \log_2\left((1 - \alpha)^{-1}\right), \log_2\left(\nicefrac{L}{L_0} \right) \right \rbrace + 1.
    \end{equation*}
\end{theorem}

\begin{theorem}\label{adaptive PL convergence only alpha text}
    Let $f$ function satisfies P$\L$ condition~\eqref{PL cond} and $L$-smoothness~\eqref{smooth cond}, $\wtgg f$ satisfies error condition~\eqref{noise condition}. Then Algorithm~\ref{alg rel adaptive gd} with $\tau = \textbf{false}$ (disable smooth adaptiveness) and $L_0 = L$ provides convergence:
    \begin{eqnarray*}    
        f(x^N) - f^* \leqslant \left( 1 - \frac{(1 - \alpha)^3}{128} \frac{\mu}{L} \right)^N (f(x^0) - f^*)
       + \frac{100}{(1 - \alpha)^{3}} \frac{\delta^2}{\mu},
    \end{eqnarray*}  
    and total number iterations of the inner loop is bounded by:
    \begin{equation*}
         N +  \log_2\left((1 - \alpha)^{-1}\right) + 1.
    \end{equation*}
\end{theorem}

\section{Relative interpretation of absolute noise} \label{section main relative interpretation}

As it was shown in Theorem~\ref{acc re agm conv text} component of absolute noise accumulated as $\left(\frac{L}{\mu} \right)^{\gamma^*} \frac{\delta^2}{\mu}$. In papers~\cite{vasin2023accelerated, devolder2013first, devolder2014first, stonyakin2021inexact} considering absolute noise (without relative component $\zeta_r$) it was shown that accelerated methods contain an absolute noise component in convergence rate as $\sqrt{\frac{L}{\mu}} \delta^2 / \mu$. On the contrary, as it was shown at Theorem~\ref{gd pl conv text} that type of noise is accumulated accordingly to $\delta^2 / \mu$. Obviously, that in presence of absolute noise $\zeta_a$ with magnitude $\delta$ any first order method can not guarantee any convergence after gradient norm $\| \nabla f(x) \|_2$ reaches $\delta$ level, which corresponds to $f(x) - f^* \leqslant \delta^2 / 2 \mu$, by P$\L$ condition~\eqref{PL cond}. However, the presence of relative noise $\zeta_r$ in the gradient does not add an additive term to the convergence of the method. Therefore, the idea arises to interpret the absolute component $\zeta_a$ of noise as part of the relative one. Let the magnitude of the additive noise be known and the method stops as soon as the gradient norm reaches $K \delta$ for some $K > 1$, then we can re-estimate relative noise level:
\begin{equation*}
    \widehat{\alpha} = \frac{\alpha \|\nabla f(x) \|_2 + \delta}{\|\nabla f(x) \|_2} \leqslant \alpha + \frac{\delta}{K \delta} = \alpha + \frac{1}{K}.
\end{equation*}
That is we can set $K = \left(\frac{L}{\mu} \right)^{\gamma}$:
\begin{equation*}
    \widehat{\alpha} = \alpha + \left(\frac{\mu}{L} \right)^{\gamma},
\end{equation*}
and we can converge to $f(x) - f^* \leqslant \left(\frac{L}{\mu} \right)^{2 \gamma} \delta^2 /\mu$. Such reasoning leads us to the following Theorem.

\begin{theorem} \label{main relative interpretation th text}
    Let $f$ is $\mu$-strongly convex~\eqref{strong conv} and $L$-smooth~\eqref{smooth cond}, $\wtgg f$ satisfies~\eqref{noise condition}. Algorithm $\mathcal{A}$ is first order method~\eqref{linear first order method}, using $\wtgg f$ satisfying relative noise~\ref{relative noise condition} with magnitude $\alpha_0$ produces $x^N$, such that:
    \begin{equation*}
        f(x^N) - f^* \leqslant C_0 LR^2 \exp \left(-A_0 \left(\frac{\mu}{L} N \right)^{\gamma(\alpha_0)} \right)^N,
    \end{equation*}
    where:
    \begin{equation*}
        \gamma: [0; 1) \to [1 / 2; + \infty]
    \end{equation*}
    Then for any $K > (1 - \alpha)^{-1}$ we can solve the optimization problem~\eqref{optim} with precision $\varepsilon_0$, where
    \begin{equation*}
        f(x^{N_0}) - f^*
        \leqslant \varepsilon_0 = 
        \frac{1}{(1 - \alpha)^2} \left( \left((1 + \alpha)K + 1 \right)^2 + 1 \right) \frac{\delta^2}{\mu},
    \end{equation*}
    and
    \begin{equation*}
        N_0 \leqslant \frac{1}{A_0} \left(\frac{L}{\mu} \right)^{\gamma(\widehat{\alpha})} \ln \left( \frac{(1 - \alpha)^2}{\left((1 + \alpha)K + 1 \right)^2 + 1} \cdot C_0 \cdot \frac{\mu}{\delta^2} \right),
    \end{equation*}
    using algorithm $\mathcal{A}$ with a relative error parameter $\widehat{\alpha} = \alpha + \frac{1}{K}$.
\end{theorem}

\begin{remark}
    Theorem~\eqref{main relative interpretation th text} can be interpreted as stopping criterion:
    \begin{equation*}
        \| \wtgg f(x) \|_2 \leqslant K \delta.
    \end{equation*}
    For absolute noise in gradient~\eqref{absolute noise condition} such criteria was already proposed in~\cite{poljak1981iterative, vasin2023accelerated, stonyakin2023stopping}. Usually such criteria used for convex~\eqref{convexity} (degenerate case, $\mu = 0$) functions, however in this paper and in~\cite{stonyakin2023stopping} used for strongly convex~\eqref{strong conv} functions.
\end{remark}

\begin{theorem} \label{re-agm rel interpretation text}
    Let $f$ is $\mu$-strongly convex~\eqref{strong conv} and $L$-smooth~\eqref{smooth cond}, $\wtgg f$ satisfies~\eqref{noise condition}, $\alpha \leqslant \frac{1}{6} \left(\frac{\mu}{2L}\right)^{\gamma_0}, 0 \leqslant \gamma_0 \leqslant \frac{1}{2}$. Using Algorithm~\ref{alg re-agm} with parameters $(L, \mu, x^0, \widehat{\alpha})$, where
    \begin{equation*}
        \widehat{\alpha} = \alpha + \frac{1}{6} \left(\frac{\mu}{2L}\right)^{\beta}, \quad 0 \leqslant \beta \leqslant \frac{1}{2},
    \end{equation*}
    we can obtain:
    \begin{equation*}
        \min \left \lbrace f(y^{N}) - f^*, f(x^{N}) - f^* \right \rbrace
        \leqslant \frac{1}{(1 - \alpha)^2} \left( \left(6(1 + \alpha)\left(\frac{2L}{\mu}\right)^{\beta} + 1 \right)^2 + 1 \right) \frac{\delta^2}{\mu}.
    \end{equation*}
    We will forcefully stop the algorithm based on the condition:
    \begin{equation*}
        \|\wtgg f(x^N) \|_2 \leqslant \left(6(1 + \alpha) \left(\frac{2L}{\mu}\right)^{\beta} + 1 \right) \delta,
    \end{equation*}
    in this case, the method will work no more than
    \begin{equation*}
        N = 300 \left(\frac{L}{\mu} \right)^{1 - \min\left \lbrace \gamma_0, \beta \right \rbrace} \ln \left( \frac{(1 - \alpha)^2 }{\left(6(1 + \alpha) \left(\frac{2L}{\mu}\right)^{\beta} + 1 \right)^2 + 1} \cdot \frac{LR^2}{\nicefrac{\delta^2}{\mu}} \right).
    \end{equation*}
\end{theorem}

\begin{remark} \label{accelerated noise accum remark}
    Assume $\alpha = 0$. In paper~\cite{vasin2023accelerated} the following convergence for accelerated method was provided for strong convex~\eqref{strong conv} and smooth~\eqref{smooth cond} function:
    \begin{equation*}
        f(x^N) - f(x^*) \leqslant L R^2 \exp{\left(-\frac{1}{2}\sqrt{\frac{\mu}{2L}}N \right)} + 2 \left(1 + \sqrt{{\frac{2L}{\mu}}} \right) \frac{\delta^2}{\mu}.
    \end{equation*}
    Therefore, to achieve the $\sim \sqrt{\frac{L}{\mu}} \frac{\delta^2}{\mu}$ level of $f(x^N) - f^*$ we will need:
    \begin{equation*}
        N = O \left( \sqrt{\frac{L}{\mu}} \ln \left( \sqrt{\frac{\mu}{L}} \frac{LR^2}{\nicefrac{\delta^2}{\mu}} \right) \right), \text{ iterations}.
    \end{equation*}
    Using Theorem~\ref{re-agm rel interpretation text} with $\beta = \frac{1}{4}$ we can obtain the $\sim \sqrt{\frac{L}{\mu}} \frac{\delta^2}{\mu}$ level of $f(x^N) - f^*$ after:
    \begin{equation*}
        N = O \left( \left(\frac{L}{\mu} \right)^{\frac{3}{4}} \ln \left( \sqrt{\frac{\mu}{L}} \frac{LR^2}{\nicefrac{\delta^2}{\mu}} \right) \right), \text{ iterations}.
    \end{equation*}
    Now we can consider the case $\beta = \frac{1}{8}$, then from result of Theorem~\ref{re-agm rel interpretation text} Algorithm~\ref{alg re-agm} reaches $~\sim \left(\frac{L}{\mu} \right)^{1 / 4} \frac{\delta^2}{\mu}$ after:
    \begin{equation*}
        N = O \left( \left(\frac{L}{\mu} \right)^{\frac{7}{8}} \ln \left( \left(\frac{\mu}{L} \right)^{1 / 4} \frac{LR^2}{\nicefrac{\delta^2}{\mu}} \right) \right), \text{ iterations}.
    \end{equation*}
    To reach level $\sim \frac{\delta^2}{\mu}$ ($\beta = 0$) it will required:
    \begin{equation*}
        N = O \left( \frac{L}{\mu} \ln \left(\frac{LR^2}{\nicefrac{\delta^2}{\mu}} \right) \right), \text{ iterations},
    \end{equation*}
    which corresponds to Gradient Descent~\ref{alg gd} convergence, Theorem~\ref{gd pl conv text}.
\end{remark}
We can also combine the regularization technique (Section~\ref{section regularization}) with the one described in this Section. This combination yields the following theorem:
\begin{theorem} \label{re-agm rel interpret and reg text}
    Let $f$ is convex~\eqref{convexity} and smooth~\eqref{smooth cond}, $\varepsilon \leqslant L R^2$, $\widetilde{\nabla} f$ satisfies relative noise condition with:
    \begin{equation*}
        \begin{gathered}
            0 < \alpha \leqslant \frac{1}{9} \left(\frac{\varepsilon}{2 L R^2} \right)^{\tau}, \quad \text{where } 0 \leqslant  \tau \leqslant \frac{1}{2}, \\
            \| \widetilde{\nabla} f(x) - \nabla f(x) \|_2 \leqslant \alpha \|\nabla f(x) \|_2.
        \end{gathered}
    \end{equation*}
    Then we can solve problem~\eqref{optim} with precision $\varepsilon$ ($f(x^N) - f^* \leqslant \varepsilon)$, using Algorithm~\ref{alg re-agm} via regularization technique with complexity:
    \begin{equation*}
        N = 72000 \left(\frac{L R^2}{\varepsilon} \right)^{1 - \tau} \ln \left( \frac{480 L R^2}{\varepsilon} \right).
    \end{equation*}
\end{theorem}

\section{Lower Bounds} \label{lower bound section}

In this section, we propose lower bounds for the convergence of first order methods~\ref{linear first order method} in the presence of relative noise~\eqref{relative noise condition}. We will reduce absolute noise~\eqref{absolute noise condition} to relative one and use the lower bounds that have been proposed~\cite{devolder2013first}.

\begin{theorem} \label{lower bound relative noise mu > 0 text}
    Let $f$ is $\mu$-strongly convex~\eqref{strong conv} and $L$-smooth~\eqref{smooth cond}, algorithm $\mathcal{A}$ is first order gradient method such that using $\wtgg f$ satisfying relative noise~\eqref{relative noise condition} with magnitude $\alpha$, produces $x^N$, such that:
    \begin{equation*}
        f(x^N) - f^* \leqslant C_0 LR^2 \exp \left(- A_0 \left(\frac{\mu}{L} \right)^{p(\alpha)} N \right),
    \end{equation*}
    where:
    \begin{equation*}
        p: [0; 1) \to [1 / 2; + \infty).
    \end{equation*}
    Then:
    \begin{equation*}
        p(\alpha) \geqslant 1 - 2 \cdot \min \left \lbrace 1/4, \log_{\mu / L}(\alpha) \right \rbrace.
    \end{equation*}
\end{theorem}

\begin{theorem} \label{lower bound relative noise mu = 0 text}
    Let $f$ is convex~\eqref{convexity} and $L$-smooth~\eqref{smooth cond}, algorithm $\mathcal{A}$ is first order gradient method.

    1. If solving problem $f(x^N) - f^* \leqslant \varepsilon$ using the algorithm $\mathcal{A}$ in the presence of relative noise~\ref{relative noise condition} with magnitude $\alpha$ requires
    \begin{equation*}
        N_{\varepsilon} \leqslant C_1 \left(\frac{L R^2}{\varepsilon} \right)^{1 / p(\alpha, \varepsilon)}, \quad p(\alpha, \varepsilon) \leqslant 2.
    \end{equation*}
    iterations, then:
    \begin{equation*}
        p(\alpha, \varepsilon) \leqslant \frac{1}{1 - 2 \min \left \lbrace 1/4, \log_{LR^2 / \varepsilon}(\alpha) \right \rbrace}.
    \end{equation*}

    2. If algorithm $\mathcal{A}$ in the presence of relative noise~\ref{relative noise condition} with magnitude $\alpha$ has convergence:
    \begin{equation*}
        f(x^N) - f^* \leqslant C_2 \frac{LR^2}{N^{p(\alpha)}},
    \end{equation*}
    then:
    \begin{equation*}
        p(\alpha) \leqslant 1.
    \end{equation*}
\end{theorem}

\begin{remark} \label{remark: unbiguos usage}
    In the conditions of Theorem~\ref{lower bound relative noise mu > 0 text} and Theorem~\ref{lower bound relative noise mu = 0 text} lower bounds are implicitly used for the corresponding classes of functions (described at~\cite{nesterov2018lectures, nemirovski1995information}). For strongly convex functions:
    \begin{equation*}
        p(\alpha) \geqslant 1/2,
    \end{equation*}
    and for convex one:
    \begin{equation*}
        p(\alpha) \leqslant 2.
    \end{equation*}
\end{remark}

\begin{remark} \label{remark:lower bound precision}
    In Theorem~\ref{acc re agm conv text} proved convergence for Algorithm~\ref{alg re-agm}, which in designations of Theorem~\ref{lower bound relative noise mu > 0 text} can be described as:
    \begin{equation*}
        p(\alpha) \leqslant 1 - \min \left \lbrace \nicefrac{1}{2}, \log_{\mu / L}(\alpha) \right \rbrace.
    \end{equation*}
    For convex case ($\mu = 0$) in Theorem~\ref{gd reg text} for Algorithm~\ref{alg gd} proved, that problem~\eqref{def:func_eps_solution} can be solved using following amount of iterations:
    \begin{equation*}
        N = O \left( \frac{LR^2}{\varepsilon} \ln \left( \frac{LR^2}{\varepsilon} \right) \right).
    \end{equation*}
    Compared to the lower bound (Theorem~\ref{lower bound relative noise mu = 0 text}), the upper bound (for $\alpha < \nicefrac{1}{2}$) differs by the log factor.
\end{remark}

\section{Experiments} \label{exp section}

In this section we will provide numerical experiments for algorithms which were introduced in sections~\eqref{section gd},\eqref{re-agm section text},\eqref{section main relative interpretation}. For algorithms above we will use "worst" on particular class functions for first-order algorithms. That functions were introduced in~\cite[p. 69, p.78]{nesterov2018lectures}. For degenerate case ($\mu$ = 0):
\begin{equation} \label{nesterov mu=0}
    \begin{gathered}
        f(x) = \frac{L}{8} \left( x_1^2 + \displaystyle\sum_{j = 0}^{k - 1} \left(x_j - x_{j + 1}\right)^2 + x_k^2 \right) - \frac{L}{4} x_1, \\
        x^* = \left(1 - \frac{1}{k + 1}, \: \dots \:,  1 - \frac{k}{k + 1}, \: 0, \: \dots \:, 0\right)^T, \\
        1 \leqslant k \leqslant \dim{x},
    \end{gathered}
\end{equation}
for strongly convex case ($\mu > 0$):
\begin{equation} \label{nesterov mu>0}
    \begin{gathered}
        f(x) = \frac{\mu \left(\chi - 1\right)}{8} \left(x_1^2 + \displaystyle\sum_{j = 1}^{n - 1} {\left(x_j - x_{j + 1}\right)^2} - 2x_1\right) + \frac{\mu}{2} \|x \|_2^2, \\
        \chi = \frac{L}{\mu}. 
    \end{gathered}
\end{equation}

Components of noise $\zeta_r, \zeta_a$ were generated independently, randomly, uniformly and unbiased, according to condition~\eqref{noise condition}:
\begin{equation} \label{uniform noise gen}
    \mathbb{E} \left[ \widetilde{\nabla}f(x) \big| x \right] = \nabla f(x).
\end{equation}

\subsection{Gradient Descent} \label{gd exps}

In this section we will consider numerical experiments for the Algorithm~\ref{alg gd}. Firstly convex degenerated case ($\mu = 0$):

\begin{figure}[H]
	\begin{center}
		\includegraphics[width=1\linewidth]{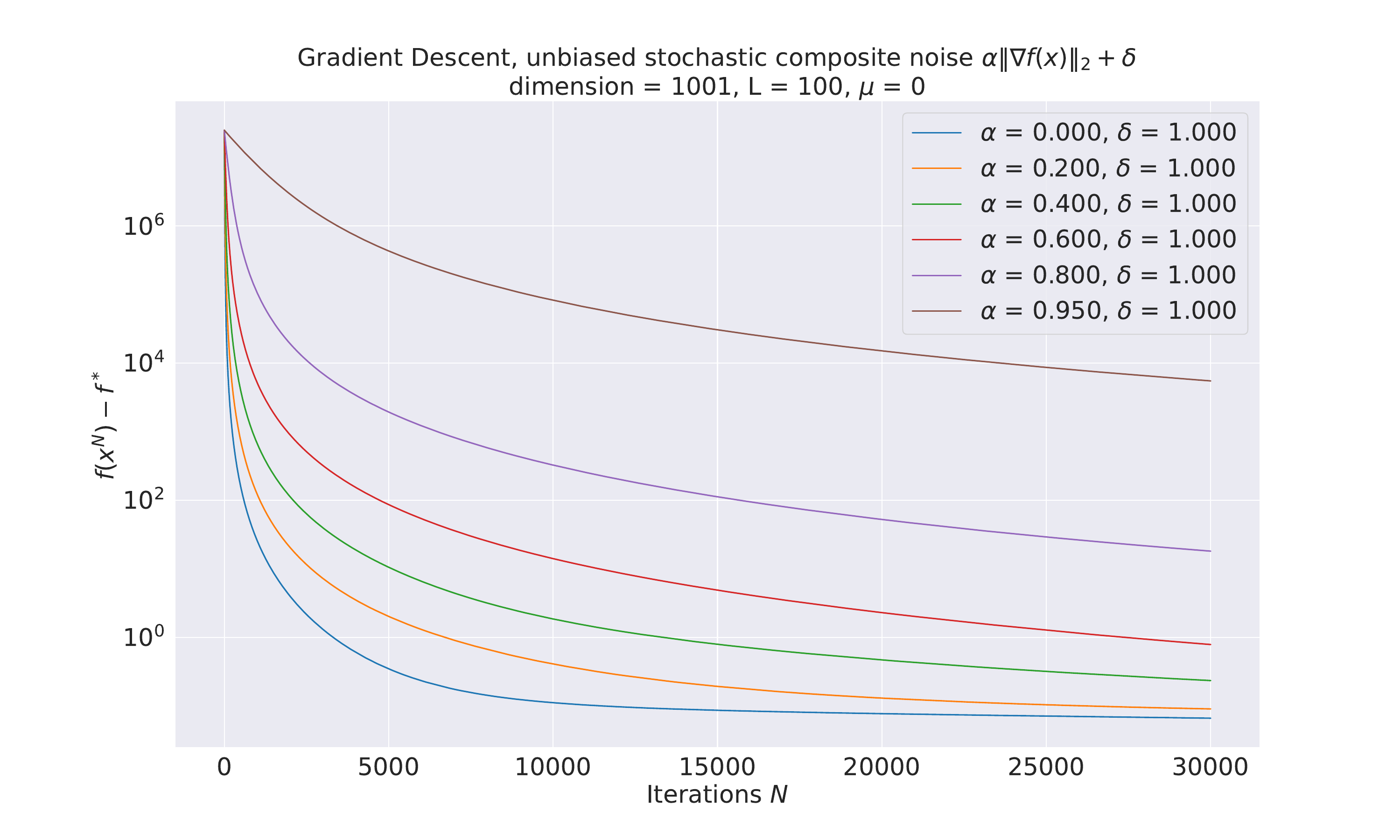}
    	\caption{The performance of GD with composite noise and $\mu = 0, L = 100$ for different values of $\alpha$ and $\delta = 1$.}
	\end{center}
\end{figure}

\begin{figure}[H]
	\begin{center}
		\includegraphics[width=1\linewidth]{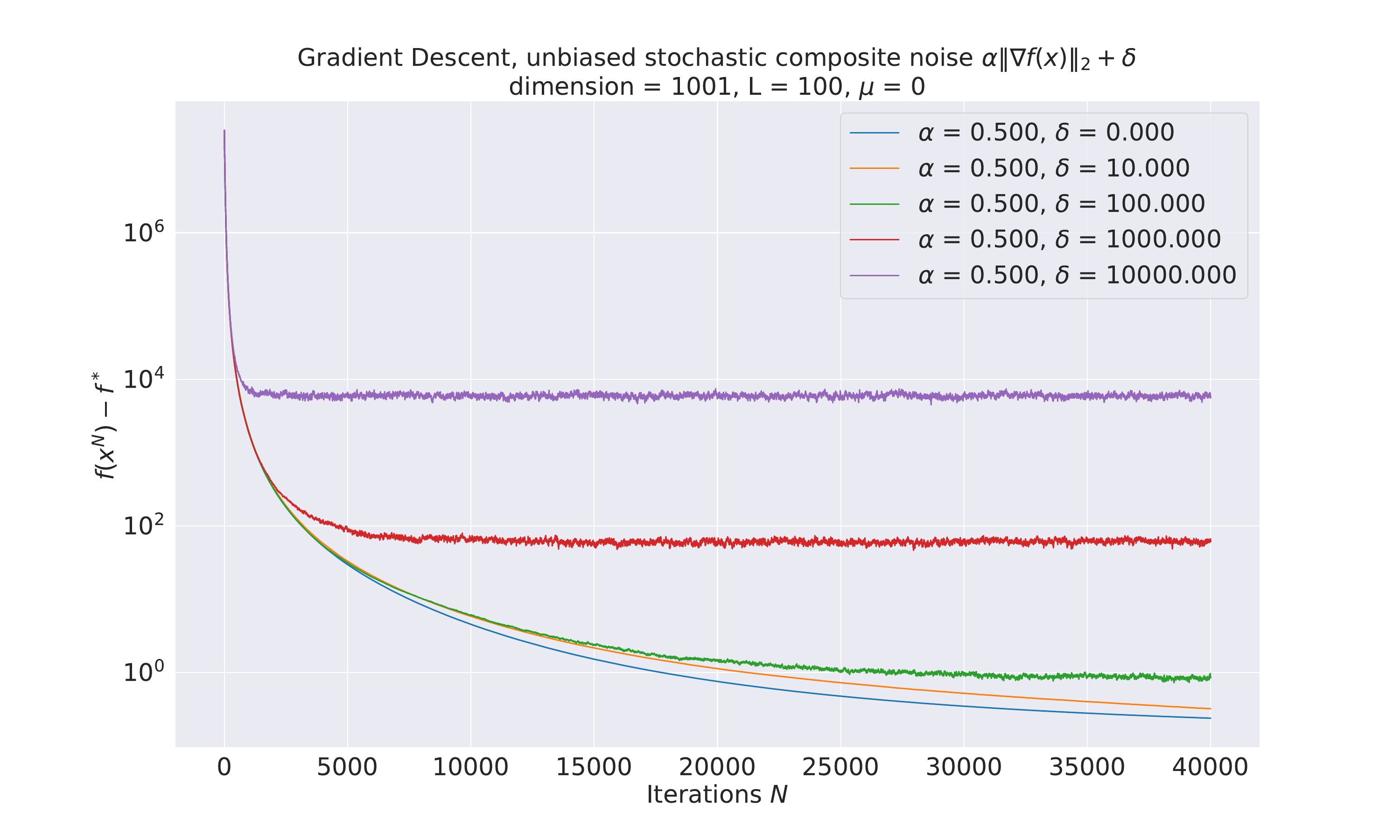}
    	\caption{The performance of GD with composite noise and $\mu = 0, L = 100$ for different values of $\delta$ and $\alpha = 0.5$.}
	\end{center}
\end{figure}

Then we move on to the strongly convex case:

\begin{figure}[H]
	\begin{center}
		\includegraphics[width=1\linewidth]{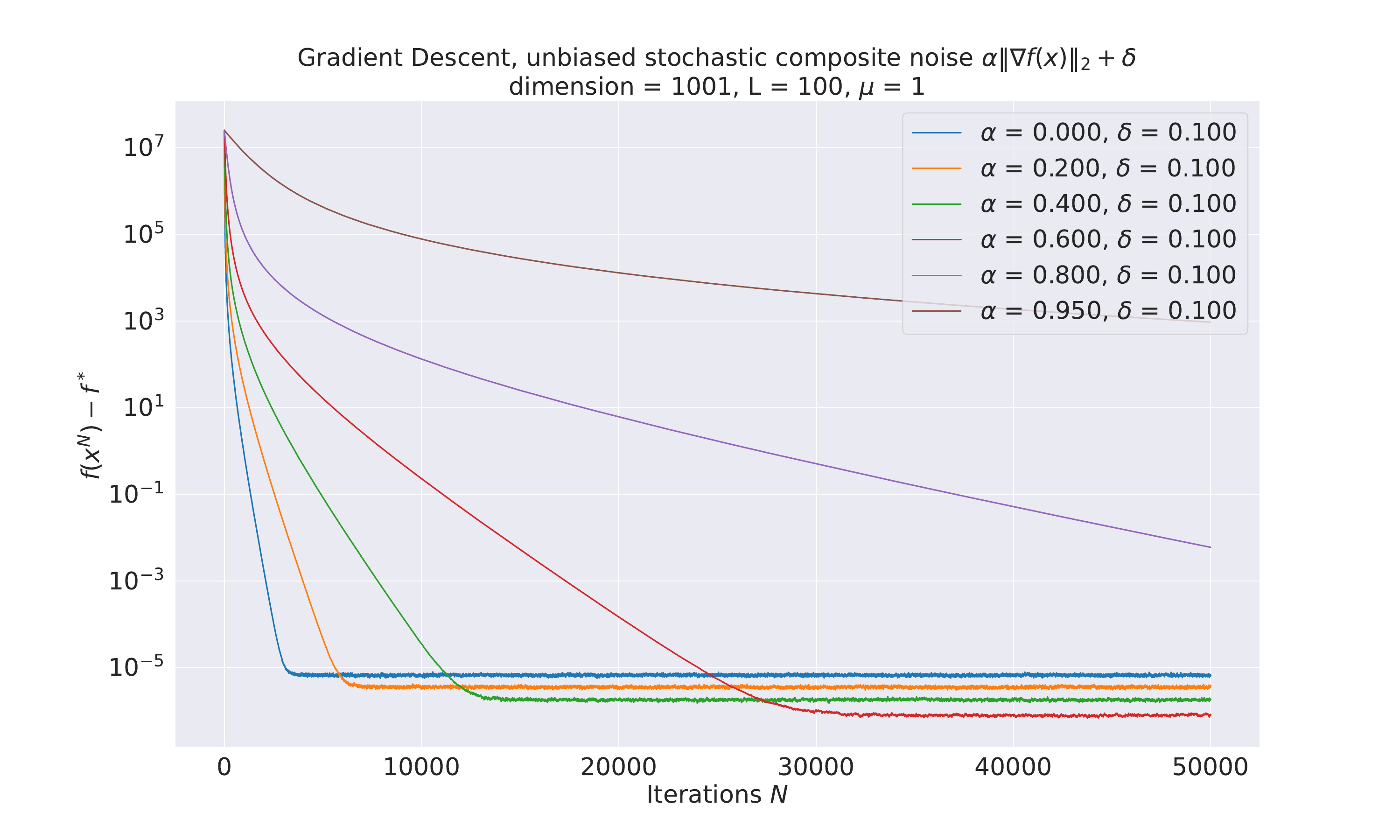}
    	\caption{The performance of GD with composite noise and $\mu = 1, L = 100$ for different values of $\alpha$ and $\delta = 0.1$.}
	\end{center}
\end{figure}

\begin{figure}[H]
	\begin{center}
		\includegraphics[width=1\linewidth]{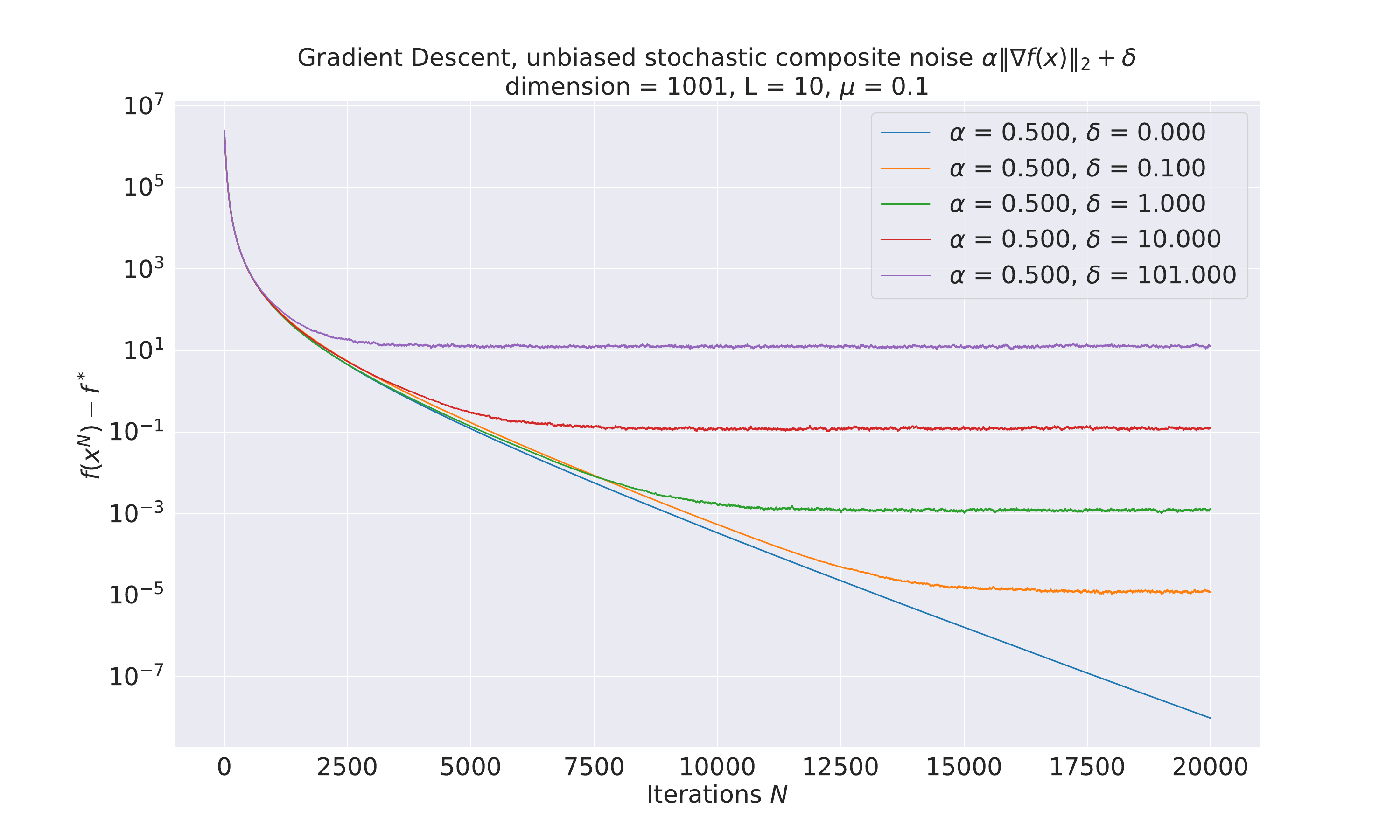}
    	\caption{The performance of GD with composite noise and $\mu = 1, L = 100$ for different values of $\delta$ and $\alpha = 0.5$.}
	\end{center}
\end{figure}

We see that for different $\alpha$ values and fixed $\delta$ value the Algorithm~\ref{alg gd} continues to converge, but at a reduced level of convergence. On the contrary, for various $\delta$ values, the algorithm converges to the corresponding level. It's worth noting that this convergence effect is due to the stochastic nature and unbiased nature of the noise.

\subsection{RE-AGM} \label{re-agm exps}

Algorithm~\ref{alg re-agm} works only with strongly convex functions, that is we will test it only on function~\eqref{nesterov mu>0}. To begin, we will consider the different levels of noise accumulation $\left(\frac{L}{\mu} \right)^{\gamma^*} \frac{\delta^2}{\mu}$ described in the Theorem~\ref{acc re agm conv text}.

\begin{figure}[H]
	\begin{center}
		\includegraphics[width=1\linewidth]{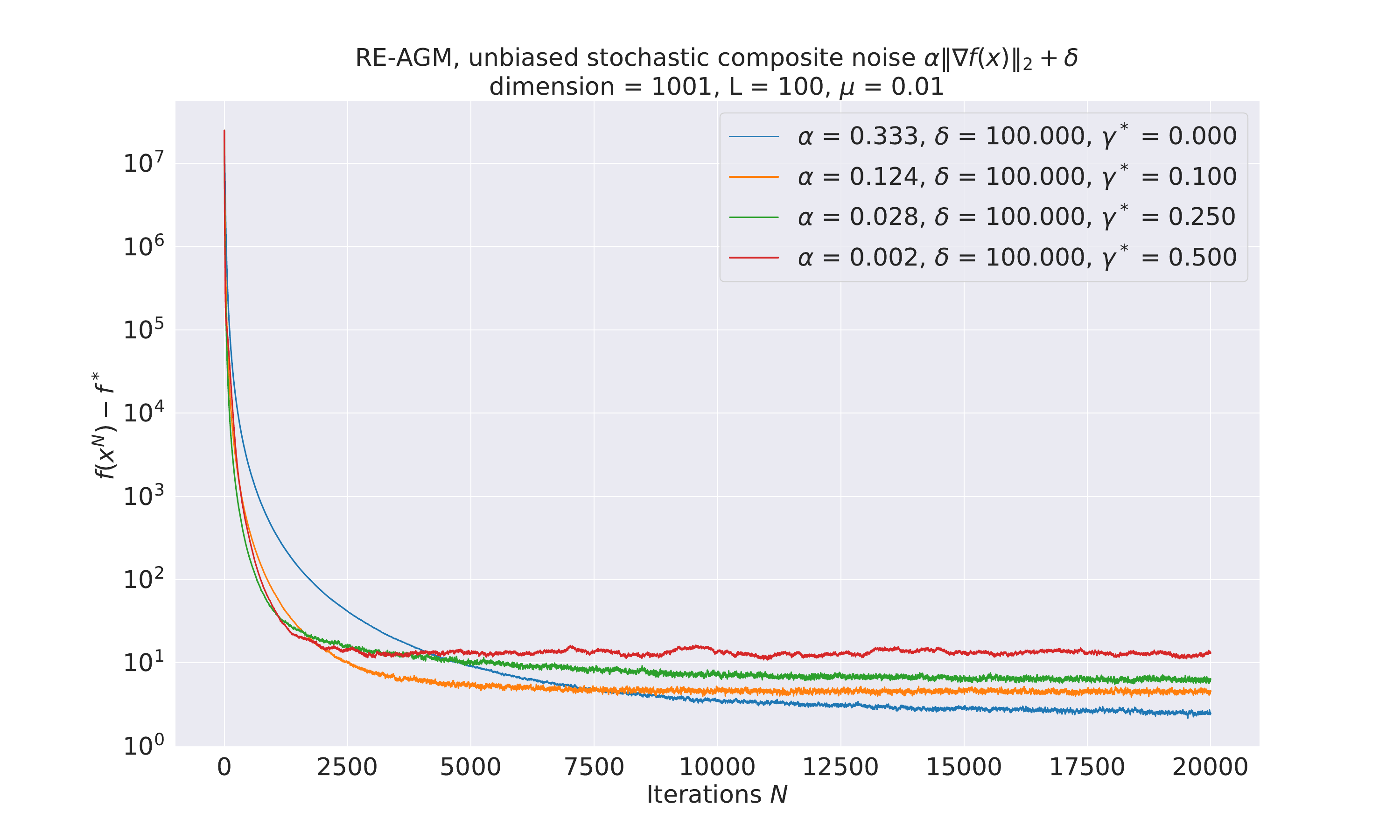}
    	\caption{The performance of RE-AGM with composite noise and $\mu = 0.01, L = 100$ for     different values of $\alpha$ and $\delta = 100$.}
	\end{center}
\end{figure}

It can be seen in the figure above that as the $\alpha$ value decreases, the limit to which the method converges increases, which demonstrates the intermediate nature of convergence rate.

\begin{figure}[H]
	\begin{center}
		\includegraphics[width=1\linewidth]{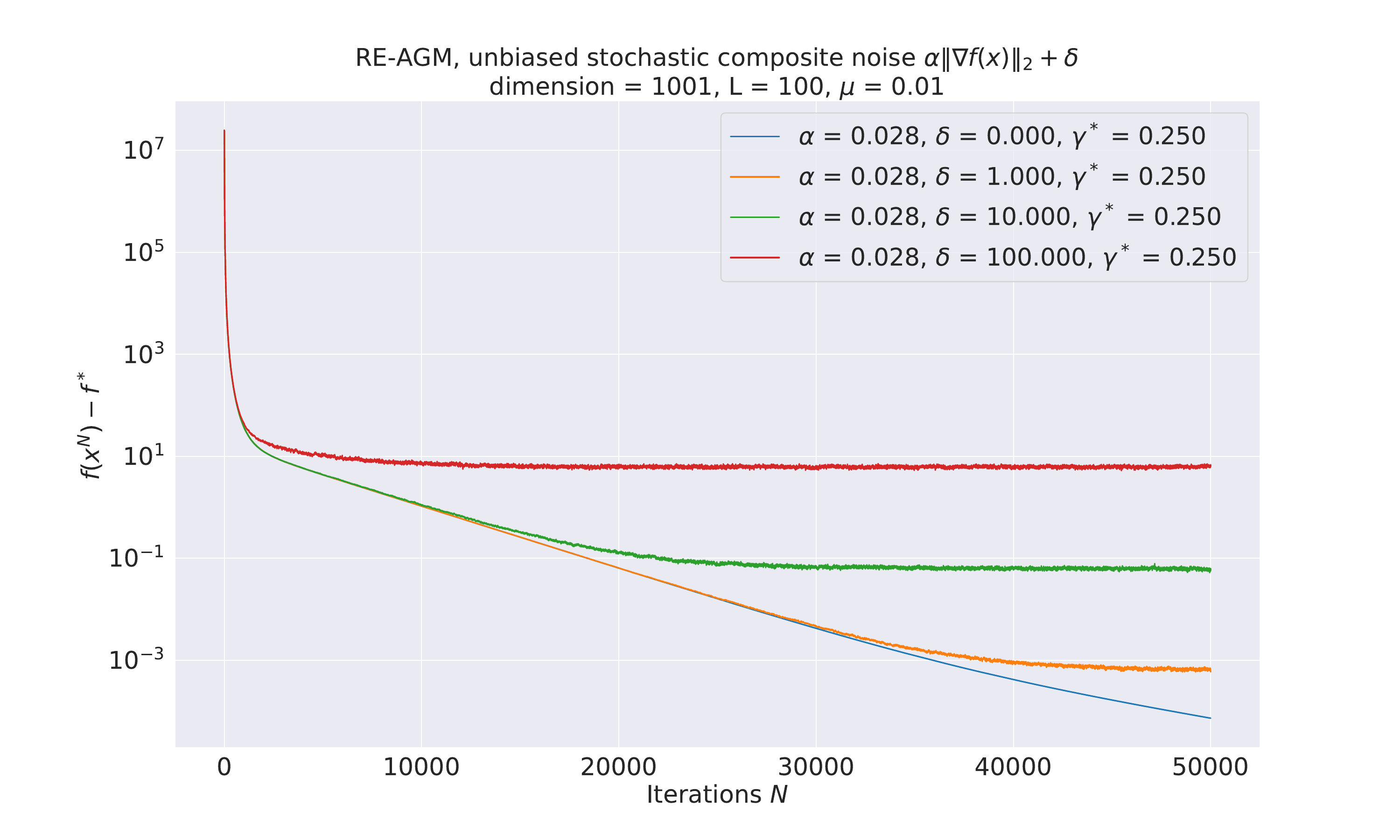}
    	\caption{The performance of RE-AGM with composite noise and $\mu = 0.01, L = 100$ for     different values of $\delta$ and $\alpha = 0.028$.}
	\end{center}
\end{figure}

Plot above demonstrates, that increasing $\delta$ value $10$ times increases limit of convergence $100$ times.

Then we can compare Algorithm~\ref{alg gd} and Algorithm~\ref{alg re-agm}. We should remind, that for $\alpha = 1/3$ convergence of Algorithm~\ref{alg re-agm} corresponds to convergence of Algorithm~\ref{alg gd}.

\begin{figure}[H]
	\begin{center}
		\includegraphics[width=1\linewidth]{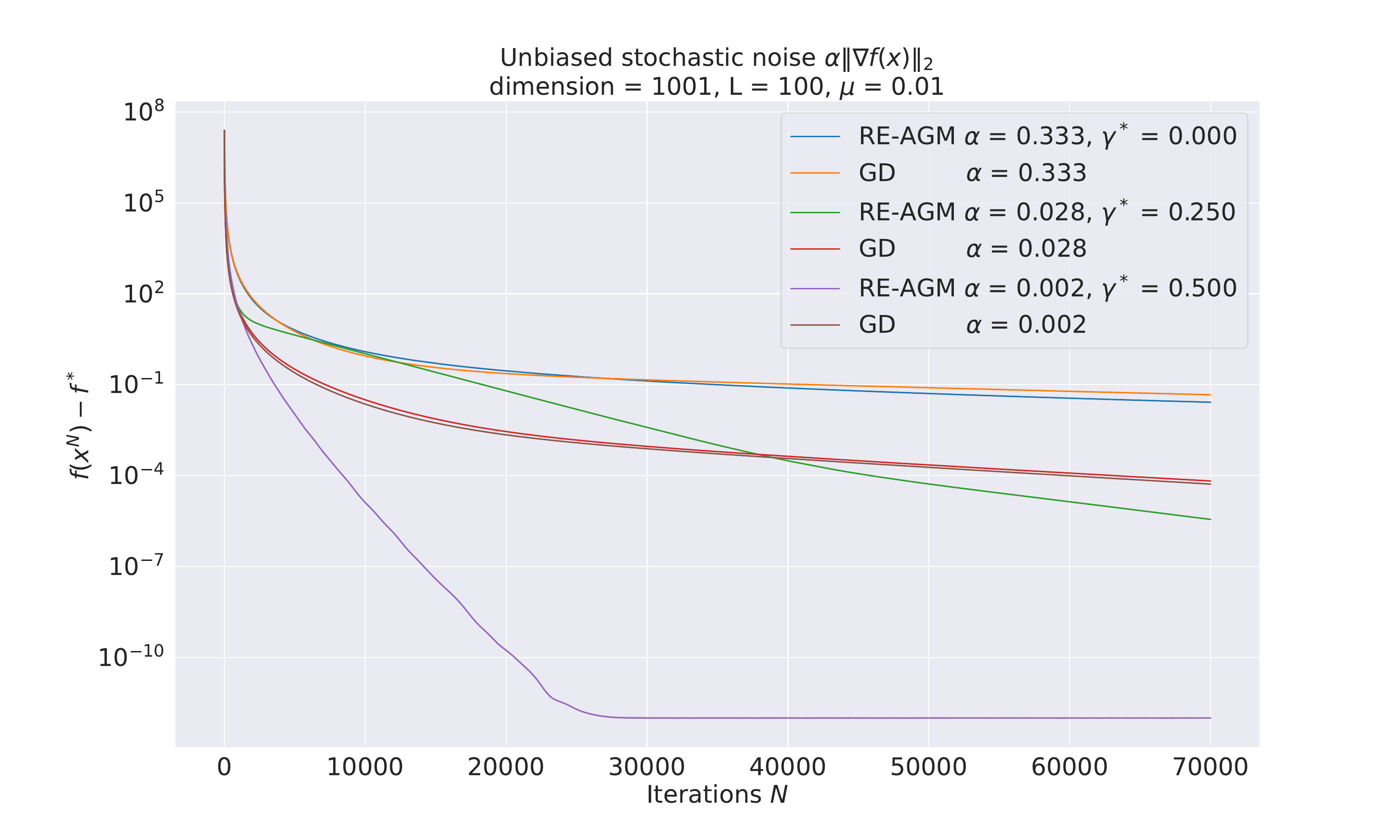}
    	\caption{The performance GD vs RE-AGM with relative noise and $\mu = 0.01, L = 100$       for  different values of $\alpha$.}
	\end{center}
\end{figure}

This comparison shows that the case $\alpha \sim \left(\mu / L \right)^{1/2}$ gives significant difference between the accelerated and non-accelerated method, case $\alpha = 1/3$ confirms their equal convergence rate and case $\alpha \sim \left(\mu / L \right)^{1/4}$ provides intermediate convergence which is better, than $GD$ and worse than accelerated methods lower bound for noiseless case ($\alpha = 0$).

\bibliographystyle{apalike}
\bibliography{literature}
\begin{appendix}

\section{Miscellaneous Used Lemmas}

\begin{lemma} \label{fenchel}
    If $u, v \in \mathbb{R}^n$, then $\forall \lambda > 0$:
    \begin{equation*}
        \langle u, v \rangle \leqslant \frac{\lambda}{2} \|u \|_2^2 + \frac{1}{2 \lambda} \|v \|_2^2.
    \end{equation*}
\end{lemma}

\begin{lemma} \label{(a - b)2 lb}
    If $a, b \geqslant 0$, then
    \begin{equation*}
        (a - b)^2 \geqslant \frac{1}{2} a^2 - b^2.
    \end{equation*}
\end{lemma}
\begin{proof}
    Using Fenchel~\ref{fenchel} inequality with $\lambda = \frac{1}{2}$.
\end{proof}

\begin{lemma} \label{basic alpha conditions}
    If $\widetilde{\nabla} f$ satisfies error conditions~\eqref{noise condition}. Following inequalities take place:
    \begin{equation*}
        \begin{gathered}
            (1 - \alpha) \|\nabla f(x) \|_2 - \delta \leqslant \|\widetilde{\nabla} f(x) \|_2 \leqslant (1 + \alpha) \|\nabla f(x) \|_2 + \delta, \\
            \frac{1}{1 + \alpha} \left( \| \widetilde{\nabla} f(x) \|_2 - \delta \right) \leqslant \| \nabla f(x) \|_2 \leqslant \frac{1}{1 - \alpha} \left( \| \widetilde{\nabla} f(x) \|_2 + \delta \right), \\
            \frac{1}{2} (1 - \alpha)^2 \|\nabla f(x) \|_2^2 - \delta^2 \leqslant \|\wtgg f(x) \|_2^2 \leqslant 2 (1 + \alpha)^2 \|\nabla f(x) \|_2^2 + 2 \delta^2, \\
            \frac{1}{2(1 + \alpha)^2} \| \widetilde{\nabla} f(x) \|_2^2 - \frac{1}{(1 + \alpha)^2} \delta^2 \leqslant \|\nabla f(x) \|_2^2 \leqslant \frac{2}{(1 - \alpha)^2} \|\wtgg f(x) \|_2^2 + \frac{2}{(1 - \alpha)^2} \delta^2.
        \end{gathered}
    \end{equation*}
\end{lemma}
\begin{proof}
    Follows from definition~\eqref{noise condition}, Lemma~\ref{fenchel} and Lemma~\ref{(a - b)2 lb}.
\end{proof}

\begin{lemma} \label{PL noised}
    If $f$ satisfies~\eqref{PL cond}, $\widetilde{\nabla} f$ satisfies error condition~\eqref{noise condition}. Then a similar condition takes place:
    \begin{equation*}
            \|\wtgg f(x) \|_2^2 \geqslant (1 - \alpha)^2 \mu (f(x) - f^*) - \delta^2
    \end{equation*}
\end{lemma}
\begin{proof}
    Using P{\L} condition and Lemma~\ref{basic alpha conditions}:
    \begin{equation*}
        \|\wtgg f(x) \|_2^2 \geqslant \frac{(1 - \alpha)^2}{2} \| \nabla f(x) \|_2^2 - \delta^2 \geqslant \mu (1 - \alpha)^2 (f(x) - f^*) - \delta^2.
    \end{equation*}
\end{proof}

\begin{lemma} \label{noise decompose}
    If $\zeta(x) = \widetilde{\nabla} f(x) - \nabla f(x)$ satisfies error condition~\ref{noise condition}, then exist two vector fields $\zeta_a, \zeta_r$:
    \begin{equation*}
        \zeta(x) = \zeta_a(x) + \zeta_r(x),
    \end{equation*}
    which satisfy:
    \begin{equation*}
        \| \zeta_a(x) \| \leqslant \delta, \hspace{5mm} \| \zeta_r(x) \|_2 \leqslant \alpha \| \nabla f(x) \|_2.
    \end{equation*}
\end{lemma}
\begin{proof}
    Let us define two sets:
    \begin{equation*}
        \begin{gathered}
            S_1 = \lbrace \nabla f(x) + v \; , \; \| v \|_2 \leqslant \alpha \| \nabla f(x) \|_2 \rbrace, \\
            S_2 = \lbrace \widetilde{\nabla} f(x) + v \; , \; \| v \|_2 \leqslant \delta \rbrace.
        \end{gathered}
    \end{equation*}
    Then from condition $\| \widetilde{\nabla} f(x) - \nabla f(x) \|_2 \leqslant \alpha \| \nabla f(x) \|_2 + \delta$ we can draw a conclusion about $S_1 \cap S_2 \not= \emptyset $ and the existence of two vectors $\zeta_r, \zeta_a$, described in theorem condition.
\end{proof}

\begin{lemma} \label{cos lb}

    1. If $\zeta_r(x)$ satisfies relative error condition~\eqref{relative noise condition}:
    \begin{equation*}
        \| \zeta_r(x) \|_2 \leqslant \alpha \| \nabla f(x) \|_2,        
    \end{equation*}
    then:
    \begin{eqnarray*}
        \langle \nabla f(x) + \zeta_r(x), \nabla f(x) \rangle
        & \geqslant & \sqrt{1 - \alpha^2} \| \nabla f(x) + \zeta_r(x) \|_2 \| \nabla f(x) \|_2 \\
        & \geqslant & \sqrt{1 - \alpha^2} (1 - \alpha) \| \nabla f(x) \|_2^2.
    \end{eqnarray*}

    2. If $\wtgg f$ satisfies composite noise~\eqref{noise condition}, then
    \begin{eqnarray*}
        \langle \nabla f(x), \wtgg f(x) \rangle
        & \geqslant & \frac{1}{4} \sqrt{\frac{1 - \alpha}{1 + \alpha}} \| \wtgg f(x) \|_2^2 
        - \frac{3}{\sqrt{(1 + \alpha)^3(1 - \alpha)}} \delta^2.
    \end{eqnarray*}
\end{lemma}
\begin{proof}
    $\\$ 

    1. Let's consider the plane generated by vectors $\nabla f(x) + \zeta_r(x)$ and $\nabla f(x)$.  One can note that maximum angle between this two vectors will be maximum, when $\nabla f(x) + \zeta_r(x)$ is tangent to the circle with radius $\alpha \| \nabla f(x) \|_2$. That is maximum sinus of that angle is $\alpha$. Using Lemma~\ref{basic alpha conditions} we get second inequality.
    $\\$

    2. 
    \begin{eqnarray*}
        \langle \nabla f(x), \wtgg f(x) \rangle
        & \geqslant & \langle \nabla f(x), \zeta_a(x) \rangle + \langle \nabla f(x), \nabla f(x) + \zeta_r(x) \rangle
        \\
        & \overset{\ref{fenchel}, \text{ previous clause}}{\geqslant} & 
        - \frac{\lambda}{2} \|\nabla f(x) \|_2^2 - \frac{2}{\lambda} \delta^2 
        \\
        & + & \sqrt{1 - \alpha^2} \| \nabla f(x) + \zeta_r(x) \|_2 \| \nabla f(x) \|_2
        \\
        & \overset{\ref{basic alpha conditions}}{\geqslant} & 
        \frac{\lambda}{2} \|\nabla f(x) \|_2^2 + \frac{2}{\lambda} \delta^2 + 
        \sqrt{\frac{1 - \alpha}{1 + \alpha}} \| \nabla f(x) + \zeta_r(x) \|_2^2
        \\
        & \geqslant & 
        - \frac{\lambda}{2} \|\nabla f(x) \|_2^2 - \frac{2}{\lambda} \delta^2 + 
        \sqrt{\frac{1 - \alpha}{1 + \alpha}} \| \wtgg f(x) - \zeta_a(x) \|_2^2
        \\
        & \overset{\ref{(a - b)2 lb}}{\geqslant} & 
        - \frac{\lambda}{2} \|\nabla f(x) \|_2^2 - \frac{2}{\lambda} \delta^2
        \\
        & + & 
        \frac{1}{2} \sqrt{\frac{1 - \alpha}{1 + \alpha}} \| \wtgg f(x) \|_2^2 - \sqrt{\frac{1 - \alpha}{1 + \alpha}} \| \zeta_a(x) \|_2^2
        \\
        & \overset{\ref{basic alpha conditions}}{\geqslant} &
        - \frac{\lambda}{2} \left( \frac{1}{2(1 + \alpha)^2} \| \widetilde{\nabla} f(x) \|_2^2 - \frac{1}{(1 + \alpha)^2} \delta^2 \right) - \frac{2}{\lambda} \delta^2
        \\
        & + & \frac{1}{2} \sqrt{\frac{1 - \alpha}{1 + \alpha}} \| \wtgg f(x) \|_2^2 - \sqrt{\frac{1 - \alpha}{1 + \alpha}} \delta^2
        \\
        & \overset{\lambda = \sqrt{(1 + \alpha)^3(1 - \alpha)}}{\geqslant} & \frac{1}{4} \sqrt{\frac{1 - \alpha}{1 + \alpha}} \| \wtgg f(x) \|_2^2 
        \\
        & - & \sqrt{\frac{1 - \alpha}{1 + \alpha}} \delta^2 - \frac{2}{\sqrt{(1 + \alpha)^3(1 - \alpha)}} \delta^2
        \\
        & \overset{\sqrt{\frac{1 - \alpha}{1 + \alpha}} \leqslant \frac{1}{\sqrt{(1 + \alpha)^3(1 - \alpha)}}}{\geqslant} & 
        \frac{1}{4} \sqrt{\frac{1 - \alpha}{1 + \alpha}} \| \wtgg f(x) \|_2^2 
        - \frac{3}{\sqrt{(1 + \alpha)^3(1 - \alpha)}} \delta^2.
    \end{eqnarray*}
\end{proof}

\begin{lemma} \label{pl dist}
    Let $f$ satisfies strong convexity~\eqref{strong conv}, then:
    \begin{equation*}
        \|x - x^* \|_2 \leqslant \frac{1}{\mu} \| \nabla f(x) \|_2.
    \end{equation*}
\end{lemma}
\begin{proof}
    From strong convexity and P$\L$ condition:
    \begin{equation*}
        \begin{gathered}    
            \frac{\mu}{2} \|x - x^* \|_2^2 \leqslant f(x) - f^* \leqslant \frac{1}{2\mu} \|\nabla f(x) \|_2^2, \\
            \|x - x^* \|_2 \leqslant \frac{1}{\mu} \|\nabla f(x) \|_2.
        \end{gathered}
    \end{equation*}
\end{proof}

\section{Gradient Descent missing proofs}

\begin{lemma} \label{GD step}
    Let $f$ smooth~\eqref{smooth cond}, $\widetilde{\nabla} f$ satisfies noise condition~\eqref{noise condition},
    \begin{equation*}
        f(x^{k + 1}) \leqslant f(x^k) - \frac{(1 - \alpha)^3}{(1 + \alpha)} \frac{1}{16 L} \|\nabla f(x^k) \|_2^2 + \frac{3}{16 L} \frac{1}{(1 + \alpha)^2} \delta^2.
    \end{equation*}
    Where:
    \begin{equation} \label{gd const def}
        \begin{gathered}
            h = \left( \frac{1 - \alpha}{ 1 + \alpha} \right)^{3 / 2} \frac{1}{4 L}.
        \end{gathered}
    \end{equation}
\end{lemma}
\begin{proof}
    Using smoothness of function $f$ we get:
    \begin{equation*}
        f(x^{k + 1}) \leqslant f(x^k) + \langle \nabla f(x^k), x^{k + 1} - x^k \rangle + \frac{L}{2} \| x^{k + 1} - x^k \|_2^2.
    \end{equation*}
    Let us estimate terms.
    Linear form:
    \begin{align*}
        \langle \nabla f(x^k), x^{k + 1} - x^k \rangle 
        & = -h \langle \nabla f(x^k), \wtgg f(x^k) \rangle \\
        & \overset{\eqref{noise def abs rel}}{=} -h \bigg (
        \langle \nabla f(x^k), \nabla f(x^k) + \zeta_r(x^k) \rangle + \langle \nabla f(x^k), \zeta_a(x^k) \rangle \bigg ) \\
        & \hspace{-10mm} \overset{\eqref{fenchel}, \eqref{cos lb}}{\leqslant}
        -h \bigg( \sqrt{1 - \alpha^2} (1 - \alpha) \|\nabla f(x^k) \|_2^2 - 
        \frac{\lambda}{2} \|\nabla f(x^k) \|_2^2 - \frac{1}{2 \lambda} \delta^2 \bigg) \\
        & \hspace{-12mm} \overset{\lambda = \sqrt{1 - \alpha^2}(1 - \alpha)}{\leqslant} -h \bigg( \frac{\sqrt{1 - \alpha^2}(1 - \alpha)}{2} \| \nabla f(x^k) \|_2^2  -
        \frac{1}{2 \sqrt{1 - \alpha^2}(1 - \alpha)} \delta^2 \bigg).
    \end{align*}
    Quadratic term:
    \begin{eqnarray*}
        \| x^{k + 1} - x^k \|_2^2 
        & = & h^2 \| \widetilde{\nabla} f(x^k) \|_2^2 \\
        & \leqslant & h^2 \|(\nabla f(x^k) + \zeta_r(x^k)) + \zeta_a(x^k) \|_2^2 \\
        & \overset{\eqref{fenchel}, \eqref{basic alpha conditions}}{\leqslant} & 2 h^2 \left((1 + \alpha)^2 \|\nabla f(x^k) \|_2^2 + \delta^2 \right). 
    \end{eqnarray*}
    Combining this inequalities:
    \begin{eqnarray}
        f(x^{k + 1}) 
        & \leqslant & f(x^k) \nonumber
        - h \left( \frac{\sqrt{1 - \alpha^2}(1 - \alpha)}{2} \| \nabla f(x^k) \|_2^2 - \frac{1}{2 \sqrt{1 - \alpha^2}(1 - \alpha)} \delta^2 \right) \\  \nonumber
        & + & L h^2 \left((1 + \alpha)^2 \|\nabla f(x^k) \|_2^2 + \delta^2 \right) \\ \nonumber
        & = & f(x^k) + h \|\nabla f(x^k) \|_2^2 \left(L h (1 + \alpha)^2 - \frac{\sqrt{1 - \alpha^2}(1 - \alpha)}{2} \right) \\ \nonumber
        & + & h \delta^2 \left( \frac{1}{2 \sqrt{1 - \alpha^2}(1 - \alpha)} + h L \right) \\ \label{h gd step ineq}
        & \overset{h \text{ definition}}{=} & f(x^k) - \left(\frac{1 - \alpha}{1 + \alpha} \right)^{3 / 2} \frac{1}{4 L} \frac{\sqrt{1 - \alpha^2}(1 - \alpha)}{4} \|\nabla f(x^k) \|_2^2 \\ \nonumber
        & + &  \left(\frac{1 - \alpha}{1 + \alpha} \right)^{3 / 2} \frac{1}{4 L} \delta^2 \left( \frac{1}{2 \sqrt{1 - \alpha^2}(1 - \alpha)} + \frac{1}{4} \left(\frac{1 - \alpha}{1 + \alpha} \right)^{3 / 2} \right) \\ \nonumber
        & = & f(x^k) - \frac{(1 - \alpha)^3}{(1 + \alpha)} \frac{1}{16 L} \|\nabla f(x^k) \|_2^2 \\
        & + & \frac{\delta^2}{8 L} \left(\frac{1}{(1 + \alpha)^2} + \frac{(1 - \alpha)^3}{2 (1 + \alpha)^3} \right) \nonumber
        \\
        & \overset{(1 - \alpha)^3/(1 + \alpha) \leqslant 1}{\leqslant} & f(x^k) - \frac{(1 - \alpha)^3}{(1 + \alpha)} \frac{1}{16 L} \|\nabla f(x^k) \|_2^2 \\
        & + & \frac{3}{16 L} \frac{1}{(1 + \alpha)^2} \delta^2. \nonumber
    \end{eqnarray}
\end{proof}

\begin{theorem}[Theorem~\ref{gd norm conv text}] \label{gd norm conv}
    Let $f$ smooth~\eqref{smooth cond},  $\widetilde{f}$ satisfies error condition~\eqref{noise condition}. Then Algorithm~\ref{alg gd} with step:
    \begin{equation*}
        h = \left( \frac{1 - \alpha}{ 1 + \alpha} \right)^{3 / 2} \frac{1}{4 L},
    \end{equation*}
    guarantees convergence:
    \begin{eqnarray*}
            \min_{k \leqslant N} \|\nabla f(x^k) \|_2^2
            & \leqslant & \frac{(1 + \alpha)}{(1 - \alpha)^3} \frac{16L (f(x^0) - f^*)}{N + 1} + \frac{3}{(1 - \alpha)^3(1 + \alpha)} \delta^2.
    \end{eqnarray*}
\end{theorem}
\begin{proof}
    From Lemma~\ref{GD step} we obtain inequality:
    \begin{equation*}
         \frac{(1 - \alpha)^3}{(1 + \alpha)} \frac{1}{16 L} \|\nabla f(x^k) \|_2^2 \leqslant f(x^k) - f(x^{k + 1}) + \frac{3}{16 L} \frac{1}{(1 + \alpha)^2} \delta^2.
    \end{equation*}
    Telescope that inequality from $0$ to $N$:
    \begin{equation*}
        \frac{(1 - \alpha)^3}{(1 + \alpha)} \frac{1}{16 L} \sum_{k = 0}^{N} \|\nabla f(x^k) \|_2^2 \leqslant f(x^0) - f(x^{N + 1}) + \frac{3(N + 1)}{16 L} \frac{1}{(1 + \alpha)^2} \delta^2.
    \end{equation*}
    Using $\|\nabla f(x^k) \|_2^2 \geqslant \min\limits_{k \leqslant N} \|\nabla f(x^k) \|_2^2$ and $f(x^{N + 1}) \geqslant f^*$ inequality above can be simplified:
    \begin{eqnarray*}
        \frac{(1 - \alpha)^3}{(1 + \alpha)} \frac{N + 1}{16 L} \min_{k \leqslant N} \|\nabla f(x^k) \|_2^2
        & \leqslant & f(x^0) - f(x^*) + \frac{3(N + 1)}{16 L} \frac{1}{(1 + \alpha)^2} \delta^2, \\
        \min_{k \leqslant N} \|\nabla f(x^k) \|_2^2
        & \leqslant & \frac{(1 + \alpha)}{(1 - \alpha)^3} \frac{16L (f(x^0) - f^*)}{N + 1} \\
        & + & \frac{3}{(1 - \alpha)^3(1 + \alpha)} \delta^2 .
    \end{eqnarray*}
\end{proof}

\begin{theorem}[Theorem~\ref{gd pl conv text}] \label{gd pl conv}
    Let $f$ function satisfies P$\L$ condition~\eqref{PL cond} and smoothness~\eqref{smooth cond}, $\wtgg f$ satisfies~\eqref{noise condition}. Then Algorithm~\ref{alg gd} with step:
    \begin{equation*}
        h = \left( \frac{1 - \alpha}{ 1 + \alpha} \right)^{3 / 2} \frac{1}{4 L},
    \end{equation*}
    guarantees convergence:
    \begin{eqnarray*}
        f(x^N) - f^* 
        & \leqslant &  \left(1 -  \frac{(1 - \alpha)^3}{1 + \alpha} \frac{\mu}{8 L} \right)^N (f(x^0) - f^*) + \frac{3}{2 \mu} \frac{1 + \alpha}{(1 - \alpha)^3} \delta^2.
    \end{eqnarray*}
\end{theorem}
\begin{proof}
    \begin{eqnarray*}
        f(x^{k + 1}) - f^* 
        & \overset{\text{Lemma}~\ref{GD step}}{\leqslant} & f(x^k) - f^* - \frac{(1 - \alpha)^3}{(1 + \alpha)} \frac{1}{16 L} \|\nabla f(x^k) \|_2^2
        +  \frac{3}{16 L} \frac{1}{(1 + \alpha)^2} \delta^2 \\
        & \overset{\eqref{PL cond}}{\leqslant} & f(x^k) - f^* - \frac{(1 - \alpha)^3}{(1 + \alpha)} \frac{\mu}{8 L} (f(x^k) - f^*)
        + \frac{3}{16 L} \frac{1}{(1 + \alpha)^2} \delta^2 \\
        & = & \left(1 -  \frac{(1 - \alpha)^3}{(1 + \alpha)} \frac{\mu}{8 L} \right) (f(x^k) - f^*)
        + \frac{3}{16 L} \frac{1}{(1 + \alpha)^2} \delta^2
    \end{eqnarray*}
    Iterating inequality above:
    \begin{eqnarray*}
        f(x^N) - f^*
        & \leqslant & \left(1 -  \frac{(1 - \alpha)^3}{1 + \alpha} \frac{\mu}{8 L} \right)^N (f(x^0) - f^*) \\
        & + & \frac{3}{16 L} \frac{\delta^2}{(1 + \alpha)^2} \sum_{k = 0}^{N - 1} \left(1 -  \frac{(1 - \alpha)^3}{(1 + \alpha)} \frac{\mu}{8 L} \right)^k \\
        & \leqslant &  \left(1 -  \frac{(1 - \alpha)^3}{1 + \alpha} \frac{\mu}{8 L} \right)^N (f(x^0) - f^*) \\
        & + & \frac{3}{16 L} \frac{\delta^2}{(1 + \alpha)^2} \frac{8 L}{\mu} \frac{1 + \alpha}{(1 - \alpha)^3} \\
        & \leqslant &  \left(1 -  \frac{(1 - \alpha)^3}{1 + \alpha} \frac{\mu}{8 L} \right)^N (f(x^0) - f^*) \\
        & + & \frac{3}{2 \mu} \frac{1 + \alpha}{(1 - \alpha)^3} \delta^2.
    \end{eqnarray*}
\end{proof}

\section{Accelerated method missing proofs} \label{appendix accelerated}

Following \citep[p. 83]{nesterov2018lectures}, let us introduce a parameterized set of functions $\Psi$, its element defined for $c \in \mathbb{R}, \kappa \in \mathbb{R}^{++},$ and $u \in \mathbb{R}^n$,  as follows
\begin{equation} \label{qudaratic family}
    \psi(x | c, \kappa, u) = c + \frac{\kappa}{2} \| x - u \|_2^2, \quad \forall x \in \mathbb{R}^n. 
\end{equation}

According to~\cite{nesterov2018lectures}, we can mention to the following useful property of the class $\Psi$.

\begin{lemma} \label{psi comb}
Let $\psi_1, \psi_2 \in \Psi$. Then $\forall \eta_1, \eta_2, c_1, c_2 \in \mathbb{R}, \forall \kappa_1, \kappa_2 \in \mathbb{R}^{++}$, and $\forall u, v \in \mathbb{R}^n$, we have 
\[
    \eta_1 \psi_1(x|c_1, \kappa_1, u) + \eta_2 \psi_2(x|c_2, \kappa_2, v) = \psi_3(x|c_3, \kappa_3, w), 
\]
where 
\begin{equation*}
    \begin{gathered}
        c_3 = \eta_1 c_1 + \eta_2 c_2 + \frac{\eta_1\eta_2\kappa_1\kappa_2}{2(\eta_1\kappa_1 + \eta_2\kappa_2)} \|u  - v \|_2^2, \\
        \kappa_3 = \eta_1\kappa_1 + \eta_2\kappa_2, \quad w = \frac{\eta_1\kappa_1 u +\eta_2\kappa_2 v}{\eta_1\kappa_1 + \eta_2\kappa_2}.
    \end{gathered}
\end{equation*}
\end{lemma}

Also from~\cite{nesterov2018lectures} we mention the following simple lemma.
\begin{lemma} \label{conv lemma}
Let $0 < \lambda < 1$, $A > 0$, $\psi_0, \psi \in \Psi$, $z \in \mathbb{R}^n$ such that 
\begin{equation*}
f(z) \leqslant \min_{x \in \mathbb{R}^n} \psi(x) + A,
\end{equation*}
and 
\begin{equation*}
     \psi(x) \leqslant \lambda \psi_0(x) + (1 - \lambda) f(x), \quad \forall x \in \mathbb{R}^n.
\end{equation*}
Then
\begin{equation*}
    f(z) - f^* \leqslant \lambda (\psi_0(x^*) - f^*) + A. 
\end{equation*}
\end{lemma}
\begin{proof}
    \begin{eqnarray*}
        f(z) - f^*
        & \leqslant & \min_{x \in \mathbb{R}^n} \psi(x) + A - f^*
        \\
        & \leqslant & \psi(x^*) - f^* + A \leqslant \lambda \psi_0(x^*) + (1 - \lambda) f(x^*) - f^* + A
        \\
        & = & \lambda ( \psi_0(x^*) - f^* ) + A. 
    \end{eqnarray*}
\end{proof}

\begin{lemma} \label{fenchel for accelerated}
    Let $\widetilde{\nabla} f$ satisfies error condition~\eqref{noise condition}, $0 < \mu \leqslant L$, then:
    \begin{equation*}
        \begin{gathered}
            \forall \gamma \geqslant 0 \quad \|\wtgg f(x) \|_2^2 \geqslant \left(1 - \frac{1}{4} \left(\frac{\mu}{2L} \right)^{\gamma} \right) (1 - \alpha)^2 \|\nabla f(x) \|_2^2 - \left(4\left(\frac{2L}{\mu} \right)^{\gamma} - 1 \right) \delta^2, \\
            \|\wtgg f(x) \|_2^2 \leqslant \left(1 + \frac{1}{4} \left(\frac{\mu}{2L} \right)^{\gamma} \right) (1 + \alpha)^2 \|\nabla f(x) \|_2^2 + \left(4\left(\frac{2L}{\mu} \right)^{\gamma} + 1 \right) \delta^2.
        \end{gathered}
    \end{equation*}
\end{lemma}
\begin{proof}
    Let us denote $g_r(x) = \nabla f(x) + \zeta_r(x)$. Then one can apply Lemma~\ref{fenchel} for $g_r, \zeta_a$:
    \begin{eqnarray*}
        \|\wtgg f(x) \|_2^2 & = & \|g_r(x) + \zeta_a(x) \|_2^2 = \|g_r(x) \|_2^2 + 2 \langle g_r(x), \zeta_a(x) \rangle + \| \zeta_a(x) \|_2^2 \\
        & \overset{\ref{fenchel}}{\leqslant} & \left(1 + \frac{1}{\lambda} \right) \| g_r(x) \|_2^2 + \left(1 + \lambda \right) \| \zeta_a(x) \|_2^2 \\
        & \overset{\ref{noise def abs rel}, \; \lambda = 4 \left(2L / \mu \right)^{\gamma}}{\leqslant} &
        \left(1 + \frac{1}{4} \left( \frac{\mu}{2L} \right) ^{\gamma} \right) (1 + \alpha)^2 \|\nabla f(x) \|_2^2 + \delta^2 \left(1 + 4\left(\frac{2L}{\mu}\right)^{\gamma} \right).
    \end{eqnarray*}
    Similarly for the lower bound.
\end{proof}

\begin{lemma} \label{accelerated quadratic estimation}
    Let $\wtgg f$ satisfies error condition~\eqref{noise condition}, $0 < \mu \leqslant L, 0 < \gamma$, then:
    \begin{eqnarray*}
        \frac{\mu}{4} \left\|x - z + \frac{2}{\mu} \widetilde{\nabla} f(z) \right \|_2^2
        & \leqslant & \langle \nabla f(z), x - z \rangle + \frac{\mu}{2} \|x - z \|_2^2 \\
        & + & \frac{1}{\mu} \left( \left(1 + \frac{1}{4} \left(\frac{\mu}{2L} \right)^{\gamma} \right) (1 + \alpha)^2 + 2\alpha^2 \right) \|\nabla f(z) \|_2^2 \\
        & + & \frac{1}{\mu} \left(4\left(\frac{2L}{\mu} \right)^{\gamma} + 3 \right) \delta^2, \\
        \frac{\mu}{4} \left\|x - z + \frac{2}{\mu} \widetilde{\nabla} f(z) \right \|_2^2
        & \geqslant & \langle \nabla f(z), x - z \rangle + \frac{\mu}{2} \|x - z \|_2^2 \\
        & + & \frac{1}{\mu} \left( \left(1 - \frac{1}{4} \left(\frac{\mu}{2L} \right)^{\gamma} \right) (1 - \alpha)^2 - 2\alpha^2 \right) \|\nabla f(z) \|_2^2 \\
        & - & \frac{1}{\mu} \left(4\left(\frac{2L}{\mu} \right)^{\gamma} + 1 \right) \delta^2.
    \end{eqnarray*}
\end{lemma}
\begin{proof}
We will find bounds for each term of sum:
\begin{equation*}
    \frac{\mu}{4} \left\|x - z + \frac{2}{\mu} \widetilde{\nabla} f(z) \right \|_2^2 = \frac{\mu}{4} \|x - z \|_2^2 + \langle \wtgg f(z), x - z \rangle + \frac{1}{\mu} \|\wtgg f(z) \|_2^2.
\end{equation*}
Linear form:
\begin{eqnarray*}
    \langle \wtgg f(z), x - z \rangle
    & \overset{\text{Lemma }\ref{fenchel}}{\leqslant} & \langle \nabla f(z), x - z \rangle + \frac{1}{\mu} \left\| \zeta_a(z) + \zeta_r(z) \right \|_2^2 + \frac{\mu}{4} \|x - z \|_2^2 \\
    & \leqslant &  \langle \nabla f(z), x - z \rangle + \frac{1}{\mu} \left(2\alpha^2 \|\nabla f(z) \|_2^2 + 2\delta^2 \right) + \frac{\mu}{4} \|x - z \|_2^2.
\end{eqnarray*}
\begin{eqnarray*}
    \langle \wtgg f(z), x - z \rangle
    & \overset{\text{Lemma }\ref{fenchel}}{\geqslant} & \langle \nabla f(z), x - z \rangle - \frac{1}{\mu} \| \zeta_a(z) + \zeta_r(z) \|_2^2 - \frac{\mu}{4} \|x - z \|_2^2 \\
    & \geqslant &  \langle \nabla f(z), x - z \rangle - \frac{1}{\mu} \left(2\alpha^2 \|\nabla f(z) \|_2^2 + 2\delta^2 \right) - \frac{\mu}{4} \|x - z \|_2^2.
\end{eqnarray*}
Quadratic term:
\begin{equation*}
    \|\wtgg f(z) \|_2^2 \overset{\text{Lemma }\ref{fenchel for accelerated}}{\leqslant} \left(1 + \frac{1}{4} \left(\frac{\mu}{2L} \right)^{\gamma} \right) (1 + \alpha)^2 \|\nabla f(z) \|_2^2 + \left(4\left(\frac{2L}{\mu} \right)^{\gamma} + 1 \right) \delta^2.
\end{equation*}
\begin{equation*}
    \|\wtgg f(z) \|_2^2 \overset{\text{Lemma }\ref{fenchel for accelerated}}{\geqslant} \left(1 - \frac{1}{4} \left(\frac{\mu}{2L} \right)^{\gamma} \right) (1 - \alpha)^2 \|\nabla f(z) \|_2^2 - \left(4\left(\frac{2L}{\mu} \right)^{\gamma} - 1 \right) \delta^2.
\end{equation*}
Then we can sum up estimations above:
\begin{eqnarray*}
    \frac{\mu}{4} \left\|x - z + \frac{2}{\mu} \widetilde{\nabla} f(z) \right \|_2^2
    & \leqslant & \langle \nabla f(z), x - z \rangle + \frac{\mu}{2} \|x - z \|_2^2 \\
    & + & \frac{1}{\mu} \left(2\alpha^2 \|\nabla f(z) \|_2^2 + 2\delta^2 \right) \\
    & + & \frac{1}{\mu} \left(1 + \frac{1}{4} \left(\frac{\mu}{2L} \right)^{\gamma} \right) (1 + \alpha)^2 \|\nabla f(z) \|_2^2 \\
    & + & \frac{1}{\mu} \left(4\left(\frac{2L}{\mu} \right)^{\gamma} + 1 \right) \delta^2 \\
    & \overset{\eqref{strong conv}}{\leqslant} & \langle \nabla f(z), x - z \rangle + \frac{\mu}{2} \|x - z \|_2^2 \\
    & + & \frac{1}{\mu} \left( \left(1 + \frac{1}{4} \left(\frac{\mu}{2L} \right)^{\gamma} \right) (1 + \alpha)^2 + 2\alpha^2 \right) \|\nabla f(z) \|_2^2 \\
    & + & \frac{1}{\mu} \left(4\left(\frac{2L}{\mu} \right)^{\gamma} + 3 \right) \delta^2.
\end{eqnarray*}
\begin{eqnarray*}
    \frac{\mu}{4} \left\|x - z + \frac{2}{\mu} \widetilde{\nabla} f(z) \right \|_2^2
    & \geqslant & \langle \nabla f(z), x - z \rangle + \frac{\mu}{2} \|x - z \|_2^2 \\
    & - & \frac{1}{\mu} \left(2\alpha^2 \|\nabla f(z) \|_2^2 + 2\delta^2 \right) \\
    & + & \frac{1}{\mu} \left(1 - \frac{1}{4} \left(\frac{\mu}{2L} \right)^{\gamma} \right) (1 - \alpha)^2 \|\nabla f(z) \|_2^2 \\
    & - & \frac{1}{\mu} \left(\left(\frac{2L}{\mu} \right)^{\gamma} - 1 \right) \delta^2 \\
    & \overset{\eqref{strong conv}}{\geqslant} & \langle \nabla f(z), x - z \rangle \\
    & + & \frac{1}{\mu} \left( \left(1 - \frac{1}{4} \left(\frac{\mu}{2L} \right)^{\gamma} \right) (1 - \alpha)^2 - 2\alpha^2 \right) \|\nabla f(z) \|_2^2 \\
    & - & \frac{1}{\mu} \left(4\left(\frac{2L}{\mu} \right)^{\gamma} + 1 \right) \delta^2.
\end{eqnarray*}

\end{proof}

Now, let us proof the following lemma which provides a lower bound of the function $f$ from the set $\Psi$. 

\begin{lemma} \label{lower bound}
Let $f$ be a $\mu$-strongly convex~\eqref{strong conv} function and $\widetilde{\nabla}{f}$ satisfies~\eqref{noise condition}. Then for any $z \in \mathbb{R}^n$ and $\gamma > 0$, there is $\phi_{z, \gamma} \in \Psi$, such that: 
\begin{equation*}
    \begin{gathered}
        \phi_{z, \gamma}(x) = c_{\phi} + \frac{\kappa_{\phi}}{2} \|x - u_{\phi} \|_2^2 \leqslant f(x), \quad \forall x \in \mathbb{R}^n. \\
        c_{\phi} := f(z) - \frac{1}{\mu} \left( \left(1 + \frac{1}{4} \left(\frac{\mu}{2L} \right)^{\gamma} \right) (1 + \alpha)^2 + 2\alpha^2 \right) \|\nabla f(z) \|_2^2 - \frac{1}{\mu} \left(4\left(\frac{2L}{\mu} \right)^{\gamma} + 3 \right) \delta^2, \\
        \kappa_{\phi} := \mu / 2, \\
        u_{\phi} := z - \frac{2}{\mu} \widetilde{\nabla}f(z).
    \end{gathered}
\end{equation*}
\end{lemma}

\begin{proof}
Proof follows directly from Lemma~\ref{accelerated quadratic estimation} and strong convexity definition~\eqref{strong conv}.
\end{proof}

Now, for Algorithm~\ref{alg re-agm}, we have $u^0 = x^0, x^0 \in \mathbb{R}^n$. Let $u^k$ be defined in Algorithm~\ref{alg re-agm} and $\omega$ defined in Algorithm~\ref{alg re-agm}. Let us define, recursively, a pair of sequences $\{\psi_k(x)\}_{k \geqslant 0}, \{\lambda_k\}_{k \geqslant 0}$ corresponding to Algorithm~\ref{alg re-agm} as follows:
\begin{align}\label{function seq}
    &  \lambda_0 = 1, \quad {c_0 = f(x^0)}, \quad \lambda_{k+1} = (1 - \omega) \lambda_k, \quad \forall k \geqslant0,
    \\& \psi_k(x| c_k, \mu / 2, u^k) = c_k + \frac{\mu}{4} \|x - u^k\|_2^2, \quad \forall k \geqslant 0.  \label{psi^k_Alg1}
\end{align}

Let $y^k\, \forall k \geqslant 0$ be generated by Algorithm~\ref{alg re-agm} (see item 10 of this algorithm), variable
\begin{equation} \label{gamma star remind}
    \gamma^* = \min \left \lbrace \log_{\mu / 2L} (3 \alpha ), \frac{1}{2} \right \rbrace.
\end{equation}

$\phi_{y^k, \gamma^*}$ be the function that is obtained from Lemma \ref{lower bound}. We define, recursively
\begin{align} \label{psi seq def}
    \psi_{k + 1}(x | c_{k + 1}, \mu / 2, u^{k + 1}) := (1 - \omega) \psi_k(x | c_k, \mu / 2, u^k) +  \omega \phi_{y^k, \gamma^*}(x), \quad \forall k \geqslant 0.
\end{align}
Using Lemma~\ref{psi comb} we obtain parameters of $\{ \psi_k \}$ that correspond to $ u^k$ in Algorithm~\ref{alg re-agm}. For simplicity, we use the values declared in the Algorithm~\ref{alg re-agm}:
\begin{equation} \label{s, m remind}
    \begin{gathered}
        s = \left(1 + \frac{1}{4} \left(\frac{\mu}{2L} \right)^{\gamma^*} \right) (1 + \alpha)^2 + 2\alpha^2, \\
        m = \left(1 - \frac{1}{4} \left(\frac{\mu}{2L} \right)^{\gamma^*} \right) (1 - \alpha)^2 - 2\alpha^2.
    \end{gathered}
\end{equation}
Also we define absolute noise accumulation sequence:
\begin{equation} \label{xi noise accumulation accelerated}
    \begin{gathered}
        \delta_0 = \left(\frac{3}{16L (1 + \alpha)^2} + \frac{4\omega}{\mu} \left(2\left(\frac{2L}{\mu} \right)^{\gamma^*} + 1 \right) \right) \delta^2, \\
        \xi_0 = 0, \\
        \xi_{k + 1} = (1 - \omega) \xi_k + \delta_0, \\
        \xi_{k} = \delta_0 \sum_{j = 1}^k (1 - \omega)^{j - 1} \leqslant \frac{\delta_0}{\omega}
    \end{gathered}
\end{equation}

\begin{lemma} \label{w solution estimation}
Let $\alpha \in [0, 1/3]$, variables $s, m$ defined above at~\eqref{s, m remind}, $\omega$ is the largest root of the quadratic equation
\begin{equation}\label{omega lemma def}
    m \omega^2 + (s - m) \omega - q = 0,
\end{equation}
where $q = \frac{\mu}{\widehat{L}}$ and $\widehat{L} = 8 \frac{1 + \alpha}{(1 - \alpha)^3} L$. Then:

\begin{equation*}
     \frac{1}{150} \left(\frac{\mu}{2L} \right)^{1 - \gamma^*} \leqslant \omega < 1,
\end{equation*}

where $\gamma^*$ is defined at~\eqref{gamma star remind}.

\end{lemma}

\begin{proof}
By solving the equation \eqref{omega lemma def} and choosing the largest root we get
\begin{equation} \label{omega quadsol}
    \omega = \frac{(m - s) + \sqrt{(s - m)^2 + 4mq }}{2 m}.
\end{equation}

Let us assume that $\omega \geqslant 1$. We have $\mu \leqslant L,$ thus $q \leqslant 1$. Also, we have $s > 1$. Therefore, 
\[
    m \omega^2 + (s - m) \omega -q  \geqslant m + s - m - q = s - q \geqslant 1 - q > 0,
\]
which is contradiction with \eqref{omega lemma def}. Hence $\omega < 1$.

We can proof, that $0 < m < 1$:

\begin{eqnarray*}
    m & = & \left(1 - \frac{1}{4} \left(\frac{\mu}{2L} \right)^{\gamma^*} \right) (1 - \alpha)^2 - 2\alpha^2 \\
    & \geqslant & \left(1 - \frac{1}{4} \right) (1 - \alpha)^2 - 2 \alpha^2
    \overset{\alpha \leqslant 1/3}{\geqslant} \frac{3}{4} \cdot \frac{4}{9} - \frac{2}{9} = \frac{1}{9} > 0, \\
    m & = & \left(1 - \frac{1}{4} \left(\frac{\mu}{2L} \right)^{\gamma^*} \right) (1 - \alpha)^2 - 2\alpha^2 \\
    & \overset{\gamma^* \leqslant 1/2}{\leqslant} & 1 - \frac{1}{4} \left(\frac{\mu}{2L} \right)^{1 / 2} < 1
\end{eqnarray*}

From the definition of $\gamma^*$~\eqref{gamma star remind} we can make an estimation:
\begin{equation} \label{gamma star ineq in lemma}
    \gamma^* \leqslant \log_{\mu / 2L}(3 \alpha) \Longrightarrow \alpha \leqslant \frac{1}{3} \left( \frac{\mu}{2L} \right)^{\gamma^*}.
\end{equation}

We can provide upper bound for linear term coefficient, using $(1 + \alpha)^2 + (1 - \alpha)^2 \leqslant 4$:
\begin{eqnarray*}
    s - m 
    & = & \left(1 + \frac{1}{4} \left(\frac{\mu}{2L} \right)^{\gamma^*} \right) (1 + \alpha)^2 + 2\alpha^2 - \left(1 - \frac{1}{4} \left(\frac{\mu}{2L} \right)^{\gamma^*} \right) (1 - \alpha)^2 + 2\alpha^2
    \\
    & = & 4\alpha + 4\alpha^2 + \frac{1}{4} \left(\frac{\mu}{2L} \right)^{\gamma^*} \left((1 + \alpha)^2 + (1 - \alpha)^2\right)
    \\
    & \leqslant & 4 (\alpha + \alpha^2) + 4 \left(\frac{\mu}{2L} \right)^{\gamma^*} 
    \leqslant 8 \alpha + \left(\frac{\mu}{2L} \right)^{\gamma^*} \overset{\eqref{gamma star ineq in lemma}}{\leqslant} \frac{11}{3} \left(\frac{\mu}{2L} \right)^{\gamma^*}.
\end{eqnarray*}

From $q$ definition:
\begin{equation*}
    q = \frac{\mu}{2\widehat{L}} = \frac{(1 - \alpha)^3}{8(1 + \alpha)} \frac{\mu}{2L} \quad \overset{\alpha \leqslant 1 / 3}{\geqslant} \quad \frac{1}{36} \cdot \frac{\mu}{2L}.
\end{equation*}

Let us prove by contradiction. To that end, we assume $\omega < \frac{1}{150} \left(\frac{\mu}{L}\right)^{1 - \gamma^*}$. Then:

\begin{align*}
    m \omega^2 + (s - m) \omega - q 
    & < \frac{m}{150^2} \left(\frac{\mu}{2L}\right)^{2(1 - \gamma^*)} + \frac{s - m}{150}\left(\frac{\mu}{2L}\right)^{1 - \gamma^*} - q \\
    &\leqslant \frac{m}{150^2} \left(\frac{\mu}{2L}\right)^{2(1 - \gamma^*)} + \frac{s - m}{150}\left(\frac{\mu}{2L}\right)^{1 - \gamma^*} - \frac{1}{36} \cdot \frac{\mu}{2L} \\ 
    & \overset{0 < m < 1}{\leqslant}
    \frac{1}{150^2} \left(\frac{\mu}{2L}\right)^{2(1 - \gamma^*)} + \frac{11}{450} \left(\frac{\mu}{2L}\right)^{\gamma^*}  \left(\frac{\mu}{2L}\right)^{1 - \gamma^*} - \frac{1}{36} \cdot \frac{\mu}{2L} \\
    &= \frac{\mu}{2L} \left( \frac{1}{150^2} \left(\frac{\mu}{2L}\right)^{1 - 2\gamma^*} + \frac{11}{450} - \frac{1}{36} \right) \overset{\gamma^* \leqslant 1/2}{<} 0
\end{align*}
Thus, we come to a contradiction because $\omega$ is a root of the equation $m \omega^2 + (s - m) \omega - q = 0$.

\end{proof}

\begin{lemma} \label{c > f lemma}
Let $f$ be an $L$-smooth and $\mu$-strongly convex function, $\widetilde{\nabla} f$ satisfies~\eqref{noise condition}, $\{x^k\}_{k \geqslant 0}$ be a sequence of points generated by Algorithm~\ref{alg re-agm}, and $c_k\, \forall k \geqslant 0$ be a minimal value of the function $\psi_k$ corresponding to Algorithm~\ref{alg re-agm}. Then $c_k \geqslant f(x^k) - \xi_k \, \forall k \geqslant 0$, where $\xi_k$ defined at~\eqref{xi noise accumulation accelerated}.
\end{lemma}
\begin{proof}
We will use the principle of mathematical induction to prove the statement of this lemma. For $k = 0$, it is obvious that $c_0 \geqslant f(x^0)$. Now, let us assume that $c_k \geqslant f(x^k) - \xi_k$, and prove the statement for $k+1$. For this, from Lemma \ref{psi comb}, \eqref{psi seq def}, denotation~\eqref{s, m remind} and definition of $\phi_{y^k}$ from Lemma~\ref{lower bound} we have 

\begin{eqnarray}\label{c combination}
    c_{k + 1}
    & = & (1 - \omega) c_k
    \\
    & + & \omega \Bigg( f(y^k) - \frac{s}{\mu} \|\nabla f(y^k) \|_2^2
    - \frac{1}{\mu} \left(4 \left(\frac{2L}{\mu} \right)^{\gamma^*} + 3 \right) \delta^2 \Bigg )
    \\
    & + & \frac{\omega (1 - \omega) \mu}{4} \left \|y^k - u^k - \frac{2}{\mu} \widetilde{\nabla} f(y^k) \right \|_2^2.
\end{eqnarray}

From Lemma~\ref{accelerated quadratic estimation} we can estimate last term:
\begin{eqnarray*}
    \frac{\omega (1 - \omega) \mu}{4} \left \|y^k - u^k - \frac{2}{\mu} \widetilde{\nabla} f(y^k) \right \|_2^2 
    & \geqslant & \omega (1 - \omega) \langle \nabla f(y^k), u^k - y^k \rangle \\
    & + & \frac{m \omega (1 - \omega)}{\mu} \|\nabla f(y^k) \|_2^2
    \\
    & - & \frac{\omega (1 - \omega)}{\mu} \left(4\left(\frac{2L}{\mu} \right)^{\gamma^*} + 1 \right) \delta^2.
\end{eqnarray*}
Combining the inequalities above and grouping the corresponding terms:
\begin{eqnarray*}
    c_{k + 1}
    & \geqslant & (1 - \omega) c_k + \omega f(y^k) + \omega (1 - \omega) \langle \nabla f(y^k), u^k - y^k \rangle
    \\
    & - & \frac{1}{\mu} \left( \omega s - \omega (1 - \omega) m \right) \|\nabla f(y^k) \|_2^2
    \\
    & \overset{(i)}{-} & \frac{\omega}{\mu} \left(8\left(\frac{2L}{\mu} \right)^{\gamma^*} + 4 \right) \delta^2, 
\end{eqnarray*}
$(i)$ can be obtained from $0 \leqslant \omega < 1$ (Lemma~\ref{omega quadsol}):
\begin{equation*}
    \omega \left(4 \left(\frac{2L}{\mu} \right)^{\gamma^*} + 3 \right) + \omega (1 - \omega) \left(4 \left(\frac{2L}{\mu} \right)^{\gamma^*} + 1 \right) \leqslant \omega \left(8\left(\frac{2L}{\mu} \right)^{\gamma^*} + 4 \right).
\end{equation*}
From convexity of function $f$ and induction hypotheses:
\begin{eqnarray*}
    (1 - \omega) c_k 
    & \geqslant & (1 - \omega) \left( f(x^k) - \xi_k \right)
    \\
    & \geqslant & (1 - \omega) \left( f(y^k) + \langle \nabla f(y^k), x^k - y^k \rangle \right) - (1 - \omega) \xi_k,
\end{eqnarray*}
then, using definition $y^k$ we can eliminate the linear form:
\begin{equation*}
     \langle \nabla f(y^k), \omega (1 - \omega) (u^k - y^k) + (1 - \omega) (x^k - y^k) \rangle = 0
\end{equation*}
and continue $c_{k + 1}$ estimation:
\begin{eqnarray*}
    c_{k + 1}
    & \geqslant & f(y^k) - (1 - \omega) \xi_k
    \\
    & - & \frac{1}{\mu} \left( \omega s - \omega (1 - \omega) m \right) \|\nabla f(y^k) \|_2^2
    \\
    & - & \frac{4\omega}{\mu} \left(2\left(\frac{2L}{\mu} \right)^{\gamma^*} + 1 \right) \delta^2, 
\end{eqnarray*}
From $\omega$ definition~\eqref{omega lemma def} one can simplify the coefficient of the gradient norm:
\begin{eqnarray*}
    c_{k + 1}
    & \geqslant & f(y^k) - (1 - \omega) \xi_k - \frac{q}{\mu} \|\nabla f(y^k) \|_2^2
    - \frac{4\omega}{\mu} \left(2\left(\frac{2L}{\mu} \right)^{\gamma^*} + 1 \right) \delta^2
    \\
    & \geqslant & f(y^k) - (1 - \omega) \xi_k
    - \frac{1}{2\widehat{L}} \|\nabla f(y^k) \|_2^2
    - \frac{4\omega}{\mu} \left(2\left(\frac{2L}{\mu} \right)^{\gamma^*} + 1 \right) \delta^2.
\end{eqnarray*}
Since $x^{k + 1} = y^k - h \wtgg f(y^k)$, where $h$ selected consistently with the Lemma~\ref{GD step} we can apply that Lemma:
\begin{eqnarray*}
    c_{k + 1}
    & \geqslant & f(x^{k + 1}) - (1 - \omega) \xi_k
    - \frac{3}{16 L} \frac{1}{(1 + \alpha)^2} \delta^2 - \frac{4\omega}{\mu} \left(2\left(\frac{2L}{\mu} \right)^{\gamma^*} + 1 \right) \delta^2 \\
    & \overset{\eqref{xi noise accumulation accelerated}}{=} & f(x^{k + 1}) - \xi_{k + 1}.
\end{eqnarray*}

\end{proof}

\begin{lemma} \label{upper bound for lemma}
Let $\{\lambda_k\}_{k \geqslant 0}$ be a sequence defined in \eqref{function seq} and $\{\psi_k\}_{k \geqslant 0}$ be the corresponding functions~\eqref{psi seq def}. Then 
\begin{equation*}
    \psi_k(x) \leqslant \lambda_k \psi_0(x) + (1 - \lambda_k) f(x), \quad \forall k \geqslant 0 \;\; \text{and} \;\; \forall x \in \mathbb{R}^n.
\end{equation*}

\end{lemma}
\begin{proof}
By the principle of induction. For $k = 0$, the inequality is obvious. Assume that inequality takes place for $k$. From \eqref{psi seq def}, and Lemma \ref{lower bound}, we get 
\begin{align*}
    \psi_{k + 1}(x)  &= (1 - \omega) \psi_k(x) + \omega \phi_{y^k, \gamma^*}(x) 
    \\& \leqslant (1 - \omega) \left( (1 - \lambda_k) f(x) + \lambda_k \psi_0(x) \right) + \omega f(x) \\
    & = \lambda_{k + 1} \psi_0(x) + (1 - \lambda_{k + 1}) f(x).
\end{align*}
\end{proof}

\begin{theorem}[Theorem~\ref{acc re agm conv text}] \label{acc re agm conv}
Let $f$ be an $L$-smooth and $\mu$-strongly convex function, $\widetilde{\nabla} f$ satisfies noise condition~\eqref{noise condition}, $\alpha \in [0, 1/3]$. Then Algorithm~\ref{alg re-agm} with parameters $(L, \mu, x^0, \alpha)$ generates $x^N$, s.t. 
\begin{eqnarray*}
    f(x^N) - f^* & \leqslant & 
    \left(1 - \frac{1}{150} \left(\frac{\mu}{2L}\right)^{1 - \gamma^*} \right)^N \left(f(x^0) - f^* + \frac{\mu}{4} R^2 \right)
    +  \left( \left(\frac{2L}{\mu} \right)^{\gamma^*} + 5 \right) \frac{\delta^2}{\mu}
    \\
    & \leqslant & 
    \left(1 - \frac{1}{300} \left(\frac{\mu}{L}\right)^{1 - \gamma^*} \right)^N \left(f(x^0) - f^* + \frac{\mu}{4} R^2 \right)
    + \left(2 \left(\frac{L}{\mu} \right)^{\gamma^*}  + 5 \right) \frac{\delta^2}{\mu}
\end{eqnarray*}
where $R := \|x^0 - x^* \|_2, \gamma^* = \min \left \lbrace \log_{\mu / 2L} (3 \alpha ), \frac{1}{2} \right \rbrace$.
\end{theorem}
\begin{proof}
From Lemma~\ref{c > f lemma} we obtain:
\begin{equation*}
    f(x^N) \leqslant c_N + \xi_N \quad \text{ see denotation~\eqref{xi noise accumulation accelerated} }. 
\end{equation*}
Applying Lemma~\ref{upper bound for lemma} and then Lemma~\ref{conv lemma}:
\begin{eqnarray*}
    f(x^N) - f^*
    & \leqslant &
    \lambda_N (\psi_0(x^*) - f^*) + \xi_N \\
    & \overset{\eqref{xi noise accumulation accelerated}}{\leqslant} & \lambda_N \left(f(x^0) - f^* + \frac{\mu}{4}R^2 \right)
    \\
    & + & \frac{1}{\omega} \left(\frac{3}{16L (1 + \alpha)^2} + \frac{\omega}{\mu} \left(\left(\frac{2L}{\mu} \right)^{\gamma^*} + 4 \right) \right) \delta^2
    \\
    & \overset{\text{Lemma}~\ref{w solution estimation}, \eqref{gamma star remind}}{\leqslant} & \left(1 - \frac{1}{150} \left(\frac{\mu}{2L}\right)^{1 - \gamma^*} \right)^N \left(f(x^0) - f^* + \frac{\mu}{4} R^2 \right)
    \\
    & + & \left( \frac{1}{800 (1 + \alpha)^2 L} \left(\frac{2L}{\mu}\right)^{1 - \gamma^*} + \frac{1}{\mu} \left(\left(\frac{2L}{\mu} \right)^{\gamma^*} + 4 \right) \right) \delta^2
    \\
    & \leqslant & \left(1 - \frac{1}{150} \left(\frac{\mu}{2L}\right)^{1 - \gamma^*} \right)^N \left(f(x^0) - f^* + \frac{\mu}{4} R^2 \right)
    \\
    & + & \left( \left(\frac{2L}{\mu}\right)^{-\gamma^*} + \left(\frac{2L}{\mu} \right)^{\gamma^*} + 4 \right) \frac{\delta^2}{\mu}
    \\
    & \overset{\gamma^* \geqslant 0}{\leqslant} & \left(1 - \frac{1}{150} \left(\frac{\mu}{2L}\right)^{1 - \gamma^*} \right)^N \left(f(x^0) - f^* + \frac{\mu}{4} R^2 \right)
    \\
    & + & \left( \left(\frac{2L}{\mu} \right)^{\gamma^*} + 5 \right) \frac{\delta^2}{\mu}
    \\
    & \overset{\gamma^* \leqslant 1}{\leqslant} & \left(1 - \frac{1}{300} \left(\frac{\mu}{L}\right)^{1 - \gamma^*} \right)^N \left(f(x^0) - f^* + \frac{\mu}{4} R^2 \right)
    \\
    & + & \left(2 \left(\frac{L}{\mu} \right)^{\gamma^*}  + 5 \right) \frac{\delta^2}{\mu}.
\end{eqnarray*}

\end{proof}

\section{Regularization missing proofs}

\begin{lemma} \label{reg dist bound}
    Let $f_{\mu}$ regularization function for $f$, then:
    \begin{equation*}
        R_{\mu} \leqslant R.
    \end{equation*}
\end{lemma}
\begin{proof}
    \begin{equation*}
        f^* + \frac{\mu}{2} R_{\mu}^2 \overset{\eqref{strong conv}}{\leqslant} f(x_{\mu}^* | x^0) + \frac{\mu}{2} R_{\mu}^2 = f_{\mu}^* \leqslant f_{\mu}(x^* | x^0) = f(x^*) + \frac{\mu}{2} R^2 \Rightarrow R_{\mu}^2 \leqslant R^2.
    \end{equation*}
\end{proof}

\begin{lemma} \label{reg noise}
    Let $f$ convex~\eqref{convexity} and $L$ smooth~\eqref{smooth cond}, $\wtgg f$ satisfies relative error condition~\eqref{relative noise condition}. Then gradient estimation of regularized function satisfies~\eqref{noise condition}:
    \begin{equation*}
        \| \wtgg f_{\mu}(x) - \nabla f_{\mu}(x) \|_2 \leqslant 2 \alpha \|\nabla f_{\mu}(x) \|_2 + \alpha \mu R.
    \end{equation*}
\end{lemma}
\begin{proof}
    \begin{eqnarray*}
        \| \wtgg f_{\mu}(x) - \nabla f_{\mu}(x) \|_2
        & = & \|\wtgg f(x) + \mu (x - x^0) - \nabla f(x) - \mu(x - x^0) \|_2 \\
        & \leqslant & \alpha \|\nabla f(x) \|_2 = \alpha \|\nabla f(x) + \mu (x - x^0) - \mu(x - x^0) \|_2 \\
        & \leqslant & \alpha \|\nabla f_{\mu}(x) \|_2 + \alpha \mu \|x - x^0 \|_2 \\
        & \leqslant & \alpha \|\nabla f_{\mu}(x) \|_2 + \alpha \mu \|x_{\mu}^* - x^0 \|_2 + \alpha \mu \|x - x_{\mu}^* \|_2 \\
        & \overset{\eqref{pl dist}}{\leqslant} & \alpha \|\nabla f_{\mu}(x) \|_2 + \alpha \mu R_{\mu} + \alpha \mu \cdot \frac{1}{\mu} \|\nabla f_{\mu}(x) \|_2 \\
        & \overset{\ref{reg dist bound}}{\leqslant} & 2 \alpha \| \nabla f(x) \|_2 + \alpha \mu R.
    \end{eqnarray*}
\end{proof}

\begin{lemma} \label{reg base}
    Let there be a method that guarantees convergence for strongly convex~\eqref{strong conv}, smooth~\eqref{smooth cond} function $f$ and gradient estimation $\wtgg f$ satisfying~\eqref{noise condition}:
    \begin{equation*}
        \begin{gathered}
            f(x^N) - f^* \leqslant C_0 LR^2 \cdot \left(1 - A_0 \left(\frac{1}{K_0}\frac{\mu}{L} \right)^{1 - \gamma(\alpha)} \right)^N + B_0 \left(K_0 \frac{L}{\mu} \right)^{\gamma(\alpha)} \frac{\delta^2}{\mu}, \\
            A_0 < 1, 1 \leqslant B_0, 0 < C_0, 0 \leqslant \gamma(\alpha) \leqslant 1/2, 0 < K_0.
        \end{gathered}
    \end{equation*}
    To solve the optimization problem for convex~\eqref{convexity}, smooth~\eqref{smooth cond} function $f$, under relative noise~\eqref{relative noise condition} with parameter $\alpha$, such that $\alpha^2 \left(K_0 \frac{B_0 L R^2}{\varepsilon} \right)^{\gamma(\alpha)} < \frac{1}{4}$, with precision $f(x^N) - f^* \leqslant \varepsilon \; (\varepsilon < LR^2)$ we can use regularization technique with parameter
    \begin{equation*}
        \mu = \frac{\varepsilon}{B_0 R^2}.
    \end{equation*}
    And required number of iterations will be:
    \begin{equation*}
        N = \frac{1}{A_0} \left(\frac{K_0 B_0 L R^2}{\varepsilon} \right)^{1 - \gamma(\alpha)} \ln \left(\frac{4 C_0 LR^2}{\varepsilon} \right).
    \end{equation*}
\end{lemma}
\begin{proof}
    Using Lemma~\ref{reg noise} we obtain, that $\| \wtgg f_{\mu}(x) - \nabla f_{\mu}(x) \|_2 \leqslant 2 \alpha \| \nabla f_{\mu}(x) \|_2 + \alpha \mu R$. We will choose $\mu$:
    \begin{equation*}
        \mu = \frac{\varepsilon}{B_0 R^2},
    \end{equation*}
    then
    \begin{eqnarray*}
        B_0 \left(K_0 \frac{L}{\mu} \right)^{\gamma(\alpha)} \frac{(\alpha \mu R)^2}{\mu} 
        & = &
        B_0 \left(K_0 \frac{L}{\mu} \right)^{\gamma(\alpha)} \alpha^2 \frac{\varepsilon}{B_0 R^2} \cdot R^2 
        \\
        & \overset{\alpha^2 \left(K_0 \frac{L}{\mu} \right)^{\gamma(\alpha)} < \frac{1}{4}}{\leqslant} & \varepsilon / 4.
    \end{eqnarray*}
    Let us estimate the number of steps to ensure $f_{\mu}(x^N | x^0) - f_{\mu}^* \leqslant \varepsilon / 4$:
    \begin{equation*}
        \begin{gathered}
        C_0 LR^2 \cdot \left(1 - A_0 \left(\frac{1}{K_0} \frac{\mu}{L} \right)^{1 - \gamma(\alpha)} \right)^N \leqslant \varepsilon / 4,
        \\
        N \geqslant \frac{1}{A_0} \left(\frac{K_0 B_0 L R^2}{\varepsilon} \right)^{1 - \gamma(\alpha)} \ln \left(\frac{4 C_0 LR^2}{\varepsilon} \right).
        \end{gathered}
    \end{equation*}
    Then:
    \begin{equation*}
         f_{\mu}(x^N | x^0) - f_{\mu}^* \leqslant C_0 \left(1 - A_0 \left(\frac{1}{K_0} \frac{\mu}{L} \right)^{\gamma(\alpha)} \right)^N + B_0 \left(K_0 \frac{L}{\mu} \right)^{\gamma(\alpha)} \frac{\alpha^2 \mu^2 R^2}{\mu} \leqslant \varepsilon / 2.
    \end{equation*}
    Now we can estimate the convergence for the original function:
    \begin{eqnarray*}
        f(x^N) - f^*
        & \leqslant & f(x^N) + \frac{\mu}{2} \|x^N - x^0 \|_2^2 - f^* \\
        & = & f_{\mu}(x^N | x^0) - f^* = f_{\mu}(x^N | x^0) - f_{\mu}(x^* | x^0) + \frac{\mu}{2} R^2 \\
        & \leqslant & f_{\mu}(x^N | x^0) - f_{\mu}^* + \frac{\varepsilon}{2B_0} \overset{B_0 \geqslant 1}{\leqslant} \varepsilon.
    \end{eqnarray*}
\end{proof}

\begin{theorem}[Theorem~\ref{gd reg text}] \label{gd reg}
    Let $f$ is convex~\eqref{convexity} and $L$ smooth~\eqref{smooth cond}, $\widetilde{\nabla} f$ satisfies relative noise condition~\eqref{relative noise condition} with $\alpha < 1 / 2$.
    Then we can solve problem~\eqref{optim} with precision $f(x^N) - f^* \leqslant \varepsilon \; (\varepsilon < LR^2)$, using Algorithm~\ref{alg gd} via regularization technique with following amount of iterations:
    \begin{equation*}
        N =  12 \frac{(1 + \alpha)^2}{(1 - \alpha)^6} \frac{LR^2}{\varepsilon} \ln \left(\frac{2 LR^2}{\varepsilon} \right).
    \end{equation*}
\end{theorem}
\begin{proof}
    From Theorem~\ref{gd pl conv text}:
    \begin{eqnarray*}
        f(x^N) - f^* 
        & \leqslant &  \left(1 -  \frac{(1 - \alpha)^3}{1 + \alpha} \frac{\mu}{8 L} \right)^N (f(x^0) - f^*) + \frac{3}{2 \mu} \frac{1 + \alpha}{(1 - \alpha)^3} \delta^2.
    \end{eqnarray*}
    We will use Lemma~\ref{reg base}. In that notation:
    \begin{equation*}
        \gamma(\alpha) \equiv 0, \; B_0  = \frac{3}{2} \frac{1 + \alpha}{(1 - \alpha)^3}, \; A_0 = \frac{1}{8} \frac{(1 - \alpha)^3}{1 + \alpha}, \; C_0 \overset{\eqref{smooth cond}}{=} \frac{1}{2}, K_0 = 1. 
    \end{equation*}
    Then we can choose:
    \begin{equation*}
        \mu = \frac{\varepsilon}{R^2} \cdot \frac{2}{3} \frac{(1 - \alpha)^3}{1 + \alpha},
    \end{equation*}
    And number of required iterations:
    \begin{equation*}
        N = 8 \frac{1 + \alpha}{(1 - \alpha)^3} \cdot \frac{3}{2} \frac{1 + \alpha}{(1 - \alpha)^3} \cdot \frac{LR^2}{\varepsilon} \ln \left(\frac{2 LR^2}{\varepsilon} \right) = 12 \frac{(1 + \alpha)^2}{(1 - \alpha)^6} \frac{LR^2}{\varepsilon} \ln \left(\frac{2 LR^2}{\varepsilon} \right).
    \end{equation*}
\end{proof}

\begin{theorem}[Theorem~\ref{re-agm reg text}] \label{re-agm reg}
    Let $f$ is convex~\eqref{convexity} and $L$ smooth~\eqref{smooth cond}, $\widetilde{\nabla} f$ satisfies relative noise condition~\eqref{relative noise condition} with $\alpha \leqslant \frac{1}{3} \left(\frac{\varepsilon}{12 LR^2} \right)^{\beta}, 0 \leqslant \beta \leqslant 1/2$.
    Then we can solve problem~\eqref{optim} with precision $f(x^N) - f^* \leqslant \varepsilon \; (\varepsilon < LR^2)$, using Algorithm~\ref{alg re-agm} via regularization technique with following amount of iterations:
    \begin{equation*}
        N = 150 \left(\frac{12 LR^2}{\varepsilon} \right)^{1 - \beta } \ln \left(\frac{12 LR^2}{\varepsilon} \right).
    \end{equation*}
\end{theorem}
\begin{proof}
    From Theorem~\ref{acc re agm conv text}:
    \begin{eqnarray*}
        f(x^N) - f^* & \leqslant & 
        \left(1 - \frac{1}{150} \left(\frac{\mu}{2L}\right)^{1 - \gamma^*} \right)^N \left(f(x^0) - f^* + \frac{\mu}{4} R^2 \right)
        +  \left( \left(\frac{2L}{\mu} \right)^{\gamma^*} + 5 \right) \frac{\delta^2}{\mu}
        \\
        & \leqslant & \left(1 - \frac{1}{150} \left(\frac{\mu}{2L}\right)^{1 - \gamma^*} \right)^N \left(f(x^0) - f^* + \frac{\mu}{4} R^2 \right)
        +  6 \left(\frac{2L}{\mu} \right)^{\gamma^*} \frac{\delta^2}{\mu}
    \end{eqnarray*}
    We will use Lemma~\ref{reg base}. In that notation:
    \begin{equation*}
        \gamma(\alpha) = \min \left \lbrace 1/2, \log_{\mu / 2L} (3 \alpha ) \right \rbrace, \; B_0 = 6, \; A_0 = \frac{1}{150}, \; C_0 \overset{\eqref{smooth cond}, \eqref{strong conv}}{=} 1, \; K_0 = 2. 
    \end{equation*}
    Then we can choose:
    \begin{equation*}
        \mu = \frac{\varepsilon}{6 R^2}.
    \end{equation*}
    Condition $\alpha^2 \cdot \left(K_0 \frac{B_0 L R^2}{\varepsilon} \right)^{\gamma(\alpha)} < \frac{1}{4}$ is also guaranteed:
    \begin{eqnarray*}
        \alpha^2 \left(K_0 \frac{B_0 L R^2}{\varepsilon} \right)^{\gamma(\alpha)} 
        & \leqslant & \alpha^2 \min \left \lbrace \left(\frac{12 L R^2}{\varepsilon} \right)^{ \log_{\varepsilon/12LR^2}(3 \alpha)},  \left(\frac{12 L R^2}{\varepsilon} \right)^{1/2} \right \rbrace
        \\
        & \leqslant & \min  \left \lbrace \frac{\alpha}{3},  \frac{1}{9} \left(\frac{12 L R^2}{\varepsilon} \right)^{2\beta - 1/2} \right \rbrace \leqslant 1 / 9 < 1 / 4,
    \end{eqnarray*}
    also:
    \begin{equation*}
        1 - \gamma(\alpha) = 1 - \min \left \lbrace 1/2, \log_{\varepsilon/12LR^2}(3 \alpha) \right \rbrace \leqslant 1 - \min \left \lbrace 1/2, \beta \right \rbrace = 1 - \beta.
    \end{equation*}
    Finally number of required iterations:
    \begin{equation*}
        N = 150 \left(\frac{12 LR^2}{\varepsilon} \right)^{1 - \beta} \ln \left(\frac{4 LR^2}{\varepsilon} \right).
    \end{equation*}
\end{proof}

\section{Adaptiveness missing proofs} \label{appendix adap}

\begin{lemma} \label{alpha est inequality}
    If $f$ satisfies smoothness~\eqref{smooth cond}, $\widetilde{\nabla} f$ satisfies error condition~\eqref{noise def abs rel}, $\widehat{\alpha} \geqslant \alpha, \widehat{L} \geqslant L$, then gradient step of Gradient Descent~\ref{alg gd} with $h$ parameter:
    \begin{equation*}
        h = \frac{1}{4 \widehat{L}} \sqrt{\frac{1 - \widehat{\alpha}}{1 + \widehat{\alpha}}},
    \end{equation*}
    provides:
    \begin{eqnarray*}
        f(x^{k + 1}) 
       & \leqslant & f(x^k) - \frac{1}{32 \widehat{L}} \frac{1 - \widehat{\alpha}}{1 + \widehat{\alpha}} \| \wtgg f(x) \|_2^2 + \frac{3}{4 \widehat{L}} \frac{1}{(1 + \widehat{\alpha})^2} \delta^2.
    \end{eqnarray*}
\end{lemma}
\begin{proof}
     Using smoothness of function $f$ we get:
    \begin{equation*}
        f(x^{k + 1}) \leqslant f(x^k) + \langle \nabla f(x^k), x^{k + 1} - x^k \rangle + \frac{L}{2} \| x^{k + 1} - x^k \|_2^2.
    \end{equation*}
    Using, Lemma~\ref{cos lb}:
    \begin{align*}
        \langle \nabla f(x^k), x^{k + 1} - x^k \rangle 
        & = -h \langle \nabla f(x), \wtgg f(x) \rangle
        \\
        & \leqslant \frac{-h}{4} \sqrt{\frac{1 - \alpha}{1 + \alpha}} \| \wtgg f(x) \|_2^2 
        + \frac{3h}{\sqrt{(1 + \alpha)^3(1 - \alpha)}} \delta^2
    \end{align*}
    Then we can continue initial inequality:
    \begin{eqnarray*}
        f(x^{k + 1}) 
        & \leqslant & f(x^k)
        - \frac{h}{4} \sqrt{\frac{1 - \alpha}{1 + \alpha}} \| \wtgg f(x) \|_2^2 + \frac{3h}{\sqrt{(1 + \alpha)^3(1 - \alpha)}} \delta^2
        \\
        & + & \frac{L h^2}{2} \|\wtgg f(x^k) \|_2^2
        \\
        & \overset{\alpha \leqslant \widehat{\alpha}, \; L \leqslant \widehat{L}}{\leqslant} &
        f(x^k) - \frac{h}{4} \sqrt{\frac{1 - \widehat{\alpha}}{1 + \widehat{\alpha}}} \| \wtgg f(x) \|_2^2
        \\
        & + & \frac{3h}{\sqrt{(1 + \widehat{\alpha})^3(1 - \widehat{\alpha})}} \delta^2
        + \frac{\widehat{L} h^2}{2} \|\wtgg f(x^k) \|_2^2 
        \\
        & \overset{h \text{ definition}}{=} & f(x^k) -\frac{1}{32 \widehat{L}} \frac{1 - \widehat{\alpha}}{1 + \widehat{\alpha}} \| \wtgg f(x) \|_2^2 + \frac{3}{4 \widehat{L}} \frac{1}{(1 + \widehat{\alpha})^2} \delta^2.
    \end{eqnarray*}
\end{proof}

\begin{theorem}[Theorem~\ref{adaptive PL convergence alpha and smooth text}] \label{th:gd_adapt_both}
    Let $f$ function satisfies P$\L$ condition~\eqref{PL cond} and smoothness~\eqref{smooth cond}, $\wtgg f$ satisfies error condition~\eqref{noise condition}. Then Algorithm~\ref{alg rel adaptive gd} with $\tau = \textbf{true}$ (enable smooth adaptiveness) provides convergence:
    \begin{eqnarray*}    
        f(x^N) - f^* & \leqslant & \left( 1 - \frac{(1 - \alpha)^2}{256} \min \left\lbrace (1 - \alpha)^{2}, \left(\nicefrac{L_0}{L} \right)^2 \right \rbrace \frac{\mu}{L_0} \right)^N (f(x^0) - f^*)
        \\
        & + & \frac{200}{(1 - \alpha)^{2}} \max \left\lbrace (1 - \alpha)^{-2}, \left(\nicefrac{L}{L_0} \right)^2 \right\rbrace \frac{\delta^2}{\mu},
    \end{eqnarray*}  
    and total number iterations of the inner loop is bounded by:
    \begin{equation*}
         N + \max \left \lbrace \log_2\left((1 - \alpha)^{-1}\right), \log_2\left(\nicefrac{L}{L_0} \right) \right \rbrace + 1.
    \end{equation*}
\end{theorem}
\begin{proof}
    Based on description of Algorithm~\ref{alg rel adaptive gd} we should justify the completion of the "while" loop. That loop is running until the following condition is met:
    \begin{equation*}
        f(x^{k + 1}) \leqslant f(x^k) -\frac{1}{32 \widehat{L}_k} \frac{1 - \widehat{\alpha}_k}{1 + \widehat{\alpha}_k} \| \wtgg f(x) \|_2^2 + \frac{3}{4 \widehat{L}_k} \frac{1}{(1 + \widehat{\alpha}_k)^2} \delta^2.
    \end{equation*}
    Then (from result of Lemma~\ref{alpha est inequality}), one can see, that after, not more than:
    \begin{equation*}
        \max \left \lbrace \log_2\left((1 - \alpha)^{-1}\right), \max \left \lbrace 0, \log_2\left(\nicefrac{L}{L_0} \right) \right \rbrace \right \rbrace + 1,
    \end{equation*}
    iterations this loop will stop, because:
    \begin{equation*}
        \begin{gathered}
            1 - 2^{-\log_2(1 - \alpha) - 1} = \frac{1 + \alpha}{2} > \alpha, \\
            L_0 2^{\log_2\left(\nicefrac{L}{L_0} \right) + 1} = 2L > L,
        \end{gathered}
    \end{equation*}
    thus the following estimation can be provided:
    \begin{equation} \label{eq:J_ub}
        \begin{gathered}
            J_k \leqslant \max \left \lbrace \log_2\left((1 - \alpha)^{-1}\right), \log_2\left(\nicefrac{L}{L_0} \right) \right \rbrace + 1, \\
            4^{J_k} \leqslant 4 \cdot \max \left\lbrace (1 - \alpha)^{-2}, \left(\nicefrac{L}{L_0} \right)^2 \right \rbrace
        \end{gathered}
    \end{equation}
    We can estimate the following fraction:
    \begin{equation} \label{eq:1-a/L}
       \frac{1 - \widehat{\alpha}_k}{\widehat{L}_k} = 
       \frac{1 - \left(1 - 2^{-J_k} \right)}{L_0 2 ^{J_k}} = \frac{1}{L_0} 4^{-J_k}.
    \end{equation}
    When stop condition of the "while" loop is met:
    \begin{equation*}
        f(x^{k + 1}) \leqslant f(x^k) -\frac{1}{32 \widehat{L}_k} \frac{1 - \widehat{\alpha}_k}{1 + \widehat{\alpha}_k} \| \wtgg f(x) \|_2^2 + \frac{3}{4 \widehat{L}_k} \frac{1}{(1 + \widehat{\alpha}_k)^2} \delta^2,
    \end{equation*}
    from Lemma~\ref{PL noised}:
    \begin{eqnarray*}
        f(x^{k + 1}) - f^* & \leqslant & 
        f(x^k) - f^* - \frac{1}{32 \widehat{L}_k} \frac{1 - \widehat{\alpha}_k}{1 + \widehat{\alpha}_k} \| \wtgg f(x) \|_2^2
        \\
        & + & \frac{3}{4 \widehat{L}_k} \frac{1}{(1 + \widehat{\alpha}_k)^2} \delta^2 
        \\
        & \overset{\text{Lemma}~\ref{PL noised}}{\leqslant} & f(x^k) - f^* - \frac{\mu}{32 \widehat{L}_k}  \frac{1 - \widehat{\alpha}_k}{1 + \widehat{\alpha}_k} (1 - \alpha)^2 (f(x^k) - f^*) \\
        & + & \frac{1}{32 \widehat{L}_k} \frac{1 - \widehat{\alpha}_k}{1 + \widehat{\alpha}_k} \delta^2
        + \frac{3}{4 \widehat{L}_k} \frac{1}{(1 + \widehat{\alpha}_k)^2} \delta^2 
        \\
        & \overset{\frac{1 - \widehat{\alpha}_k}{1 + \widehat{\alpha}_k} \leqslant \frac{1}{(1 + \widehat{\alpha}_k)^2}}{\leqslant} & 
        \left( 1 - \frac{\mu}{32\widehat{L}_k}  \frac{1 - \widehat{\alpha}_k}{1 + \widehat{\alpha}_k} (1 - \alpha)^2 \right) (f(x^k) - f^*)
        \\
        & + & \frac{1}{(1 + \widehat{\alpha}_k)^2} \frac{25 \delta^2}{32 \widehat{L}_k}
        \\
        & \overset{\eqref{eq:1-a/L}, 0 \leqslant \wha_k \leqslant 1, L_0 \leqslant \whL_k }{\leqslant} & 
        \left( 1 - \frac{\mu}{64 L_0}  4^{-J_k} (1 - \alpha)^2 \right) (f(x^k) - f^*)
        + \frac{25 \delta^2}{32 L_0}
    \end{eqnarray*}
    Denoting $\BJ_N = \max\limits_{0 \leqslant k \leqslant N - 1} J_k$, we can iterate inequality above:
    \begin{eqnarray*}
       f(x^N) - f^*
       & \leqslant & \left( 1 - \frac{\mu}{64 L_0}  4^{-\BJ_N} (1 - \alpha)^2 \right)^N (f(x^0) - f^*)
       \\
       & + & \sum_{k = 0}^{N - 1} \left( 1 - \frac{\mu}{64 L_0}  4^{-\BJ_N} (1 - \alpha)^2 \right)^k \frac{25 \delta^2}{32 L_0} 
       \\
       & \leqslant & \left( 1 - \frac{\mu}{64 L_0}  4^{-\BJ_N} (1 - \alpha)^2 \right)^N (f(x^0) - f^*)
       + \frac{50}{(1 - \alpha)^{2}} 4^{\BJ_N} \frac{\delta^2}{\mu} 
       \\
       & \overset{\eqref{eq:J_ub}}{\leqslant} & \left( 1 - \frac{(1 - \alpha)^2}{256} \min \left\lbrace (1 - \alpha)^{2}, \left(\nicefrac{L_0}{L} \right)^2 \right \rbrace \frac{\mu}{L_0} \right)^N (f(x^0) - f^*)
       \\
       & + & \frac{200}{(1 - \alpha)^{2}} \max \left\lbrace (1 - \alpha)^{-2}, \left(\nicefrac{L}{L_0} \right)^2 \right \rbrace \frac{\delta^2}{\mu}.
    \end{eqnarray*}
    Now we should estimate total amount iterations of the "while" loop. On the $k$-th iteration "while" loop will work
    $J_{k + 1} - \max \left \lbrace J_k - 1, 1 \right \rbrace \leqslant J_{k + 1} - J_k + 1$. Then total number of iterations can be bounded:
    \begin{equation*}
        \sum_{k = 0}^{N - 1} \left( J_{k + 1} - J_k + 1 \right) = N + J_{N - 1} - J_0 \overset{\eqref{eq:J_ub}}{\leqslant} N +  \max \left \lbrace \log_2\left((1 - \alpha)^{-1}\right), \log_2\left(\nicefrac{L}{L_0} \right) \right \rbrace + 1.
    \end{equation*}
\end{proof}

\begin{theorem}[Theorem~\ref{adaptive PL convergence only alpha text}]
    Let $f$ function satisfies P$\L$ condition~\eqref{PL cond} and $L$-smoothness~\eqref{smooth cond}, $\wtgg f$ satisfies error condition~\eqref{noise condition}. Then Algorithm~\ref{alg rel adaptive gd} with $\tau = \textbf{false}$ (disable smooth adaptiveness) and $L_0 = L$ provides convergence:
    \begin{eqnarray*}    
        f(x^N) - f^* \leqslant \left( 1 - \frac{(1 - \alpha)^3}{128} \frac{\mu}{L} \right)^N (f(x^0) - f^*)
       + \frac{100}{(1 - \alpha)^{3}} \frac{\delta^2}{\mu},
    \end{eqnarray*}  
    and total number iterations of the inner loop is bounded by:
    \begin{equation*}
         N +  \log_2\left((1 - \alpha)^{-1}\right) + 1.
    \end{equation*}
\end{theorem}
\begin{proof}
    We will repeat the proof of the Theorem~\ref{th:gd_adapt_both} with clarifications for the case $\whL_k = L$.
    Based on description of Algorithm~\ref{alg rel adaptive gd} we should justify the completion of the "while" loop. That loop is running until the following condition is met:
    \begin{equation*}
        f(x^{k + 1}) \leqslant f(x^k) -\frac{1}{32 L} \frac{1 - \widehat{\alpha}_k}{1 + \widehat{\alpha}_k} \| \wtgg f(x) \|_2^2 + \frac{3}{4 L} \frac{1}{(1 + \widehat{\alpha}_k)^2} \delta^2.
    \end{equation*}
    From result of Lemma~\ref{alpha est inequality}, after not more than:
    \begin{equation*}
        \log_2\left((1 - \alpha)^{-1}\right) + 1,
    \end{equation*}
    iterations this loop will stop, thus the following estimation can be provided:
    \begin{equation} \label{eq:a_ub}
        \begin{gathered}
            J_k \leqslant \log_2\left((1 - \alpha)^{-1}\right) + 1, \\
            \wha_k \leqslant \frac{1 + \alpha}{2}, \; 1 - \wha_k \geqslant \frac{1 - \alpha}{2}.
        \end{gathered}
    \end{equation}
    After inner "while" loop is finished we can provide (see proof of Theorem~\ref{th:gd_adapt_both} for details):
    \begin{eqnarray*}
        f(x^{k + 1}) - f^* & \leqslant & 
        f(x^k) - f^*  -\frac{1}{32 L} \frac{1 - \widehat{\alpha}_k}{1 + \widehat{\alpha}_k} \| \wtgg f(x) \|_2^2 + \frac{3}{4 L} \frac{1}{(1 + \widehat{\alpha}_k)^2} \delta^2 
        \\
        & \overset{\text{Lemma}~\ref{PL noised}}{\leqslant} & 
        \left( 1 - \frac{\mu}{32 L}  \frac{1 - \widehat{\alpha}_k}{1 + \widehat{\alpha}_k} (1 - \alpha)^2 \right) (f(x^k) - f^*)
        \\
        & + & \frac{1}{(1 + \widehat{\alpha}_k)^2} \frac{25 \delta^2}{32 L}
        \\
        & \overset{\eqref{eq:a_ub}, 0 \leqslant \wha_k \leqslant 1}{\leqslant} & \left( 1 - \frac{(1 - \alpha)^3}{128} \frac{\mu}{L} \right) (f(x^k) - f^*) + \frac{25 \delta^2}{32 L}.
    \end{eqnarray*}
    Iterating inequality above:
    \begin{eqnarray*}
       f(x^N) - f^*
       & \leqslant & \left( 1 - \frac{(1 - \alpha)^3}{128} \frac{\mu}{L} \right)^N (f(x^0) - f^*)
       \\
       & + & \sum_{k = 0}^{N - 1} \left( 1 - \frac{(1 - \alpha)^3}{128} \frac{\mu}{L} \right)^k \frac{25 \delta^2}{32 L} 
       \\
       & \leqslant & \left( 1 - \frac{(1 - \alpha)^3}{128} \frac{\mu}{L} \right)^N (f(x^0) - f^*)
       + \frac{100}{(1 - \alpha)^{3}} \frac{\delta^2}{\mu}.
    \end{eqnarray*}
    And similarly to the Theorem~\ref{th:gd_adapt_both} one can estimate total number of the "while" loop:
    \begin{equation*}
        \sum_{k = 0}^{N - 1} \left( J_{k + 1} - J_k + 1 \right) = N + J_{N - 1} - J_0 \overset{\eqref{eq:a_ub}}{\leqslant} N +  \log_2\left((1 - \alpha)^{-1}\right) + 1.
    \end{equation*}
\end{proof}

\section{Relative interpretation of absolute noise missing proofs} \label{appendix relative interpretation}

\begin{theorem}[Theorem~\eqref{main relative interpretation th text}] \label{main relative interpretation th}
    Let $f$ is $\mu$-strongly convex~\eqref{strong conv} and $L$-smooth~\eqref{smooth cond}, $\wtgg f$ satisfies~\eqref{noise condition}. Algorithm $\mathcal{A}$ is first order method~\eqref{linear first order method}, using $\wtgg f$ satisfying relative noise~\ref{relative noise condition} with magnitude $\alpha_0$ produces $x^N$, such that:
    \begin{equation*}
        f(x^N) - f^* \leqslant C_0 LR^2 \exp \left(-A_0 \left(\frac{\mu}{L} N \right)^{\gamma(\alpha_0)} \right)^N,
    \end{equation*}
    where:
    \begin{equation*}
        \gamma: [0; 1) \to [1 / 2; + \infty]
    \end{equation*}
    Then for any $K > (1 - \alpha)^{-1}$ we can solve the optimization problem~\eqref{optim} with precision $\varepsilon_0$, where
    \begin{equation*}
        f(x^{N_0}) - f^*
        \leqslant \varepsilon_0 = 
        \frac{1}{(1 - \alpha)^2} \left( \left((1 + \alpha)K + 1 \right)^2 + 1 \right) \frac{\delta^2}{\mu},
    \end{equation*}
    and
    \begin{equation*}
        N_0 \leqslant \frac{1}{A_0} \left(\frac{L}{\mu} \right)^{\gamma(\widehat{\alpha})} \ln \left( \frac{(1 - \alpha)^2}{\left((1 + \alpha)K + 1 \right)^2 + 1} \cdot C_0 \cdot \frac{\mu}{\delta^2} \right),
    \end{equation*}
    using algorithm $\mathcal{A}$ with a relative error parameter $\widehat{\alpha} = \alpha + \frac{1}{K}$.
\end{theorem}
\begin{proof}
    We will run algorithm $\mathcal{A}$ with relative error parameter $\widehat{\alpha} = \alpha + \frac{1}{K}$ until $\| \wtgg f(x^{N}) \|_2 > \widetilde{K} \delta$, where
    \begin{equation*}
        \widetilde{K} = (1 + \alpha)K + 1.
    \end{equation*}
    Firstly we should note that if $\| \wtgg f(x^{N}) \|_2 > \widetilde{K} \delta$ one can obtain, using~\eqref{basic alpha conditions}:
    \begin{equation*}
        \|\nabla f(x^N) \|_2 \geqslant (1 + \alpha)^{-1} \left(\| \wtgg f(x^N) \|_2 - \delta \right) > (1 + \alpha)^{-1} \left(\widetilde{K} \delta - \delta \right) = K \delta,
    \end{equation*}
    that is while algorithm $\mathcal{A}$ is working we can guarantee, that:
    \begin{equation*}
        \frac{\|\zeta_a(x^{k}) + \zeta_r(x^{k}) \|_2}{\|\nabla f(x^{k}) \|_2} \leqslant \frac{\alpha \|\nabla f(x^{k}) \|_2 + \delta}{\| \nabla f(x^{k}) \|_2} < \alpha + \frac{1}{K} = \widehat{\alpha} < 1.
    \end{equation*}
    On the contrary, if $\| \wtgg f(x^{N}) \|_2 \leqslant \widetilde{K} \delta$:
    \begin{eqnarray*}
        f(x^{N}) - f^* 
        & \overset{\text{Lemma}~\ref{PL noised}}{\leqslant} & \frac{1}{\mu (1 - \alpha^2)} \left( \|\wtgg f(x^{N}) \|_2^2 + \delta^2 \right)
        \\
        & \leqslant &
        \frac{1}{(1 - \alpha)^2} \left( \left((1 + \alpha)K + 1 \right)^2 + 1 \right) \frac{\delta^2}{\mu}.
    \end{eqnarray*}
    Now we can estimate number of iterations when condition $f(x^N) - f^* \leqslant \varepsilon_0$ will be reached (while $\| \wtgg f(x^k) \|_2 > \widetilde{K} \delta$):
    \begin{equation*}
        \begin{gathered}
            f(x^N) - f^* \leqslant C_0 \exp \left( - A_0 \left(\frac{\mu}{L} \right)^{\gamma(\widehat{\alpha})} N \right) \leqslant \varepsilon_0,
            \\
            \exp \left( - A_0 \left(\frac{\mu}{L} \right)^{\gamma(\widehat{\alpha})} N \right) \leqslant \frac{\varepsilon_0}{C_0},
            \\
            A_0 \left(\frac{\mu}{L} \right)^{\gamma(\widehat{\alpha})} N  \geqslant \ln \left( \frac{C_0}{\varepsilon_0} \right),
            \\
            N \geqslant \frac{1}{A_0} \left(\frac{L}{\mu} \right)^{\gamma(\widehat{\alpha})} \ln \left( \frac{C_0}{\varepsilon_0} \right),
            \\
            N \geqslant \frac{1}{A_0} \left(\frac{L}{\mu} \right)^{\gamma(\widehat{\alpha})} \ln \left( \frac{(1 - \alpha)^2}{\left((1 + \alpha)K + 1 \right)^2 + 1} \cdot C_0 \cdot \frac{\mu}{\delta^2} \right).
        \end{gathered}
    \end{equation*}
    That is, one of the two conditions will be met - either $f(x^{N_0}) - f^* \leqslant \varepsilon_0$ or $\| \wtgg f(x^{N_0}) \|_2 > \widetilde{K} \delta$ will happen first after the following amount of iterations:
    \begin{equation*}
        N_0 \leqslant \frac{1}{A_0} \left(\frac{L}{\mu} \right)^{\gamma(\widehat{\alpha})} \ln \left( \frac{(1 - \alpha)^2}{\left((1 + \alpha)K + 1 \right)^2 + 1} \cdot C_0 \cdot \frac{\mu}{\delta^2} \right).
    \end{equation*}
\end{proof}

\begin{theorem}[Theorem~\ref{re-agm rel interpretation text}] \label{re-agm rel interpretation}
    Let $f$ is $\mu$-strongly convex~\eqref{strong conv} and $L$-smooth~\eqref{smooth cond}, $\wtgg f$ satisfies~\eqref{noise condition}, $\alpha \leqslant \frac{1}{6} \left(\frac{\mu}{2L}\right)^{\gamma_0}, 0 \leqslant \gamma_0 \leqslant \frac{1}{2}$. Using Algorithm~\ref{alg re-agm} with parameters $(L, \mu, x^0, \widehat{\alpha})$, where
    \begin{equation*}
        \widehat{\alpha} = \alpha + \frac{1}{6} \left(\frac{\mu}{2L}\right)^{\beta}, \quad 0 \leqslant \beta \leqslant \frac{1}{2},
    \end{equation*}
    we can obtain:
    \begin{equation*}
        \min \left \lbrace f(y^{N}) - f^*, f(x^{N}) - f^* \right \rbrace
        \leqslant \frac{1}{(1 - \alpha)^2} \left( \left(6(1 + \alpha)\left(\frac{2L}{\mu}\right)^{\beta} + 1 \right)^2 + 1 \right) \frac{\delta^2}{\mu}.
    \end{equation*}
    We will forcefully stop the algorithm based on the condition:
    \begin{equation*}
        \|\wtgg f(x^N) \|_2 \leqslant \left(6(1 + \alpha) \left(\frac{2L}{\mu}\right)^{\beta} + 1 \right) \delta,
    \end{equation*}
    in this case, the method will work no more than
    \begin{equation*}
        N = 300 \left(\frac{L}{\mu} \right)^{1 - \min\left \lbrace \gamma_0, \beta \right \rbrace} \ln \left( \frac{(1 - \alpha)^2 }{\left(6(1 + \alpha) \left(\frac{2L}{\mu}\right)^{\beta} + 1 \right)^2 + 1} \frac{LR^2}{\nicefrac{\delta^2}{\mu}} \right).
    \end{equation*}
\end{theorem}
\begin{proof}
    We will use Theorem~\ref{main relative interpretation th}. Firstly, in notation of Theorem~\ref{main relative interpretation th} and result of Theorem~\ref{acc re agm conv}:
    \begin{equation*}
        \begin{gathered}
            \gamma(\alpha) = 1 - \min \left \lbrace \log_{\mu / 2L} (3 \alpha ), \frac{1}{2} \right \rbrace, \\
            A_0 = \frac{1}{300}, \\
            C_0 = 1, \; \text{We used: } f(x^0) - f^* + \frac{\mu}{4} R^2 \leqslant LR^2, \\
            K = 6 \left(\frac{2L}{\mu}\right)^{\beta}, \\
            \widehat{\alpha} = \alpha + \frac{1}{6} \left(\frac{\mu}{2L}\right)^{\beta} \leqslant  \frac{1}{6}\left(\frac{\mu}{2L}\right)^{\gamma_0} + \frac{1}{6} \left(\frac{\mu}{2L}\right)^{\beta} \leqslant \frac{1}{3} \left(\frac{\mu}{2L}\right)^{\min \left \lbrace \beta, \gamma_0 \right \rbrace}, \\
            \log_{\mu / 2L} (3 \widehat{\alpha} ) \geqslant \min \left \lbrace \beta, \gamma_0 \right \rbrace, \\
            \min \left \lbrace \log_{\mu / 2L} (3 \widehat{\alpha} ), \frac{1}{2} \right \rbrace \geqslant \min \left \lbrace \frac{1}{2}, \min \left \lbrace \beta, \gamma_0 \right \rbrace \right \rbrace = \min \left \lbrace \beta, \gamma_0 \right \rbrace, \\
            \gamma(\widehat{\alpha}) \leqslant 1 - \min \left \lbrace \beta, \gamma_0 \right \rbrace.
        \end{gathered}
    \end{equation*}
    We will run Algorithm~\ref{alg re-agm} while $\| \nabla f(y^N) \|_2 > K \delta$. Thus after $N_0$:
    \begin{eqnarray*}
        \min \left \lbrace f(y^{N_0}) - f^*, f(x^{N_0}) - f^* \right \rbrace
        & \leqslant & \frac{1}{(1 - \alpha)^2} \left( \left((1 + \alpha)K + 1 \right)^2 + 1 \right) \frac{\delta^2}{\mu}
        \\
        & \leqslant & \frac{1}{(1 - \alpha)^2} \left( \left(6(1 + \alpha)\left(\frac{2L}{\mu}\right)^{\beta} + 1 \right)^2 + 1 \right) \frac{\delta^2}{\mu},
    \end{eqnarray*}
    where:
    \begin{eqnarray*}
        N_0 & \leqslant & \frac{1}{A_0} \left(\frac{L}{\mu} \right)^{\gamma(\widehat{\alpha})} \ln \left( \frac{(1 - \alpha)^2}{\left((1 + \alpha)K + 1 \right)^2 + 1} \cdot C_0 \cdot \frac{\mu}{\delta^2} \right)
        \\
        & \leqslant & 300 \left(\frac{L}{\mu} \right)^{1 - \min\left \lbrace \gamma_0, \beta \right \rbrace} \ln \left( \frac{(1 - \alpha)^2 }{\left(6(1 + \alpha) \left(\frac{2L}{\mu}\right)^{\beta} + 1 \right)^2 + 1} LR^2 \frac{\mu}{\delta^2} \right).
    \end{eqnarray*}
\end{proof}

\begin{theorem}[Theorem~\ref{re-agm rel interpret and reg text}] \label{re-agm rel interpret and reg}
    Let $f$ is convex~\eqref{convexity} and smooth~\eqref{smooth cond}, $\varepsilon \leqslant L R^2$, $\widetilde{\nabla} f$ satisfies relative noise condition with:
    \begin{equation*}
        \begin{gathered}
            0 < \alpha \leqslant \frac{1}{9} \left(\frac{\varepsilon}{2 L R^2} \right)^{\tau}, \quad \text{where } 0 \leqslant \tau \leqslant \frac{1}{2}, \\
            \| \widetilde{\nabla} f(x) - \nabla f(x) \|_2 \leqslant \alpha \|\nabla f(x) \|_2.
        \end{gathered}
    \end{equation*}
    Then we can solve problem~\eqref{optim} with precision $\varepsilon$ ($f(x^N) - f^* \leqslant \varepsilon)$, using Algorithm~\ref{alg re-agm} via regularization technique with complexity:
    \begin{equation*}
        N = 72000 \left(\frac{L R^2}{\varepsilon} \right)^{1 - \tau} \ln \left( \frac{480 L R^2}{\varepsilon} \right).
    \end{equation*}
\end{theorem}
\begin{proof}
    We will use regularization technique described at Section~\ref{section regularization}. Let us choose
    \begin{equation*}
        \mu = \frac{\varepsilon}{120 R^2}
    \end{equation*}
    and consider $f_{\mu}$~\eqref{reg def}. Then we define:
    \begin{equation*}
        K = \alpha^{-1}.
    \end{equation*}
    Since we use regularization technique we can apply Lemma~\ref{reg noise}:
    \begin{equation*}
        \| \wtgg f_{\mu}(x) - \nabla f_{\mu}(x) \|_2 \leqslant 2 \alpha \|\nabla f_{\mu}(x) \|_2 + \alpha \mu R.
    \end{equation*}
    Of course we will use Theorem~\ref{main relative interpretation th} for Algorithm~\ref{alg re-agm}:
    \begin{equation*}
        \widehat{\alpha} = 2 \alpha + \frac{1}{K} \leqslant 3 \alpha \leqslant \frac{1}{3} \left(\frac{\varepsilon}{2 L R^2} \right)^{\tau}.
    \end{equation*}
    Then we can run Algorithm~\ref{alg re-agm} with $\widehat{\alpha}$ parameter and obtain function residual upper bound: 
    \begin{eqnarray*}
        f_{\mu}(x^{N} | x^0) - f_{\mu}^* 
        & \leqslant &
        \frac{1}{\mu (1 - 2\alpha)^2} \left( \left(6(1 + 2\alpha)K + 1 \right)^2 + 1 \right) \delta^2
        \\
        & \overset{\alpha \leqslant 1 / 9}{\leqslant} &
        \frac{81}{49} \left( \left(\frac{22}{3}K + 1 \right)^2 + 1 \right) \mu \alpha^2 R^2
        \\
        & \overset{K \geqslant 1}{\leqslant} & \frac{81}{49} \frac{25^2 + 9}{9} (K \alpha)^2 \mu R^2
        \\
        & \leqslant & 117 \mu R^2 \leqslant \frac{117}{120} \varepsilon.
    \end{eqnarray*}
    and the method will be stopped after no more than:
    \begin{eqnarray*}
        N 
        & = & 300 \left(\frac{L}{\mu} \right)^{1 - \tau} \ln \left( \frac{4 (1 + \alpha)^2 L^2 R^2}{\left((1 + \alpha)K + 1 \right)^2 - 2} \delta^{-2} \right)
        \\
        & \leqslant & 300 \left(\frac{120 L R^2}{\varepsilon} \right)^{1 - \tau} \ln \left( \frac{16 L^2 R^2}{K^2} \frac{1}{\alpha^2 \mu^2 R^2} \right)
        \\
        & = & 36000 \left(\frac{L R^2}{\varepsilon} \right)^{1 - \tau} \ln \left( \frac{16 \cdot 120^2 L^2 R^4}{\varepsilon^2} \right)
        \\
        & = & 72000 \left(\frac{L R^2}{\varepsilon} \right)^{1 - \tau} \ln \left( \frac{480 L R^2}{\varepsilon} \right).
    \end{eqnarray*}
    Similar to proof of Lemma~\ref{reg base} one can achieve estimation for original function:
    \begin{eqnarray*}
        f(x^N) - f^*
        & \leqslant & f(x^N) + \frac{\mu}{2} \|x^* - x^0 \|_2^2 - f^* \\
        & = & f_{\mu}(x^N | x^0) - f^* = f_{\mu}(x^N | x^0) - f_{\mu}(x^* | x^0) + \frac{\mu}{2} R^2 \\
        & \leqslant & f_{\mu}(x^N | x^0) - f_{\mu}^* + \frac{\varepsilon}{240} \leqslant \frac{117}{120} \varepsilon + \frac{1}{240} \varepsilon \leqslant \varepsilon.
    \end{eqnarray*}
\end{proof}

\section{Lower bounds missing proofs} \label{appendix lower bound}

\begin{lemma} \label{reduction to oracle}
    Let $f$ is $\mu$ strongly convex and $L$ smooth, $\wtgg f$ satisfies absolute noise condition~\ref{absolute noise condition}, then $(\wtgg f(x), f(x) - \delta^2 / \mu)$ is $(2\delta^2 / \mu, 2L, \mu / 2)$ oracle (defined at~\eqref{delta L mu oracle}) for function $f$, $\forall x \in \mathbb{R}^n$.
\end{lemma}
\begin{proof}
    \begin{eqnarray*}
        f(y)
        & \geqslant & f(x) + \langle \nabla f(x), y - x \rangle + \frac{\mu}{2} \|x - y \|_2^2
        \\
        & \geqslant & f(x) + \langle \wtgg f(x), y - x \rangle + \langle \nabla f(x) - \wtgg f(x), y - x \rangle + \frac{\mu}{2} \|x - y \|_2^2
        \\
        & \overset{\ref{fenchel}, \lambda = \mu / 2}{\geqslant} & f(x) + \langle \wtgg f(x), y - x \rangle - \frac{\mu}{4} \|x - y \|_2^2 - \frac{\delta^2}{\mu} + \frac{\mu}{2} \|x - y \|_2^2
        \\
        & \geqslant & f(x) + \langle \wtgg f(x), y - x \rangle + \frac{\mu}{4} \|x - y \|_2^2 - \frac{\delta^2}{\mu}
    \end{eqnarray*}
    \begin{eqnarray*}
        f(y)
        & \leqslant & f(x) + \langle \nabla f(x), y - x \rangle + \frac{L}{2} \|x - y \|_2^2
        \\
        & \leqslant & f(x) + \langle \wtgg f(x), y - x \rangle + \langle \nabla f(x) - \wtgg f(x), y - x \rangle + \frac{L}{2} \|x - y \|_2^2
        \\
        & \overset{\ref{fenchel}, \lambda = L}{\leqslant} & f(x) + \langle \wtgg f(x), y - x \rangle - \frac{L}{2} \|x - y \|_2^2 + \frac{\delta^2}{2L} + \frac{L}{2} \|x - y \|_2^2
        \\
        & \leqslant & f(x) + \langle \wtgg f(x), y - x \rangle + L \|x - y \|_2^2 + \frac{\delta^2}{2L}
    \end{eqnarray*}
    That is:
    \begin{eqnarray*}
        \frac{\mu}{4} \|x - y \|_2^2
        & \leqslant & f(x) - \left(f(y) - \frac{\delta^2}{\mu} + \langle \wtgg f(x), y - x \rangle \right)
        \\
        & \leqslant & L \|x - y \|_2^2 + \frac{\delta^2}{2L} + \frac{\delta^2}{\mu} \overset{\mu \leqslant L}{\leqslant} L \|x - y \|_2^2 + \frac{2\delta^2}{\mu}
    \end{eqnarray*}
\end{proof}

\begin{theorem}[Insignificant modification of Theorem 8, p. 31~\cite{devolder2013first}] \label{lower bound oracle absolute}
    Let $\mathcal{A}$ is first order method~\eqref{linear first order method}.
    Assume that the bounds on the performance of this method, as applied to a problem equipped with an $(\delta, L, \mu)$ oracle, are given by inequality:
    \begin{eqnarray*}
        f(x^N) - f^*
        & \leqslant & C_1 LR^2 \exp \left(-\mathbf{A_0} \left(\frac{\mu}{L} \right)^{p} N \right)
        + C_2 \left(\frac{L}{\mu}\right)^{q} \delta.
    \end{eqnarray*}
    where $C_1, C_2, \mathbf{A_0}$ - constants.
    Then the inequalities must hold:
    \begin{equation*}
        q \geqslant 1 - p.
    \end{equation*}
\end{theorem}

\begin{remark} \label{remark for devolder theorem modification}
    We provided small modification of Theorem 8 from~\cite{devolder2013first}. The essence of such a change:
    \begin{equation*}
        \min \left \lbrace C_1 \frac{LR^2}{N^{p_1}}, C_2 LR^2 \exp \left(-\mathbf{A_0} \left(\frac{\mu}{L} \right)^{p_2} N \right) \right \rbrace \longrightarrow \exp \left(\mathbf{A_0} \left(\frac{\mu}{L} \right)^{p_2} N \right), 
    \end{equation*}
    \begin{equation*}
         \min \left \lbrace C_3 N^{q_1} \delta, C_4 \left(\frac{L}{\mu}\right)^{q_2} \delta \right \rbrace \longrightarrow  C_2 \left(\frac{L}{\mu}\right)^{q} \delta.
    \end{equation*}
    Such modification does not change the proof and result of the theorem, but adds flexibility to its use. The proof can be obtained by leaving the second part of the proof of the original theorem.
\end{remark}

\begin{theorem}[Theorem~\ref{lower bound relative noise mu > 0 text}] \label{lower bound relative noise mu > 0}
    Let $f$ is $\mu$-strongly convex~\eqref{strong conv} and $L$-smooth~\eqref{smooth cond}, algorithm $\mathcal{A}$ is first order gradient method such that using $\wtgg f$ satisfying relative noise~\eqref{relative noise condition} with magnitude $\alpha_0$, produces $x^N$, such that:
    \begin{equation*}
        f(x^N) - f^* \leqslant C_0 LR^2 \exp \left(- A_0 \left(\frac{\mu}{L} \right)^{p(\alpha_0)} N \right),
    \end{equation*}
    where:
    \begin{equation*}
        p: [0; 1) \to [1 / 2; + \infty).
    \end{equation*}
    Then:
    \begin{equation*}
        p(\alpha_0) \geqslant 1 - 2 \cdot \min \left \lbrace 1/4, \log_{\mu / L}(\alpha) \right \rbrace.
    \end{equation*}
\end{theorem}
\begin{proof}
    Consider absolute noise $\zeta_a(x)$ with $\delta > 0$ and $\wtgg f(x) = \nabla f(x) + \zeta_a(x)$. Let us try to solve the minimization problem~\ref{optim} using the method $\mathcal{A}$ and noised gradient estimation $\wtgg f$. We can use relative interpretation - result of Theorem~\ref{main relative interpretation th}, that is we use the following algorithm $\widetilde{\mathcal{A}}$:
    \begin{align*}
        & \text{If: } \|\nabla f(x^N) \|_2 > K \delta \\
        & \quad \quad \text{continue running } \mathcal{A} \\
        & \text{Else: } \\
        & \quad \quad x^{N + 1} = x^N
    \end{align*}
    where:
    \[
        K = \left(\frac{L}{\mu} \right)^{q/2}, \; 0 \leqslant q \leqslant 1, \quad \widehat{\alpha} = \frac{1}{K}.
    \]
    Theorem~\ref{main relative interpretation th} guarantees for algorithm $\mathcal{A}$:
    \[
        f(x^N) - f^* \leqslant \left((K + 1)^2 + 1\right) \frac{\delta^2}{\mu} \overset{K \geqslant 1}{\leqslant} 5K^2 \frac{\delta^2}{\mu} = 25 \left(\frac{L}{\mu} \right)^{q} \frac{\delta^2}{\mu}.
    \]
    One can see, that algorithm $\widetilde{\mathcal{A}}$ has the following convergence rate:
    \begin{eqnarray*}
        f(x^N) - f^*
        & \leqslant & C_0 LR^2 \exp \left(-A_0 \left(\frac{\mu}{L} \right)^{p(\widehat{\alpha})} N \right) + 25 \left(\frac{L}{\mu}\right)^{q} \frac{\delta^2}{\mu} \\
        & \overset{L \to 2L, \mu \to \mu / 2}{\leqslant} & C_0 (2L)R^2 \exp \left(-A_0 \left(\frac{\mu}{4L} \right)^{p(\widehat{\alpha})} N \right) + 25 \left(\frac{4L}{\mu}\right)^{q} \frac{2\delta^2}{\mu}
    \end{eqnarray*}
    where $C_3$ is positive absolute constant large enough for each pair $\mu, L$. From Lemma~\ref{reduction to oracle} we obtain, that $(\wtgg f(x), f(x) - \frac{\delta^2}{\mu})$ is $(2\delta^2 / \mu, 2L, \mu / 2)$ oracle. From convergence rate of algorithm $\widetilde{\mathcal{A}}$ provided above and Theorem~\ref{lower bound oracle absolute}: $p(\widehat{\alpha}) \geqslant 1 - q$. From $\widehat{\alpha}$ definition: $q = 2 \log_{\mu / L}(\widehat{\alpha})$, thus:
    \begin{equation*}
        p(\widehat{\alpha}) \geqslant 1 - q = 1 - 2 \log_{\mu / L}(\widehat{\alpha}), \quad \alpha \geqslant \mu / L.
    \end{equation*}
    Using lower bound convergence for smooth, strongly convex functions:
    \begin{equation*}
        p(\widehat{\alpha}) \geqslant 1 - 2 \cdot \min \left \lbrace 1/4, \log_{\mu / L}(\widehat{\alpha}) \right \rbrace, \quad 0 \leqslant \alpha \leqslant 1.
    \end{equation*}
\end{proof}

\begin{theorem}[Theorem~\ref{lower bound relative noise mu = 0 text}] \label{lower bound relative noise mu = 0}
    Let $f$ is convex~\eqref{convexity} and $L$-smooth~\eqref{smooth cond}, algorithm $\mathcal{A}$ is first order gradient method.

    1. If solving problem $f(x^N) - f^* \leqslant \varepsilon$ using the algorithm $\mathcal{A}$ in the presence of relative noise~\ref{relative noise condition} with magnitude $\alpha$ requires
    \begin{equation*}
        N_{\varepsilon} \leqslant C_1 \left(\frac{L R^2}{\varepsilon} \right)^{1 / p(\alpha, \varepsilon)}, \quad p(\alpha, \varepsilon) \leqslant 2.
    \end{equation*}
    iterations, then:
    \begin{equation*}
        p(\alpha, \varepsilon) \leqslant \frac{1}{1 - 2 \min \left \lbrace 1/4, \log_{\varepsilon / LR^2}(\alpha) \right \rbrace}.
    \end{equation*}

    2. If algorithm $\mathcal{A}$ in the presence of relative noise~\ref{relative noise condition} with magnitude $\alpha$ has convergence:
    \begin{equation*}
        f(x^N) - f^* \leqslant C_2 \frac{LR^2}{N^{p(\alpha)}},
    \end{equation*}
    then:
    \begin{equation*}
        p(\alpha) \leqslant 1.
    \end{equation*}
\end{theorem}
\begin{proof}
    1. To begin with, we note that we assume, that $\varepsilon \leqslant LR^2$, because from smoothness~\eqref{smooth cond} implies $f(x^0) - f^* \leqslant LR^2 / 2$. We will use restarts technique as a reduction strongly convex optimization problem to a convex one. Let us consider a $\mu$ strongly convex and $L$ smooth function $f$. Algorithm $\mathcal{A}$ can be applied to function $f$ and we will run it for $N_0$, so that:
    \begin{equation*}
        f(x^{N_{\varepsilon_0}}) - f^* \overset{\eqref{strong conv}}{\leqslant} \frac{\mu}{4} R^2 \leqslant \frac{1}{2} \left( f(x^0) - f^* \right),
    \end{equation*}
    where $\varepsilon_0 = \frac{\mu}{4} R^2 \leqslant \frac{LR^2}{4}$ and from theorem condition:
    \begin{equation*}
        N_{\varepsilon_0} \leqslant C_1 \left(\frac{L R^2}{\varepsilon_0} \right)^{1 / p(\alpha, \varepsilon_0)} = C_1 \left(\frac{4L}{\mu} \right)^{1 / p(\alpha, \varepsilon_0)}.
    \end{equation*}
    Thus, to solve the problem $f(x^N) - f^* \leqslant \varepsilon$ it will require:
    \begin{eqnarray*}
        N & = & N_{\varepsilon_0} \cdot \log_2\left( \frac{f(x^0) - f^*}{\varepsilon} \right)
        \leqslant C_1 \left(\frac{4L}{\mu} \right)^{1 / p(\alpha, \varepsilon_0)} \log_2\left( \frac{f(x^0) - f^*}{\varepsilon} \right) 
        \\
        & \overset{\eqref{smooth cond}}{\leqslant} & 16 C_1 \left(\frac{L}{\mu} \right)^{1 / p(\alpha, \varepsilon_0)} \log_2\left( \frac{L R^2}{2 \varepsilon} \right) \text{ iterations}.
    \end{eqnarray*}
    Using Theorem~\ref{lower bound relative noise mu > 0} as the lower bound we conclude:
    \begin{equation*}
        N \geqslant C_3 \left( \frac{L}{\mu} \right)^{\gamma(\alpha, \mu, L)} \ln \left( C_4 \frac{LR^2}{\varepsilon} \right), \quad \gamma(\alpha, \mu, L) \geqslant 1 - 2 \cdot \min \left \lbrace 1/4, \log_{\mu / L}(\alpha) \right \rbrace.
    \end{equation*}
    for some absolute constants $C_3, C_4 > 0$, that is:
    \begin{equation*}
        16 C_1 \left(\frac{L}{\mu} \right)^{1 / p(\alpha, \varepsilon_0)} \log_2\left( \frac{L R^2}{2\varepsilon} \right) \leqslant C_3 \left( \frac{L}{\mu} \right)^{\gamma(\alpha, \mu, L)} \ln \left( C_4 \frac{LR^2}{\varepsilon} \right). 
    \end{equation*}
    From the inequality above follows:
    \begin{equation*}
        p(\alpha, \varepsilon_0) \leqslant \frac{1}{\gamma(\alpha, \mu, L)} \leqslant \frac{1}{1 - 2 \cdot \min \left \lbrace 1/4, \log_{\mu / L}(\alpha) \right \rbrace}.
    \end{equation*}
    Since $\mu = \frac{4 \varepsilon_0}{R^2}$:
    \begin{equation*}
        p(\alpha, \varepsilon_0) \leqslant \frac{1}{1 - 2 \cdot \min \left \lbrace 1/4, \log_{4 \varepsilon_0 / L R^2}(\alpha) \right \rbrace}.
    \end{equation*}
    Definition $\varepsilon_0$ implies $\varepsilon_0 \leqslant LR^2 / 4$, so (making substitution $\widehat{\varepsilon} = 4 \varepsilon_0 \leqslant LR^2$):
    \begin{equation*}
        p(\alpha, \widehat{\varepsilon}) \leqslant \frac{1}{1 - 2 \cdot \min \left \lbrace 1/4, \log_{ \widehat{\varepsilon} / L R^2}(\alpha) \right \rbrace}.
    \end{equation*}

    2. From previous point, we obtain lower bound on convergence rate, depending on both $\alpha$ and $\varepsilon$, which takes place $\forall \varepsilon \leqslant LR^2$, then we can take minimum on left side of inequality:
    \begin{equation*}
        p(\alpha) \leqslant \inf_{0 < \varepsilon \leqslant LR^2} \frac{1}{1 - 2 \cdot \min \left \lbrace 1/4, \log_{ \widehat{\varepsilon} / L R^2}(\alpha) \right \rbrace} = 1.
    \end{equation*}
    That is, for fixed $\alpha$ we can guarantee convergence lower bound $\Omega \left( LR^2 / N \right)$, described at the theorem condition.
\end{proof}

\end{appendix}

\end{document}